\documentclass[mathpazo]{cicp}

\usepackage{amsmath,amsfonts,amsthm,amssymb}
\usepackage{algorithm2e}
\usepackage{bm,dsfont}
\usepackage{booktabs}
\usepackage{graphicx}
\usepackage{array}
\usepackage{xcolor}
\usepackage{sidecap}
\usepackage{listings}
\usepackage{times}
\usepackage{url}
\usepackage{hyperref}

\usepackage[utf8]{inputenc}
\usepackage[margin=1in]{geometry}
\usepackage[titletoc,title]{appendix}
\usepackage{amsmath,amsfonts,amssymb,mathtools}
\usepackage{graphicx,float}
\usepackage{times}
\usepackage{mlmath}
\usepackage{color}
\usepackage{amsthm} 
\usepackage{hyperref}
\usepackage{subfigure}
\usepackage{comment}
\usepackage{graphicx}     % 插入图片
\usepackage{subcaption}   % 支持子图布局 (推荐替代 subfigure)
\usepackage[margin=1in]{geometry}
\PassOptionsToPackage{numbers, compress}{natbib}
\usepackage{natbib}
\newtheorem{prop}{Proposition}
\newtheorem{assump}{Assumption}

\ifpdf
  \DeclareGraphicsExtensions{.eps,.pdf,.png,.jpg}
\else
  \DeclareGraphicsExtensions{.eps}
\fi
\ifpdf
\hypersetup{
	pdftitle={SSBE},
	pdfauthor={Zhou, Q. and Chen, C. and Luo, T. and Xiang, Y.}
}
\fi

\begin{document}
	
	\title{SSBE-PINN: A Sobolev Boundary Scheme Boosting Stability and Accuracy in Elliptic/Parabolic PDE Learning}

%Understanding and Correcting Instability in PINNs via Sobolev-Regularized Boundary Loss

\author[Zhou Q.X. et.~al.]{
    Qixuan Zhou\affil{1}, 
    Chuqi Chen\affil{2,5,}\footnotemark[1], 
    Tao Luo\affil{1,3}, 
    and Yang Xiang\affil{2,4,}\footnotemark[1]
}
\address{
    \affilnum{1}\ School of Mathematical Sciences, Shanghai Jiao Tong University, Shanghai, China \\
    \affilnum{2}\ Department of Mathematics, Hong Kong University of Science and Technology,
    Clear Water Bay, Hong Kong SAR, China\\
    \affilnum{3}Institute of Natural Sciences, MOE-LSC,
    Shanghai Jiao Tong University, CMA-Shanghai\\
    \affilnum{4}\ Algorithms of Machine Learning and Autonomous Driving Research Lab,
    HKUST Shenzhen-Hong Kong Collaborative Innovation Research Institute,
    Futian, Shenzhen, China \\
    \affilnum{5}\ Department of Mathematics, University of Michigan,
    Ann Arbor, MI, USA\\
}
\emails{
    {\tt zhouqixuan@sjtu.edu.cn }, 
    {\tt cchenck@umich.edu },
    {\tt maxiang@ust.hk}
    {\tt luotao41@sjtu.edu.cn},
}

\footnotetext[1]{Corresponding author {\tt cchenck@umich.edu}, {\tt maxiang@ust.hk}}

\begin{abstract}
		Physics-Informed Neural Networks (PINNs) have emerged as a powerful framework for solving partial differential equations (PDEs), yet they often fail to achieve accurate convergence in the 
$H^1$ norm, especially in the presence of boundary approximation errors. In this work, we propose a novel method called \textit{Sobolev-Stable Boundary Enforcement} (SSBE), which redefines the boundary loss using Sobolev norms to incorporate boundary regularity directly into the training process. We provide rigorous theoretical analysis demonstrating that SSBE ensures bounded $H^1$ error via a stability guarantee and derive generalization bounds that characterize its robustness under finite-sample regimes. Extensive numerical experiments on linear and nonlinear PDEs—including Poisson, heat, and elliptic problems—show that SSBE consistently outperforms standard PINNs in terms of both relative $L^2$ and $H^1$ errors, even in high-dimensional settings. The proposed approach offers a principled and practical solution for improving gradient fidelity and overall solution accuracy in neural network-based PDE solvers.
	\end{abstract}

	
	
	\maketitle	
\section{Introduction}
Partial Differential Equations (PDEs) are fundamental in the mathematical modeling of diverse physical phenomena, including heat conduction, fluid dynamics, electromagnetism, and quantum mechanics. Despite substantial advancements in numerical techniques over recent decades, accurately and efficiently solving PDEs remains challenging, particularly for problems involving complex geometries or high-dimensional spaces. Traditional numerical approaches, such as the Finite Element Method (FEM) and Finite Difference Method (FDM), have been extensively utilized due to their robustness, clear theoretical foundations, and well-established convergence properties~\cite{Zienkiewic2013TheFinite}. These methods discretize the computational domain using structured grids or meshes, enabling reliable numerical approximations. However, they can encounter difficulties when dealing with high-dimensional problems due to the exponential growth in computational complexity (the curse of dimensionality) and may require sophisticated mesh refinement strategies to accurately resolve solutions near singularities or steep gradients~\cite{Strang2008Ananalysis}.

In recent years, there has been increasing interest in leveraging Machine Learning (ML) techniques for solving PDEs, motivated by the capability of deep learning models to approximate complex, high-dimensional solutions without explicit domain discretization. Several neural network-based schemes have emerged, including Physics-Informed Neural Networks (PINNs)~\cite{Raissi2019Physics-informed}, Deep Ritz Method (DRM)~\cite{DeepRitz}, Weak Adversarial Neural Networks (WANs)~\cite{wan}, along with their various extensions. Among these, PINNs have gained prominence due to their conceptual simplicity, adaptability, and mesh-free formulation. These approaches utilize neural networks to approximate PDE solutions directly, circumventing the need for predefined computational grids. ML-based PDE solvers have shown considerable potential in handling both forward and inverse problems~\cite{Raissi2019Physics-informed, Sirignano2018DGM}, particularly in high-dimensional scenarios. However, existing methods typically ensure convergence primarily in the $L^2$ norm, without guaranteeing accurate convergence of solution gradients. This limitation may result in inaccurate solutions when gradient information is critical, such as velocity fields in fluid dynamics or temperature gradients in heat transfer problems~\cite{GU2021110444,krishnapriyan2021characterizing}.

In the numerical analysis of partial differential equations (PDEs), while $L^2$ convergence ensures that the solution is accurate on average, it provides an incomplete measure of fidelity as it solely considers the solution's pointwise values. A more stringent and physically meaningful criterion is convergence in the Sobolev space $H^1$, which additionally incorporates the error in the solution's gradients~\cite{brenner2008mathematical,suli2003introduction}. This holistic assessment is crucial for physical systems where first-order derivatives are not merely mathematical abstractions but correspond to vital physical quantities such as velocity, flux, stress, and strain. For instance, in solid mechanics, the $H^1$ seminorm is directly proportional to the elastic strain energy, meaning that minimizing the error in this norm is equivalent to finding a solution that is physically more realistic in terms of energy. Furthermore, ensuring $H^1$ convergence yields significant numerical benefits. It leads to more robust error estimates and often superior convergence rates, particularly for problems with complex boundary conditions or in regions where solution gradients vary sharply~\cite{ciarlet2002finite}. By penalizing large, oscillatory gradient errors, the $H^1$ norm inherently promotes smoother numerical approximations. This property helps to mitigate numerical instabilities and spurious artifacts, such as Runge's phenomenon, which can arise from high-order polynomial approximations, thereby ensuring that the numerical method is not only consistent but also stable~\cite{arnold2002differential}.
These considerations are especially pertinent to modern deep learning-based PDE solvers like Physics-Informed Neural Networks (PINNs). In such methods, the PDE residual is minimized via optimization, and using a loss function based on the $L^2$ norm can lead to ill-conditioned training landscapes, especially for multiscale problems~\cite{LI2021100429, Wang2021Understanding}. Adopting an $H^1$-based loss, however, regularizes the optimization process by enforcing constraints on the gradients, leading to more stable training and more physically accurate predictions. Consequently, achieving $H^1$ accuracy is not just a matter of mathematical rigor; it is fundamental to obtaining numerically stable, physically reliable, and robust solutions for a wide range of scientific and engineering problems.

In this paper, our analysis and investigation primarily focus on improving the $H^1$ convergence of neural network-based PDE solvers, particularly PINNs. We first introduce the problem setting of our paper: 
\paragraph{Problem Settings} 

Let $\Omega$ be an open and bounded domain in $\sR^d$. We consider the partial differential equation (PDE) problem of the following general form:
\begin{equation}\label{eq::setup}
\left\{
\begin{aligned}
    \fL u &= f, \quad &&\text{$x \in \Omega$}, \\
    \fB u &= g, \quad &&\text{$x \in \partial\Omega$}.
\end{aligned}
\right.
\end{equation}
where $\fL$ is the differential operator and $\fB$ denotes the boundary operator, which may represent Dirichlet, Neumann, periodic conditions, etc. Throughout this paper, we assume that the PDE defined by Eq.~\eqref{eq::setup} admits a unique analytical solution $u(x)$, with $x=(x_1,\ldots,x_d)$, satisfying Eq.~\eqref{eq::setup} almost everywhere within $\Omega$ and on its boundary $\partial\Omega$. For problems involving time-dependence, the time variable $t$ can be incorporated as an additional dimension within $x$, allowing $\Omega$ to represent a spatio-temporal domain.

Our primary objective is to \textbf{approximate the solution of Eq.~\eqref{eq::setup} up to first-order accuracy} using data randomly sampled from both the domain and the boundary. Specifically, we denote these sample sets by $S_{\Omega}:=\{x_{\Omega}^i\}_{i=1}^{n_{\Omega}}$ and $S_{\partial\Omega}:=\{x_{\partial\Omega}^{j}\}_{j=1}^{n_{\partial\Omega}}$. In this paper, we focus on Physics-Informed Neural Network (PINN), which approximates the PDE solution using a neural network. The PDE problem can typically be reformulated as the following nonlinear least-squares optimization problem, aiming to find parameters $\theta$ for the neural network $u(x;\theta)$ (often a multi-layer perceptron, MLP):
\begin{equation}\label{eq::PINN}
    \fR_{S}^{\operatorname{PINN}}:=\underbrace{\frac{1}{n_{\Omega}}\sum_{i=1}^{n_{\Omega}}\Abs{\fL u(x^i_{\Omega})-f(x^i_{\Omega})}^2}_{L_{\text{res}}}+\underbrace{\frac{\lambda}{n_{\partial\Omega}}\sum_{j=1}^{n_{\partial\Omega}}\Abs{\fB u(x^j_{\partial\Omega})-g(x^j_{\partial\Omega})}^2}_{ L_{\text{bc}}}.
\end{equation}

Here $L_{\text{res}}$ is the PDE residual loss, $ L_{\text{bc}}$ is the boundary loss and $\lambda$ is a constant used to balance these two terms. In cases of exact boundary conditions—where minimizers of Eq.~\eqref{eq::PINN} coincide precisely with the true boundary values—previous work has demonstrated that $H^1$ approximation properties can be successfully obtained. However, as we illustrate in Section~\ref{sec::CounterExample}, even small perturbations or inaccuracies in the boundary loss can lead to significant deviations from the ground truth solution.  Specifically, for the population loss of PINN $\fR_{\fS}^{\operatorname{PINN}}$ defined in Eq.~\eqref{eq::PINN}, for given operators $\fL$ and $\fB$, there is $\{u_i\}_{i=1}^{\infty}$ such that

\begin{equation}\label{eq::FailureH1}
    \begin{aligned}
        &\lim_{i\to \infty}\frac{1}{\Abs{\Omega}}\norm{\fL u_i-f}_{L^2(\Omega)}^2+\frac{1}{\Abs{\partial\Omega}}\norm{\fB u_i-g}_{L^2(\partial\Omega)}^2\to 0,\\
        &\text{yet,} \quad 
        \frac{\norm{u_i-u}_{H^1(\Omega)}}{\norm{u}_{H^1(\Omega)}}=O(1)\quad \text{for all $i\in \{1,\ldots,\infty\}$}.
    \end{aligned}
\end{equation}

This result highlights a fundamental limitation of PINNs: even minor errors in boundary approximations can yield $O(1)$ relative errors between the neural network approximation and the exact solution in the $H^1$ norm. Consequently, our work focuses on developing a novel approach that explicitly ensures neural networks achieve robust and accurate approximations in the $H^1$ sense. We propose the following method to achieve this goal.

\textbf{Sobolev-Stable Boundary Enforcement (SSBE)}: By introducing a Sobolev space restricted to boundary (see Section~\ref{sec::BoundarySobolev}), we propose a novel algorithm named \textbf{SSBE} that redefines the population loss with objective function $v$ as
\begin{equation}\label{eq::H1convergenceAlgorithm}
    \fR_{D}(v)=\frac{1}{\Abs{\Omega}}\norm{\fL v-f}_{L^2(\Omega)}^2+\frac{\lambda}{\Abs{\partial\Omega}}\norm{\fB v-g}_{H^1(\partial\Omega)}^2,
\end{equation}
where $\lambda>0$ is an adjustable weight parameter. This formulation retains the general structure of physics-informed neural network (PINN) losses while augmenting the boundary term with an $H^1(\partial\Omega)$ penalty. Control of the $H^1$ norm of the boundary discrepancy on $\partial\Omega$ ensures that the objective function $v$ converges to the ground-truth solution in the $H^1$ sense, thereby guaranteeing convergence of both the solution and its first-order derivatives.

In the proposed \textbf{SSBE} framework, we circumvent this ambiguity by performing a local dimension reduction to the boundary. For instance, in two dimensions, we locally parameterize the boundary so that one coordinate is expressed as a function of the other (e.g., $y=y(x)$ along $\partial\Omega$), and then compute derivatives tangentially with respect to the intrinsic boundary coordinate. The $H^1(\partial\Omega)$ norm of the boundary discrepancy is then evaluated in this reduced coordinate system, ensuring a well-defined and consistent treatment of boundary derivatives without requiring an ambient extension of $g$. This flattening procedure locally transforms the boundary into a lower-dimensional coordinate system, thereby simplifying its geometry and enabling a consistent computation of tangential derivatives without requiring any ambient extension of the boundary data. Within this intrinsic coordinate framework, we evaluate the $H^1(\partial\Omega)$ norm of the discrepancy between the objective function and the prescribed boundary data. This construction ensures that the learned solution is close to the ground-truth solution in the $L^2$ sense while also achieving $H^1$-norm convergence, thereby controlling both the solution itself and its first-order derivatives. Such accuracy in derivative information is particularly advantageous in applications such as fluid dynamics, heat conduction, and other PDE models in which gradients carry essential physical meaning.

The necessity of this intrinsic-boundary approach becomes apparent when noting that, in many PDE problems, the boundary function $g$ is prescribed only as trace data on $\partial\Omega$ and is not defined in the ambient domain $\Omega$. Consequently, the strategy of directly computing the ambient-space derivatives of $g$ and comparing them with those of the learned solution $v$ through the penalty 
$
\| D_{\alpha} g - D_{\alpha} v \|_{L^{2}(\partial\Omega)}
$
is generally not well-posed. Such an approach implicitly requires a smooth extension of $g$ from $\partial\Omega$ to a neighborhood in $\mathbb{R}^d$, yet these extensions are typically non-unique: distinct extensions that coincide pointwise on $\partial\Omega$ may yield different derivatives on the boundary. The proposed \textbf{SSBE} effectively avoids this issue by ensuring that the boundary function $g$ admits a unique and consistent representation through a segmentation and flattening procedure. This approach simplifies the computation of boundary gradients, eliminating ambiguities and ensuring accurate evaluations. Moreover, the SSBE method does not rely on additional external information beyond the PDE's original definition; instead, the boundary flattening and associated gradient computations are fully determined by the intrinsic geometry and structure of the PDE problem itself. 

From a theoretical perspective, the SSBE framework requires the boundary data $g$ to belong to $H^{1/2}(\partial\Omega)$, ensuring, by standard PDE regularity results, that the corresponding forcing term $f$ lies in $L^{2}(\Omega)$. In practice, for numerical stability and ease of implementation, we often assume $g \in H^{1}(\partial\Omega)$, which corresponds to $f \in H^{1/2}(\Omega)$. This regularity assumption is standard in PDE analysis~\cite{McLean2000StronglyES,kufner1977function,gilbarg1977elliptic} and is not restrictive for the range of applications considered in this work.

Several existing studies have also investigated alternative formulations of boundary conditions in PINNs beyond the standard $L^2$ norm~\cite{shin2020convergence,bonito2024convergence,bachmayr2024variationally}. In~\cite{jiao2024stabilized}, the analysis is conducted in the context of the wave equation, where achieving $H^1$ convergence requires imposing an $H^1$-norm convergence condition in time $t$ on the boundary. However, no additional constraints on spatial derivatives along the boundary are explicitly considered.  It is worth mentioning that a similar idea called
Sobolev training has been proposed to improve the efficiency for regression ~\cite{czarnecki2017sobolev}. Later in ~\cite{son2021sobolev} and ~\cite{vlassis2021sobolev}, the authors generalized this idea to the training of
PINNs, with applications to heat equation, Burgers’ equation, Fokker-Planck
equation and elasto-plasticity models. One main difference between our proposed method
and these works is that, we still use the $L^2$ norm, rather than $H^1$ norm for the
residual and initial condition in the loss of SSBE. This designing, as we will demonstrate, turns out to be a sufficient condition to guarantee the stability. These findings provide evidence for the effectiveness of the proposed method and its efficiency from a computational standpoint. In this work, we focus on the convergence of the $H^1$ norm and provide a comprehensive analysis for both elliptic and parabolic types of partial differential equations. To address the limitations of existing methods in capturing $H^1$ convergence, we propose the SSBE loss, which is specifically designed to overcome this issue. The proposed SSBE loss is computationally convenient, and empirical results demonstrate that it effectively improves both the $L^2$ and $H^1$ accuracy of the solution. 

\begin{color}{black}
    Beyond stability, a growing literature provides quantitative analyses of neural network PDE solvers under finite sampling, by relating population objectives to their empirical counterparts. 
For variational formulations, such as the deep Ritz method and high-dimensional Schr\"odinger eigenvalue problems, 
\cite{lu2021apriorigeneralization,lu2022apriorigeneralization} derived dimension-independent \emph{a priori} generalization bounds by combining Barron-type approximation theory with Rademacher complexity, thereby controlling the population--empirical gap of the energy functional. 
For residual-based approaches (including PINN-type formulations), there are also theoretical studies on convergence and error decompositions under discrete interior/boundary sampling and numerical quadrature, e.g., \cite{shin2020convergence, mishra2022estimates, luo2024twolayer}. In the present work, we adapt this finite-sample viewpoint to the proposed SSBE loss and close the continuous-to-discrete error pathway tailored to Sobolev-stable boundary enforcement. 
Starting from a population-level quantitative $H^1$-stability inequality for SSBE, we bound the discrepancy between the population SSBE loss and its empirical counterpart induced by finite interior and boundary samples, and translate it into explicit $H^1$-error estimates. 
This yields a coherent picture in which the same SSBE loss simultaneously restores continuous $H^1$-stability and admits favorable sample-based generalization guarantees, consistent with our numerical observations.
\end{color}

In this work, we target $H^1$-convergence and provide rigorous analyses for both elliptic and parabolic PDEs. To overcome the limitations of existing objectives in controlling $H^1$ errors, we propose the \textbf{SSBE} loss—a theory-grounded objective explicitly tailored to $H^1$ accuracy. In addition to the standard $L^2$ terms for the PDE residual and initial condition, the SSBE loss incorporates an $H^1(\partial\Omega)$ boundary term, which simultaneously constrains the boundary values and their tangential derivatives. A key ingredient is a boundary-aligned re-coordinate (reparameterization) motivated by PDE trace theory, which sharpens the handling of boundary terms and reduces boundary residuals. This design ensures that the learned solution converges to the ground truth in the $H^1$ norm while maintaining computational efficiency. Empirically, SSBE yields substantial improvements in both $L^2$ and $H^1$ accuracy. Our approach is complementary to advances in sampling~\cite{luo2025imbalanced,wight2020solving,wu2023comprehensive}, architecture design~\cite{liu2020multi,huang2025frequency,chen2024quantifying}, and training strategies~\cite{zheng2024hompinns,chen2025learn,krishnapriyan2021characterizing}, and can be seamlessly integrated with these techniques to further enhance the solution quality of PINNs.

To summarize, the advantages of incorporating our proposed $H^1$-based \textbf{SSBE} method can be outlined as follows:

\begin{color}{black}
\begin{itemize}
    \item \textbf{$H^1$-stable SSBE-PINN formulation.}
    We propose a Sobolev-Stable Boundary Enforcement (SSBE) framework in Section~\ref{sec::Experiments} and in Section~\ref{sec::Stability}, establish a quantitative $H^1$-stability estimate
    \begin{equation*}
        \norm{u-v}_{H^1(\Omega)}^2\le C\fR_{D}(v),
    \end{equation*}
    thereby rigorously linking the population SSBE loss to the approximation error in the natural energy norm.

    \item \textbf{Practical Sobolev boundary enforcement via local parameterization.}
    We design a practical implementation of the SSBE loss based on local parameterizations of the boundary into coordinate charts, on which boundary traces and tangential derivatives are evaluated and assembled.
    This construction makes the theoretically motivated Sobolev boundary scheme amenable to efficient numerical realization in PINN training.
    \item \textbf{From continuous $H^1$-stability to sample-based generalization bounds.}
    Building on the $H^1$-stability, Section~\ref{sec::GeneralizationBound} derives \emph{apriori} generalization bounds for SSBE (note as $\text{GenErr}(\fR_D(\cdot),\fR_S(\cdot))$) by combining Rademacher complexity estimates with Barron-type approximation theory.
    We quantify the gap between population and empirical SSBE losses under finite interior and boundary sampling and translate it into explicit $H^1$-error bounds, i.e. 
    \begin{equation*}
        \norm{u-u_{\theta}}_{H^1(\Omega)}\le C(\fR_S(u_{\theta})+\text{GenErr}(\fR_D(u_{\theta}),\fR_S(u_{\theta}))).
    \end{equation*}
    
    \item \textbf{Numerical validation of stability and generalization.}
    We present a series of numerical experiments in Section~\ref{sec:example} on representative elliptic and parabolic PDEs. The results systematically indicate that SSBE achieve improved approximation accuracy in both the $L^2$ and $H^1$ norms compared with conventional PINN formulations, thereby providing empirical support for the theoretical stability and generalization properties established in this work.
\end{itemize}
\end{color}

The paper is organized as follows. In Section~\ref{sec:setup}, we provide a detailed description of several preliminary concepts necessary for our discussion. In Section~\ref{sec::CounterExample}, we construct an illustrative example demonstrating that conventional PINNs can result in inaccurate convergence in the $H^1$ norm. Building upon this, we propose our SSBE algorithm in Section~\ref{sec::Experiments}. Section~\ref{sec::Stability} presents a theoretical analysis of the $H^1$ stability and convergence of our method, while Section~\ref{sec::GeneralizationBound} provides an analysis of the generalization error. Finally, Section~\ref{sec:example}  presents numerical experiments that demonstrate the effectiveness of our proposed approach.

\begin{figure}[htpb]
    \centering
    \includegraphics[scale=0.45]{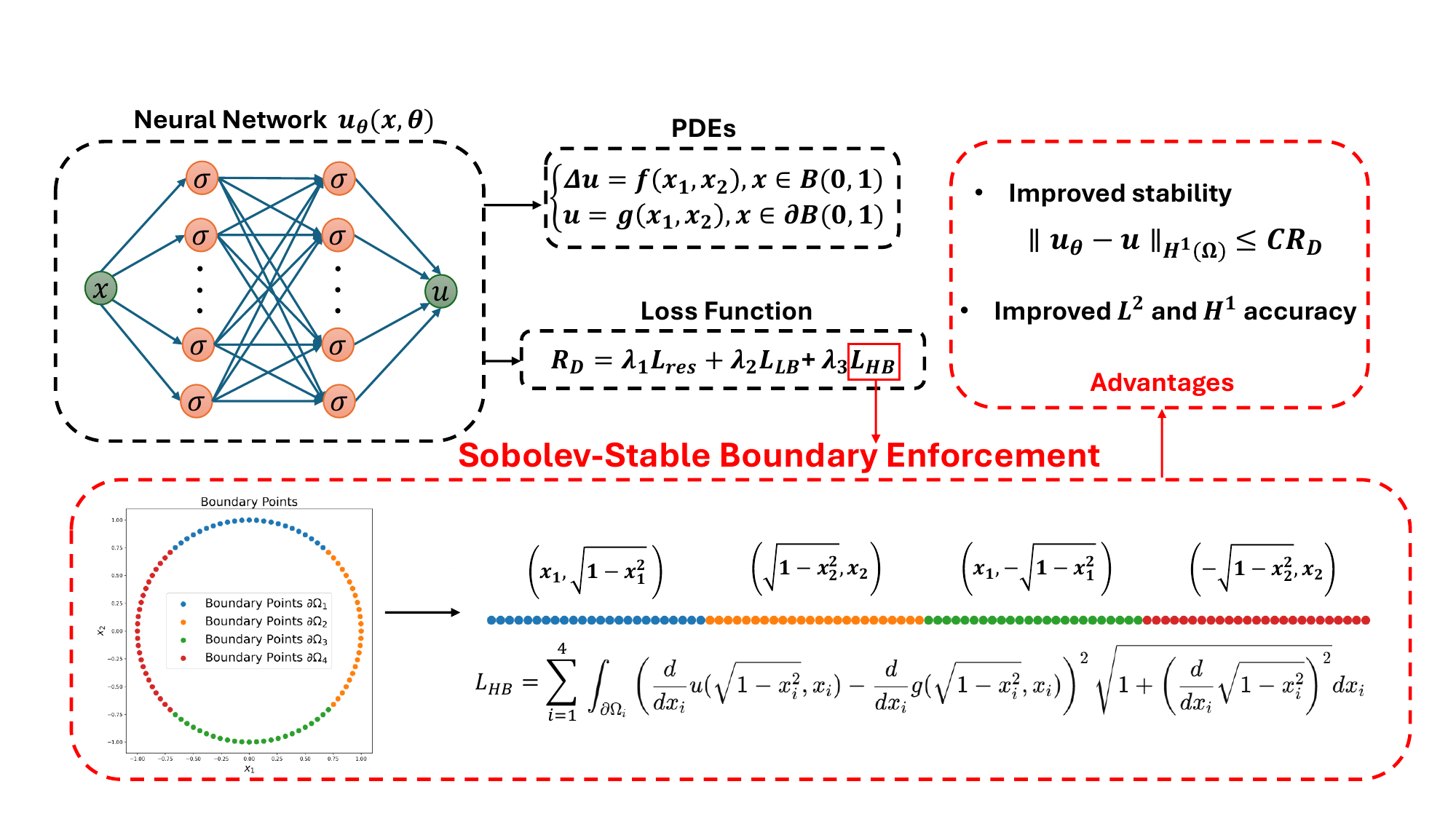}
    \label{Fig.Sketch}
    \caption{Framework of sobolev-stable boundary enforcement.}
\end{figure}
\section{Mathematical Setup and Preliminaries} \label{sec:setup}
In this section, before analyzing our proposed \textbf{SSBE} method, we first present several preliminaries necessary for our discussion. For completeness, we begin by introducing Sobolev spaces restricted to the boundary. Then, we define the PDE problems (both elliptic and parabolic) and state the assumptions employed in our analysis. Lastly, we introduce the notation related to neural networks and function spaces used throughout this paper.

\subsection{Sobolev Space Restricted on Boundary}\label{sec::BoundarySobolev}

To introduce our proposed \textbf{SSBE} method, we begin by recalling several concepts regarding Sobolev spaces restricted to the boundary.

\begin{definition}[$\fC^{k,\lambda}$ domain]\label{def::spcialdomain}
Let $k\in \sN_0$, $\lambda\in (0,1]$. A bounded domain $\Omega\subseteq \sR^{d}$ is said to belong to the class $\fC^{k,\lambda}$ if and only if there exist
\begin{enumerate}
    \item $l$ Cartesian coordinate $\fX_{r}$ ($l\in \sN$, $r\in \{1,\ldots, l\}$)
    \begin{equation*}
        \fX_{r}=(x_{r_1},\ldots,x_{r_{d-1}},x_{r_d})=:(x'_{r},x_{r_d}),
    \end{equation*}
    with $x'_{r}:=(x_{r_1},\ldots,x_{r_{d-1}})$. And suppose a fixed Cartesian coordinate $\fX$ is given in $\sR^d$, then $\fX_{r}$ is obtained by $\fX_{r}=\fO_{r}\fX+\fB_{r}=:\fF_{r}\fX$, where $\fO_{r}$ is an orthogonal matrix with determinant 1, and $\fB_{r}\in \sR^d$ is a vector. $\fF_{r}$ is used to denote the coordinate transformation.
    \item a number $\alpha>0$ and $m$ functions $\gamma_r$
    \begin{equation*}
        \gamma_r\in \fC^{k,\lambda}(\overline{B_{r}^{\kappa}}),
    \end{equation*}
    where $B_{r}^{\kappa}:=\{x'_{r}:\abs{x'_{r}}<\kappa\}\subseteq \sR^{d-1}$.
    \item a number $\mu>0$ and bounded open sets $U_{r}:=\Gamma_r\bigcup U_{r}^{+}\bigcup U_{r}^{-}$such that
    \begin{enumerate}
        \item the sets
        \begin{equation*}
            \Gamma_{r}=\fF_{r}^{-1}\{(x'_{r},x_{r_d}):x'_{r}\in B_{r}^{\kappa}, x_{r_d}=\gamma_r(x'_{r})\},
        \end{equation*}
        are subsets of $\partial\Omega$ for $r\in\{1,\ldots,l\}$ and $\partial\Omega=\bigcup_{r=1}^{m}\Gamma_r$.
        \item for $r\in \{1,\ldots,l \}$ the set
        \begin{equation*}
            U_{r}^{+}=\fF_{r}^{-1}\{(x'_{r},x_{r_d}):x'_{r}\in B_{r}^{\kappa}, \gamma_r(x'_{r})< x_{r_d}< \gamma_r(x'_{r})+\mu\},
        \end{equation*}
        is a subset of $\Omega$.
        \item for $r\in \{1,\ldots,l \}$ the set
        \begin{equation*}
            U_{r}^{-}=\fF_{r}^{-1}\{(x'_{r},x_{r_d}):x'_{r}\in B_{r}^{\kappa}, \gamma_r(x'_{r})-\mu< x_{r_d}< \gamma_r(x'_{r})\},
        \end{equation*}
        is a subset of $\sR^d-\overline{\Omega}$. 
    \end{enumerate}
\end{enumerate}
\end{definition}
Unless otherwise stated, the domain $\Omega$ is supposed to be of class $\fC^{1,0}$ throughout the paper. And for each Cartesian coordinate $\fX_{r}$, we can define a corresponding ``straighten'' coordinate $\fY_r:=\{y_{r_1},\ldots,y_{r_{d-1}},y_{r_d}\}$ by 
\begin{equation}
\left\{
\begin{aligned}
    y_{r_{\alpha}} &= x_{r_{\alpha}} := \Phi_{r_{\alpha}}(\fX_{r}),
    &&\quad \text{for $\alpha \in \{1,\ldots,d-1\}$}, \\
    y_{r_d} &= x_{r_d} - \gamma_r(x'_{r}) := \Phi_{r_d}(\fX_{r}).
\end{aligned}
\right.
\end{equation}
and define $\fY_r=\Phi_r(\fX_{r})$. Similarly, its inverse can be defined as $\fX_{r}=\Psi(\fY_r)$ with $\Psi=\Phi^{-1}$ and $\operatorname{det}\Phi=\operatorname{det}\Psi.$

Note that for a function $w$ defined on $\Gamma_r$, corresponding $_{r}w(x'_{r})$ is defined as 
\begin{equation*}
    _{r}w(x'_{r})=w\circ \fF_{r}^{-1}(x'_{r},\gamma_r(x'_{r})).
\end{equation*}

\begin{definition}[$L^2$ integrable on boundary]\label{def::LpIntegrationOnBoundary}
    A function $w$ defined almost everywhere on $\partial \Omega$ is said to belong to the space $L^{2}_{r}(\partial\Omega)$ if each function $_{r}u \in L^2(B_{r}^{\kappa})$ 
    and satisfies
    \begin{equation*}
        \int_{B_{r}^{\kappa}}\Abs{_{r}w}^2 \diff{x'_{r}}<\infty,\quad \text{for $r\in\{1,\ldots,l\}$}.
    \end{equation*}
    And the class $L^2(\partial \Omega)$ is equipped with norm
    \begin{equation*}
        \norm{w}_{L^2(\partial\Omega)}=\left(\sum_{r=1}^{l}\int_{B_{r}^{\kappa}}\Abs{_{r}w}^2 \diff{x'_{r}}\right)^{1/2}=:\left(\sum_{r=1}^{l}\norm{{_{r}w}}^{2}_{L_{r}^{2}(B_{r}^{\kappa})}\right)^{1/2}.
    \end{equation*}
\end{definition}
\begin{remark}
    Note that the subscript $r$ in $L^{2}_{r}$ and $H^1_{r}$ means $L^2$ and $H^1$ integrable under Cartesian coordinate $\fX_{r}$.
\end{remark}

\begin{remark}
    One may doubt that the definition of $L^2(\partial\Omega)$ depends on the description of $\partial \Omega$. However, it can be proved that if another $\tilde{L}^2(\partial\Omega)$ is obtained by another description of $\partial\Omega$, then $\norm{\cdot}_{L^2((\partial\Omega))}$ and $\norm{\cdot}_{\tilde{L}^2(\partial\Omega)} $ are equivalent
\end{remark}

\begin{definition}[$H^1$ integrable on boundary]\label{def::WkpIntegrationOnBoundary}
    Let $\Omega\in \fC^{0,1}$. Denote by $H^1(\partial\Omega)$ the subspace of all functions $w\in L^2(\partial\Omega)$ such that $_{r}w\in H^1(B_{r}^{\kappa})$ for $r\in\{1,\ldots,l\}$ and equipped with the norm
    \begin{equation*}
        \begin{aligned}
            \norm{w}_{H^1(\partial\Omega)}&=\left(\sum_{r=1}^{l}\left(\norm{{_{r}w}}^{2}_{L_{r}^{2}(B_{r}^{\kappa})}+\sum_{\alpha=1}^{d-1}\int_{B_{r}^{\kappa}}\abs{D_{x_{r_{\alpha}}} {_{r}w}}^2  \diff{x'_{r}}\right)\right)^{1/2}\\
            &=:\left(\sum_{r=1}^{l}\norm{{_{r}w}}^{2}_{H_{r}^{1}(B_{r}^{\kappa})}\right)^{1/2}.
        \end{aligned}
    \end{equation*}
\end{definition}

\subsection{PDEs and Assumptions}\label{sec::Setup}

 In this paper, we primarily focus on analyzing the stability properties of our proposed \textbf{SSBE} method for elliptic and parabolic problems. Next, we present the assumptions associated with the PDE problems under consideration.

Let a bounded domain $\Omega\subset \sR^d$ belong the class $\fC^{1,0}$. Application of \textbf{SSBE}~\eqref{eq::H1convergenceAlgorithm} and analysis of stability and a priori generalization bound are exhibited on the following PDE problems
\begin{enumerate}
    \item Elliptic problem
        \begin{equation}\label{eq::EllipticProblem}
\left\{
\begin{aligned}
    \fL u :=\,& -\operatorname{div}(A(x)Du) + \sum_{\alpha=1}^{d} b_{\alpha}(x) D_{\alpha}u + c(x)u = f(x),
    &&\quad \text{for $x \in \Omega$}, \\
    u =\,& g(x),
    &&\quad \text{for $x \in \partial\Omega$}.
\end{aligned}
\right.
\end{equation}

    \item Parabolic problem
        \begin{equation}\label{eq::ParabolicProblem}
\left\{
\begin{aligned}
    u_t + \fL_t u &= f(x_t), 
    &&\quad \text{for $x_t \in \Omega_T$}, \\
    u(x_t) &= g(x_t), 
    &&\quad \text{for $x_t \in \partial\Omega_T$}, \\
    u(x,0) &= \nu(x), 
    &&\quad \text{for $x \in \Omega$}.
\end{aligned}
\right.
\end{equation}
where $\Omega_{T}=\Omega\times[0,T]$, $\partial\Omega_{T}:=\partial\Omega\times [0,T]$, $x_t:=(x,t)\in\sR^{d+1}$ and 
        \begin{equation*}
            \fL_tu:=-\operatorname{div}(A(x_t)D_{\alpha\beta}u)+\sum_{\alpha}^{d}b_{\alpha}(x_t)D_{\alpha}u+c(x_t)u.
        \end{equation*}
\end{enumerate}
The following assumptions are required to guarantee the existence of solution so that the mathematical analysis can be performed.
\begin{assump}[Assumption for elliptic problem]\label{assum::Elliptic}
    Let $A(x)=[a_{\alpha\beta}(x)]^{d\times d}\in [W^{1,\infty}(\Omega)]^{d\times d}$ satisfies that there is $\lambda>0$ such that for all $\mu\in \sR^d$ and $x\in\Omega$
    \begin{equation}
        \lambda\Abs{\mu}^2\le \mu^{\T}A(x)\mu\le \frac{1}{\lambda}\Abs{\mu}^2,
    \end{equation} 
    $\vb(x):=\{b_{\alpha}(x)\}_{\alpha\in\{1,\ldots,d\}}^{\T}\in [L^{\infty}(\Omega)]^{d}$ and $c(x)\in L^{\infty}(\Omega)$ satisfying $\operatorname*{ess\,inf}_{x}  c(x)>\frac{2d\norm{\vb(x)}_{L^{\infty}(\Omega)}+2}{\lambda}$.   And $f\in L^2(\Omega)$, $g\in H^1(\partial \Omega)$.
\end{assump}

\begin{assump}[Assumption for parabolic problem]\label{assum::Parabolic}
    Let $A(x_t)=[a_{\alpha\beta}(x_t)]^{d\times d}\in [W^{1,\infty}(\Omega_{T})]^{d\times d}$ satisfies that there is $\lambda>0$ such that for all $\mu\in \sR^d$ and $x_t\in\Omega_{T}$
    \begin{equation}
        \lambda\Abs{\mu}^2\le \mu^{\T}A(x_t)\mu\le \frac{1}{\lambda}\Abs{\mu}^2,
    \end{equation} 
    $\vb(x_t):=\{b_{\alpha}(x_t)\}_{\alpha\in[1,\ldots,d]}^{\T}\in [L^{\infty}(\Omega_{T})]^{d}$ and $c(x_t) \in L^{\infty}(\Omega_{T})$ satisfying $\operatorname*{ess\,inf}_{x}   c(x_t)>\frac{2d\norm{\vb(x)}_{L^{\infty}(\Omega_{T})}+2}{\lambda}.$ And $f\in L^2(\Omega_{T})$, $g\in L^2(0,T;H^1(\partial\Omega))$, $g_t\in L^2(0,T;L^2(\partial\Omega))$.
\end{assump}

\begin{remark}[A non-divergence representation of PDE operators]\label{rmk::NonDivergence}
    Note that $\fL u$  can be written as
        \begin{equation}
            \fL u:=-\sum_{\alpha,\beta=1}^{d}a_{\alpha\beta}(x)D_{\alpha\beta}u(x)+\sum_{\alpha=1}^{d}\hat{b}_{\alpha}(x)D_{\alpha}u+c(x)u,
        \end{equation}
        where $\hat{b}_{\alpha}(x)=b_{\alpha}(x)+\sum_{\beta=1}^{d}D_{\beta}a_{\alpha\beta}(x)$ and $\hat{\vb}(x)=\{\hat{b}_{\alpha}(x)\}_{\alpha\in\{1,\ldots,d\}}^{\T}$.
\end{remark}

\subsection{Barron Space}\label{sec::BarronSpace}

Our proposed \textbf{SSBE} method builds upon deep learning-based PDE solvers. Here, we introduce several essential concepts in preparation for our analysis presented in Section~\ref{sec::GeneralizationBound}.

\begin{definition}[Fully connected feed-forward neural network (FNN)]
    An $L+1$-layer FNN defined on $\sR^d$ takes the form
    \begin{equation*}
        \phi(x,\theta)=\va^{\T}\vh^{[L]}\circ\vh^{[L-1]}\circ\ldots\circ\vh^{[1]}(x),
    \end{equation*}
    where $\vh^{[l]}(x)=\sigma(\vW^{[l]}x+\vb^{[l]})$ with $\vW^{[l]}\in \sR^{m_l\times m_{l-1}}$, $\vb^{[l]}\in\sR^{m_l}$ for $l\in \{1,\ldots,L\}$, $m_0=d$ and $\va\in \sR^{m_L}$. $\sigma$ is a non-linear activation function. And $\vtheta:=\operatorname{vec}\{\va,\{\vW^{[l]},\vb^{[l]}\}_{l=1}^{L}\}$ denotes all the parameters in $\phi$.
\end{definition}
In fact, for an FNN with $\vb^{[l]}$, we can define $\tilde{x}=(x^{\T},1)^{\T}$ and $\widetilde{\vW}^{[l]}=(\vW^{[l]},\vb^{[l]})$, then $\widetilde{\vW}^{[l]}\tilde{x}=\vW^{[l]}x+\vb^{[l]}$. Thus we abuse the notation of $\widetilde{\vW}^{[l]}, \tilde{x}$ and $\vW^{[l]}, x$ in the following discussion.

We focus on two-layer FNN with defined path norm throughout the paper and introduce Barron space with its Barron norm below.

\begin{definition}[Path norm]
    The path norm of a two-layer FNN
    \begin{equation*}
        \phi(x,\theta)=\va^{\T}\sigma(\vW x)=\sum_{k=1}^{m}a_k\sigma(\vw_k^{\T}x),
    \end{equation*}
    with an activation $\sigma$ and parameter $\vtheta$ is
    \begin{equation*}
        \norm{\vtheta}_{\fP}=\sum_{k=1}^{m}\abs{a_k}\abs{\vw_k}^3,
    \end{equation*}
    where $\abs{\cdot}$ of a vector means the $l_2$ norm.
\end{definition}

\begin{definition}[Barron pair]
    A function pair $(f,g): \Omega\times\partial\Omega\to\sR^2$ is said to belong to the class of Barron type if it takes an integral representation
    \begin{equation}\label{eq::BarronPair}
    \left\{
        \begin{aligned}
            f(x)&= \Exp_{(a,\vw)\sim\rho}a\left[\vw^{\T}A(x)\vw\sigma''(\vw^{\T}x)+\hat{\vb}^{\T}(x)\vw\sigma'(\vw^{\T}x)+c(x)\sigma(\vw^{\T}x)\right], &&\quad x\in \Omega, \\
            g(x)&= \Exp_{(a,\vw)\sim\rho} a\sigma(\vw^{\T}x), &&\quad x\in \partial \Omega.
        \end{aligned}
        \right.
    \end{equation}
    where $\rho$ is a probability distribution over $\sR^{d+1}$. The associate Barron norm is
    \begin{equation}
        \norm{(f,g)}_{\fB}=\inf_{\rho\in\fP}\left(\Exp_{(a,\vw)\sim\rho}\Abs{a}^2\Abs{\vw}^6\right)^{1/2},
    \end{equation}
    where $\fP$ is the class of probability distribution such that~\eqref{eq::BarronPair} holds. And the Barron-type space is 
    \begin{equation}
        \fB(\Omega,\partial\Omega):=\{(f,g): \Omega\times\partial\Omega\to\sR^2|\norm{(f,g)}_{\fB}<+\infty\}.
    \end{equation}
\end{definition}
\begin{remark}
    For a given bounded domain $\Omega$, the Dirichlet problem $\fL u=f, x\in\Omega;$ $u=g, x\in\partial \Omega$ is considered over $\Omega$. If $(f,g)$ belongs to the class of Barron pair, then the solution $u$ takes the form $u=\Exp_{(a,\vw)\sim \rho}a\sigma(\vw^{\T}x)$ and belongs to the Barron space, which was introduced with the motivation to create a “reasonably simple and transparent framework for machine learning”~\cite{E2020Some} and consists of the union of all reproducing kernel Hilbert spaces (RKHS) introduced by random feature kernels~\cite{E2022Thefluid}.
\end{remark}
 
\begin{definition}[Barron triplet]
    Let $x_t=(x,t)\in \sR^{d+1}$, a function triplet $(f(x_t),g(x_t),\nu(x)): \Omega_{T}\times\partial\Omega_{T}\times\Omega\to\sR^3$ is said to belong to the class of Barron type if it takes an integral representation
    \begin{equation}\label{eq::BarronTriplet}
\left\{
\begin{aligned}
    f(x_t) &= \mathbb{E}_{(a,\vw)\sim\rho} \, a \Big[\vw^{\T} A(x_t) \vw\, \sigma''(\vw^{\T}x_t)
    + \hat{\vb}^{\T}(x_t) \vw\, \sigma'(\vw^{\T}x_t)
    + c(x_t) \sigma(\vw^{\T}x_t)\Big],
    &&\quad x_t \in \Omega_T, \\
    g(x_t) &= \mathbb{E}_{(a,\vw)\sim\rho} \, a\, \sigma(\vw^{\T}x_t),
    &&\quad x_t \in \partial \Omega_T, \\
    \nu(x) &= \mathbb{E}_{(a,\vw)\sim\rho} \, a\, \sigma(\vw^{\T}(x,0)),
    &&\quad x \in \Omega.
\end{aligned}
\right.
\end{equation}
where $\rho$ is a probability distribution over $\sR^{d+2}$. The associate Barron norm is
    \begin{equation}
        \norm{(f,g,\nu)}_{\fB}=\inf_{\rho\in\fP}\left(\Exp_{(a,\vw)\sim\rho}\Abs{a}^2\Abs{\vw}^6\right)^{1/2},
    \end{equation}
    where $\fP$ is the class of probability distribution such that~\eqref{eq::BarronTriplet} holds. And the Barron-type space is 
    \begin{equation}
        \fB(\Omega_{T},\partial\Omega_{T},\Omega):=\{(f,g),\nu: \Omega_{T}\times\partial\Omega_{T}\times\Omega\to\sR^3|\norm{(f,g,\nu)}_{\fB}<+\infty\},
    \end{equation}
    where $\Omega_{T}=\Omega\times[0,T]$, $\partial\Omega_{T}:=\partial\Omega\times [0,T]$.
\end{definition}
\section{Failure of PINNs to Achieve $H^1$ Convergence} \label{sec::CounterExample}

In this section, we construct a two-dimensional model to demonstrate that the Physics-Informed Neural Network (PINN) method is unsuitable for solving partial differential equations (PDEs) requiring convergence in the $H^1$ norm. Consider the PDE defined on a domain $\Omega\in\sR^2$: $$\fL u(x_1,x_2)=f(x_1,x_2), \ (x_1,x_2)\in \Omega,\quad \text{and}\quad u(x_1,x_2)=g(x_1,x_2), \ (x_1,x_2)\in \partial\Omega.$$
The PINN loss function associated with the approximation $v$ is defined as
\begin{equation}
    \fR_{D}^{\operatorname{PINN}}(v)=\frac{1}{\Abs{\Omega}}\norm{\fL v-f}_{L^2(\Omega)}^2+\frac{\lambda}{\Abs{\partial\Omega}}\norm{v-g}_{L^2(\partial\Omega)}^2.
\end{equation}
Given an operator $\mathcal{L}$, we can construct a sequence of functions ${u_i(x_1,x_2)}_{i=1}^{\infty}$ satisfying
\begin{equation}
    \begin{aligned}
        \lim_{i\to \infty}\frac{1}{\Abs{\Omega}}\norm{\fL u_i-f}_{L^2(\Omega)}^2+\frac{1}{\Abs{\partial\Omega}}\norm{ u_i-g}_{L^2(\partial\Omega)}^2\to 0,\quad \norm{\fB u_i-g}_{L^2(\partial\Omega)}^2\neq 0.
    \end{aligned}
    \end{equation} while simultaneously exhibiting non-convergence in the $H^1$ norm:
 \begin{equation}
        \frac{\norm{u_i-u}_{H^1(\Omega)}}{\norm{u}_{H^1(\Omega)}}=O(1),\quad \text{for all $i\in \{1,\ldots,\infty\}$}.
    \end{equation}

\noindent\textbf{Example Construction}
To illustrate this phenomenon, consider a specific example. 
Let $\Omega=B(0,1)$ be the unit ball in $\mathbb{R}^2$, and set the operator $\mathcal{L}=-\Delta$ with $f(x_1,x_2)=4$ and boundary condition $g(x_1,x_2)=0$. The PDE problem becomes
\begin{equation}\label{eq::CounterExample}
\left\{
\begin{aligned}
    -\Delta u &= 4, 
    &&\quad (x_1, x_2) \in B(0,1), \\
    u &= 0, 
    &&\quad (x_1, x_2) \in \partial B(0,1).
\end{aligned}
\right.
\end{equation}
with the exact solution $u(x_1,x_2)=1-(x_1^2+x_2^2)$.

Now define a sequence of perturbation functions $\{v_i\}_{i=1}^{\infty}$ in polar coordinates $(r,\theta)$ as
\begin{equation}\label{eq::PerturbationFunction}
v_i(r,\theta)=\frac{1}{i}\sin(i\theta) r^i,
\end{equation}
where $r=\sqrt{x_1^2+x_2^2}$. Each $v_i$ is harmonic in $B(0,1)$, satisfying $-\Delta v_i=0$ for all $i$. These perturbations exhibit the following properties:

\begin{itemize}
\item The boundary $L^2$ norm of each $v_i$ is given by
\begin{equation}\label{eq::L2BoundaryPerturbation}
           \norm{v_i}_{L^2(\partial B(0,1))}^2=\int_{0}^{2\pi}\frac{1}{i^2}\sin^2 (i\theta)\diff{\theta}=\frac{\pi}{i^2}=O\left(\frac{1}{i^2}\right).
        \end{equation}

\item The domain $H^1$ norm of each $v_i$ is calculated as
            \begin{equation}
                \norm{v_i}_{L^2(B(0,1))}^2=\int_{0}^{1}\int_{0}^{2\pi}\frac{1}{i^2}\sin^2 (i\theta)r^{2i}\diff{\theta}\diff{r}=\frac{\pi}{(2i+1)i^2},
            \end{equation} 
            and
            \begin{equation}
                \begin{aligned}
                    \norm{D_{x_1} v_i}_{L^2(B(0,1))}^{2}+\norm{D_{x_2} v_i}_{L^2(B(0,1))}^{2}&=\int_{B(0,1)}(D_{x_1} v_i)^2+(D_{x_2} v_i)^2\diff{x_1}\diff{x_2}
                    \\&=\int_{0}^{1}\diff{r}\int_{0}^{2\pi}\left((D_r v_i)^2+\frac{1}{r^2}(D_{\theta} v_i)^2\right)r\diff{\theta}=\frac{\pi}{i}.
                \end{aligned}
            \end{equation} 
            Thus, the full $H^1$ norm satisfies
            \begin{equation}\label{eq::H1DomainPerturbation}
                 \norm{v_i}_{H^1(B(0,1))}^{2}=\frac{(2i^2+i+1)\pi}{(2i+1)i^2}=O\left(\frac{1}{i}\right).
            \end{equation}
\end{itemize}

Consider the sequence $u_i=u+\frac{v_i}{\norm{v_i}_{H^1(B(0,1))}}$. According to equations~\eqref{eq::L2BoundaryPerturbation} and~\eqref{eq::H1DomainPerturbation}, we observe that
    \begin{equation}
        \lim_{i\to\infty} \left(\frac{1}{\Abs{B(0,1)}}\norm{-\Delta u_i-4}_{L^2(B(0,a))}^2+\frac{1}{\Abs{\partial B(0,1)}}\norm{u_i-0}_{L^2(\partial\Omega)}^2\right)=\lim_{i\to\infty}O\left(\frac{1}{i}\right)\to 0.
    \end{equation}
which means the PINN objective approaches zero. However, considering the relative $H^1$ norm error between $u_i$ and the exact solution $u$, we find
    \begin{equation}
        \frac{\norm{u_i-u}_{H^1(\Omega)}^2}{\norm{u}_{H^1(\Omega)}^2}= \frac{1}{\norm{u}_{H^1(\Omega)}^2}=O(1),\quad \text{ for all $i\in\{1,\ldots,\infty\}$},
    \end{equation}
since $\norm{\frac{v_i}{\norm{v_i}_{H^1(B(0,1))}}}_{H^1(\Omega)}^2=1$.

We present a pedagogical counterexample illustrating that achieving minimal boundary loss in the $L^2$-norm (on the order of $O(\varepsilon)$) can still lead to a significant $O(1)$ relative error between the numerical solution and the ground truth. This phenomenon arises because the conventional boundary loss does not adequately capture the intrinsic regularity constraints imposed by PDE operators. Consequently, a negligible boundary mismatch can propagate disproportionately, causing substantial errors within the domain.
\section{SSBE: Sobolev-Stable Boundary Enforcement}\label{sec::Experiments}
To address the critical limitation discussed above, we introduce a modified boundary loss formulation named \textbf{SSBE}, which redefines the population loss for an objective function $v$ as follows:
\begin{equation}\label{eq::H1convergenceAlgorithm}
    \fR_{D}(v)=\frac{1}{\Abs{\Omega}}\norm{\fL v-f}_{L^2(\Omega)}^2+\frac{\lambda}{\Abs{\partial\Omega}}\norm{\fB v-g}_{H^1(\partial\Omega)}^2,
\end{equation}
where $\lambda>0$ is an adjustable weight parameter. This reformulated loss explicitly incorporates regularity constraints through the $H^1$ norm, aligning the error propagation dynamics more closely with the underlying PDE structure and thus enhancing convergence properties. For simplicity in analysis, we set the weight parameter to $\lambda=\frac{\Abs{\partial\Omega}}{\Abs{B_{r}^{\kappa}}}$, where $\Abs{B_{r}^{\kappa}}$ represents the volume of a $(d-1)$-dimensional ball of radius $\kappa$, which remains invariant for all $r\in\{1,\ldots,l\}$.

Consider the Cartesian coordinate-function pairs $\{\fX_r,\gamma_{r}\}_{r=1}^{l}$ or $\{(\fX_{r},t),\gamma_{r}\}_{r=1}^{l}$ used to straighten the boundary $\partial\Omega$ or $\partial\Omega_{T}$, as discussed in Section~\ref{sec::BoundarySobolev}. For the objective function $v\in H^2(\Omega)$ or, in the time-dependent case, $v\in L^2(0,T;H^2(\Omega))$ with $v_t\in L^2(0,T;L^2(\Omega))$, the SSBE algorithm~\eqref{eq::H1convergenceAlgorithm} applied to elliptic and parabolic problems can be described explicitly with definitions in Section~\ref{sec::BoundarySobolev} as follows:
\begin{itemize}
    \item Elliptic problem
        \begin{equation}\label{eq::AlgorithmRepre4EllipticPopulation}
            \begin{aligned}
                \fR_{D}(v)
                &=\frac{1}{\Abs{\Omega}}\norm{\fL v(x)-f(x)}_{L^2(\Omega)}^{2}+\sum_{r=1}^{l}\frac{1}{\Abs{B_{r}^{\kappa}}}\norm{{_{r}v(x_r')}-{_{r}g(x'_r)}}_{H_{r}^{1}(B_{r}^{\kappa})}^2.
            \end{aligned}
        \end{equation}
Given randomly sampled data points $S_{\Omega}=\{x^{i}\}_{i=1}^{n{\Omega}}\subset \sR^{d}$ from the domain and $S_r=\{{x_{r}'}^{i}\}_{i=1}^{n_r}\subset \sR^{d-1}$ from each straightened boundary $B_{r}^{\kappa}$, the corresponding empirical loss is:
        \begin{equation}\label{eq::AlgorithmRepre4EllipticEmpirical}
            \begin{aligned}
                \fR_{S}(v)
                &=\frac{1}{n_{\Omega}}\sum_{i=1}^{n_{\Omega}}\Abs{\fL v(x^{i})-f(x^i)}^{2}\\
                &~~~+\sum_{r=1}^{l}\frac{1}{n_r}\sum_{i=1}^{n_r}\left(\Abs{{_{r}v({x_{r}'}^{i})}-{_{r}g({x_{r}'}^{i})}}^2+\sum_{\alpha=1}^{d-1}\Abs{D_{x_{r_{\alpha}}}{_{r}v({x_{r}'}^{i})}-D_{x_{r_{\alpha}}}{_{r}g({x_{r}'}^{i})}}^2\right).
            \end{aligned}
        \end{equation}
    \item Parabolic problem
        \begin{equation}\label{eq::AlgorithmRepre4ParabolicPopulation}
            \begin{aligned}
                \fR_D(v)
                &=\frac{1}{\Abs{\Omega_{T}}}\norm{v_t(x_t)+\fL v(x_t)-f(x_t)}_{L^2(\Omega_{T})}^2+\frac{1}{\Abs{\Omega}}\norm{v(x,0)-\nu(x)}_{L^2(\Omega)}^2\\
                &+\frac{1}{T}\left(\frac{1}{\abs{\partial\Omega}}\norm{v_t(x,t)-g_t(x_t)}_{L^2(0,T;L^2(\partial\Omega))}^2+\sum_{r=1}^{l}\frac{1}{\Abs{B_{r}^{\kappa}}}\norm{{_{r}v(x_{r}',t)}-{_{r}g(x'_{r},t)}}_{L^2(0,T;H_{r}^{1}(B_{r}^{\kappa}))}^2\right).
            \end{aligned}
        \end{equation}
        
Given randomly sampled data points $S_{\Omega_{T}}=\{x_{t}^{i}\}_{i=1}^{n{\Omega_{T}}}\subset \sR^{d+1}$ in the domain $\Omega_{T}$, boundary samples $S_{\partial\Omega}=\{(x_{\partial\Omega}^{i},t_j)\}_{i=1,j=1}^{n{\partial\Omega}, n_{T}}\subset\sR^{d+1}$ on $\partial\Omega_{T}$, samples $S_r=\{({x_{r}'}^{i},t_j)\}_{i=1,j=1}^{n_r,n_T}\subset \sR^{d}$ on each straightened boundary segment $B{r}^{\kappa}\times [0,T]$ with $r\in\{1,\ldots,l\}$, and initial domain samples $S_{\Omega}=\{x^{i}\}_{i=1}^{n{\Omega}}\subset \sR^{d}$, the empirical loss is expressed as follows:
         \begin{equation}\label{eq::AlgorithmRepre4ParabolicEmpirical}
             \begin{aligned}
                 \fR_{S}(v)
                &=\frac{1}{n_{\Omega}}\sum_{i=1}^{n_{\Omega_{T}}}\Abs{v_t(x_{t}^{i})+\fL v(x_{t}^{i})-f(x_{t}^i)}^{2}+\frac{1}{n_{\Omega}}\sum_{i=1}^{n_{\Omega}}\Abs{v(x^i,0)-\nu(x^i)}^2\\
                &+\frac{1}{n_{\partial\Omega}n_T}\sum_{i=1,j=1}^{n_{\partial\Omega},n_T}\Abs{v_t(x_{\partial\Omega}^i,t_j)-g_t(x_{\partial\Omega}^i,t_j)}^2\\
                &+\sum_{r=1}^{l}\frac{1}{n_rn_T}\sum_{i=1,j=1}^{n_r,n_{T}}\left(\Abs{{_{r}v({x_{r}'}^{i},t_j)}-{_{r}g({x_{r}'}^{i},t_j)}}^2+\sum_{\alpha=1}^{d-1}\Abs{D_{x_{r_{\alpha}}}{_{r}v({x_{r}'}^{i},t_j)}-D_{x_{r_{\alpha}}}{_{r}g({x_{r}'}^{i},t_j)}}^2\right).
             \end{aligned}
         \end{equation}
\end{itemize}
\begin{remark}
   The proposed algorithm can also be naturally extended to hyperbolic problems, though we do not elaborate on these cases to maintain clarity and conciseness in this paper.
\end{remark}

In the subsequent sections, we will analyze the stability and convergence properties of the proposed SSBE loss and validate the effectiveness of our method through numerical experiments.
\section{$H^1$-Stability and Convergence Analysis of \textbf{SSBE} for PDEs}\label{sec::Stability}
In this section, we present a comprehensive stability analysis to demonstrate \textbf{SSBE} reliably achieves $H^1$-convergence for both elliptic and parabolic partial differential equations. We rigorously establish the conditions under which the proposed method maintains stability in the $H^1$ norm, thereby ensuring its robustness in practical applications. Through a detailed examination of the underlying mathematical framework, we elucidate the key factors that govern the algorithm's performance and offer theoretical insights that substantiate its convergence properties. 

Let us start with an extension theorem from boundary to interior domain, it states as follows:
\begin{theorem}[Extension Theorem: $H^{1}(\partial\Omega)\to H^{1}(\Omega)$]\label{thm::ExtensionH1}
    Suppose $\Omega\in \sR^{d}$ belongs to the class $\fC^{0,1}$ with $d\ge 2$, and there is an $g$ defined on $\partial \Omega$ such that $g\in H^{1}(\partial \Omega)$, then there are a bounded linear operator
    \begin{equation*}
        T: H^{1}\left(\partial \Omega\right) \rightarrow H^{1} \left(\Omega\right).
    \end{equation*}
    such that $Tg \lfloor_{\partial \Omega} = g$ and a constant $C$ depending only on $\Omega, d$ such that
    \begin{equation}
        \norm{Tg}_{H^1(\Omega)}\le C\norm{g}_{H^1(\partial\Omega)}.
    \end{equation}
\end{theorem}
\begin{remark}
    While a more canonical theorem exists in the literature which states the existence of bounded linear operator $T:W^{1-\frac{1}{p},p}(\partial\Omega)\to W^{1,p}(\Omega)$ for $p>1$, its proof inherently relies on theories of fractional derivatives. To maintain the rigor and self-contained nature of this work, we provide an alternative fundamental proof that circumvents these technical complexities which can be found in Section~\ref{sec::Extension} of Appendix.
\end{remark}

\noindent \textbf{Stability Analysis for Elliptic Problem}
\begin{theorem}[Stability for elliptic problem]\label{thm::Stability4EllipticProblem}
    Suppose Assumption~\ref{assum::Elliptic} holds, $\Omega\subset \sR^d$ belongs to the class $\fC^{1,0}$. Let $u\in H^1(\Omega)$ be the solution to~\eqref{eq::EllipticProblem} and, for objective function $v\in H^2(\Omega)$, define
    \begin{equation*}
        \begin{aligned}
            \fR_{D}(v)
                &=\frac{1}{\Abs{\Omega}}\norm{\fL v(x)-f(x)}_{L^2(\Omega)}^{2}+\frac{1}{\Abs{B_{r}^{\kappa}}}\norm{v(x)-u(x)}_{H^1(\partial\Omega)}^2\\
                &=\frac{1}{\Abs{\Omega}}\norm{\fL v(x)-f(x)}_{L^2(\Omega)}^{2}+\sum_{r=1}^{l}\frac{1}{\Abs{B_{r}^{\kappa}}}\norm{{_{r}v(x_r')}-{_{r}g(x'_r)}}_{H_{r}^{1}(B_{r}^{\kappa})}^2.
        \end{aligned}
    \end{equation*}
    Then there is a constant $C$ depending only on $\Omega$, $d$, $A(x)$, $\vb(x)$ and $c(x)$ such that for all $v\in H^2(\Omega)$
    \begin{equation}
        \norm{u-v}_{H^1(\Omega)}<C\fR_{D}(v).
    \end{equation}
\end{theorem}
Note that, by definition, $\Abs{B_{r}^{\kappa}}$ is the volume of $(d-1)$-dimensional ball with radius $\kappa$ and stay invariant for all $r\in \{1,\ldots,l\}$.
\begin{proof}
    By letting $q:=f-\fL v$ in $\Omega$, $h:=u-v$ on $\partial\Omega$ and $w:=u-v$ in $\Omega$, we have
\begin{equation}
\left\{
\begin{aligned}
    \fL w &= q, 
    &&\quad \text{in $\Omega$}, \\
    w &= h, 
    &&\quad \text{on $\partial\Omega$}.
\end{aligned}
\right.
\end{equation}

    By Theorem~\ref{thm::ExtensionH1}, there is an extension $Eh\in H^1(\Omega)$ such that $Eh\lfloor_{\partial\Omega}=h$ and $\norm{Eh}_{H^1(\Omega)}\le C\norm{h}_{H^1(\partial \Omega)}$. Further defining $\tilde{w}=w-Eh$, then $\tilde{w}$ is the solution to 
    \begin{equation}\label{eq::BVPTransfer}
\left\{
\begin{aligned}
    \fL \tilde{w} &= q - \fL Eh, 
    &&\quad \text{in $\Omega$}, \\
    \tilde{w} &= 0, 
    &&\quad \text{on $\partial\Omega$}.
\end{aligned}
\right.
\end{equation}

    By multiplying $\tilde{w}$ on both sides of the first equation of~\eqref{eq::BVPTransfer} and integration over $\Omega$, there is
    \begin{equation}\label{eq::Estimate4LHS}
        \begin{aligned}
             (LHS)&=\int_{\Omega}\tilde{w}\fL\tilde{w}\diff{x}=\int_{\Omega}-\tilde{w}\operatorname{div}(A(x)D\tilde{w})+\sum_{\alpha=1}^{d}\tilde{w}b_{\alpha}(x)D_{\alpha}\tilde{w}+c(x)\tilde{w}^2\diff{x}\\
             &=\int_{\Omega}D\tilde{w}A(x)D\tilde{w}+\sum_{\alpha=1}^{d}\tilde{w}b_{\alpha}(x)D_{\alpha}\tilde{w}+c(x)\tilde{w}^2\diff{x}\\
             &\ge \lambda\norm{D\tilde{w}}_{L^2(\Omega)}^2+\int_{\Omega}\sum_{\alpha=1}^{d}\tilde{w}b_{\alpha}(x)D_{\alpha}\tilde{w}+c(x)\tilde{w}^2\diff{x},
        \end{aligned}
    \end{equation}
    also, we have
    \begin{equation}\label{eq::Estimate4RHS}
        \begin{aligned}
            (RHS)&=\int_{\Omega}\tilde{w}(q-\fL Eh)\diff{x}\\
            &=\int_{\Omega}\tilde{w}q-D\tilde{w}A(x)DEh-\sum_{\alpha=1}^{d}\tilde{w}b_{\alpha}(x)D_{\alpha}Eh-\tilde{w}c(x)Eh\diff{x}\\
            &\le \frac{\lambda}{4}\norm{D\tilde{w}}_{L^2(\Omega)}^2+\frac{d\norm{\vb(x)}_{L^{\infty}(\Omega)}+2}{\lambda}\norm{\tilde{w}}_{L^2(\Omega)}^2+\left(\frac{\lambda}{4}+\frac{1}{\lambda^3}\right)\norm{DEh}_{L^2(\Omega)}^2\\
            &~~~+\frac{\lambda\norm{c(x)}_{L^{\infty}(\Omega)}}{4}\norm{Eh}_{L^2(\Omega)}^2+\frac{\lambda}{4}\norm{q}_{L^2(\Omega)}^2.
        \end{aligned}
    \end{equation}
    Combining~\eqref{eq::Estimate4LHS} and~\eqref{eq::Estimate4RHS}, we have
    \begin{equation}
        \begin{aligned}
            \lambda\norm{D\tilde{w}}_{L^2(\Omega)}^2+\int_{\Omega}c(x)\tilde{w}^2\diff{x}&\le (RHS)-\int_{\Omega}\sum_{\alpha=1}^{d}\tilde{w}b_{\alpha}(x)D_{\alpha}\tilde{w}\diff{x}\\
            &\le (RHS)+\frac{\lambda}{4}\norm{D\tilde{w}}_{L^2(\Omega)}^2+\frac{d\norm{\vb(x)}_{L^{\infty}(\Omega)}}{\lambda}\norm{\tilde{w}}_{L^2(\Omega)}^2\\
            &\le \frac{\lambda}{2}\norm{D\tilde{w}}_{L^2(\Omega)}^2+\frac{2d\norm{\vb(x)}_{L^{\infty}(\Omega)}+2}{\lambda}\norm{\tilde{w}}_{L^2(\Omega)}^2+\left(\frac{\lambda}{4}+\frac{1}{\lambda^3}\right)\norm{DEh}_{L^2(\Omega)}^2\\
            &~~~+\frac{\lambda\norm{c(x)}_{L^{\infty}(\Omega)}}{4}\norm{Eh}_{L^2(\Omega)}^2+\frac{\lambda}{4}\norm{q}_{L^2(\Omega)}^2.
        \end{aligned}
    \end{equation}
    Since $\min c(x)>\frac{2d\norm{\vb(x)}_{L^{\infty}(\Omega)}+2}{\lambda}$, we have
    \begin{equation}
        \frac{\lambda}{2}\norm{D\tilde{w}}_{L^2(\Omega)}^2\le \left(\frac{\lambda}{4}+\frac{1}{\lambda^3}\right)\norm{DEh}_{L^2(\Omega)}^2+\frac{\lambda\norm{c(x)}_{L^{\infty}(\Omega)}}{4}\norm{Eh}_{L^2(\Omega)}^2+\frac{\lambda}{4}\norm{q}_{L^2(\Omega)}^2,
    \end{equation}
    this further leads to the existence of constant $C$ such that
    \begin{equation}
        \norm{D\tilde{w}}_{L^2(\Omega)}^2\le C(\norm{Eh}_{H^1(\Omega)}^2+\norm{q}_{L^2(\Omega)}^2),
    \end{equation}
    by the Poincare's inequality and the fact that $\norm{Eh}_{H^1(\Omega)}\le C\norm{h}_{H^1(\partial \Omega)}$ for some $C>0$, we have 
    \begin{equation}
        \norm{w-Eh}_{H^1(\Omega)}^2\le C(\norm{h}_{H^1(\partial\Omega)}^2+\norm{q}_{L^2(\Omega)}^2),
    \end{equation}
    hence
    \begin{equation}
        \norm{w}_{H^1(\Omega)}^2\le C\fR_D(v), 
    \end{equation}
    where $C$'s are constants being different from line to line.
\end{proof}

\noindent \textbf{Stability Analysis for Parabolic Problem}

\begin{theorem}[Stability for parabolic problem]
    Suppose Assumption~\ref{assum::Parabolic} holds, $\Omega\subset \sR^d$ belongs to the class $\fC^{1,0}$. Let $u\in L^2(0,T;H^1(\Omega))$ be the solution to~\eqref{eq::ParabolicProblem}, and, for any objective function $v\in L^2(0,T;H^2(\Omega))$, $v_t\in L^2(0,T;L^2(\Omega))$, define
    \begin{equation*}
        \begin{aligned}
            \fR_D(v)
            &=\frac{1}{\Abs{\Omega_{T}}}\norm{v_t(x_t)+\fL v(x_t)-f(x_t)}_{L^2(\Omega_{T})}^2+\frac{1}{\Abs{\Omega}}\norm{v(x,0)-\nu(x)}_{L^2(\Omega)}^2\\
            &~~~+\frac{1}{\Abs{\partial\Omega}T}\norm{v_t(x_t)-g_t(x_t)}_{L^2(0,T;L^2(\partial\Omega))}^2+\frac{1}{\Abs{B_{r}^{\kappa}}T}\norm{v(x_t)-g(x_t)}_{L^2(0,T;H^1(\partial\Omega))}^2\\
            &=\frac{1}{\Abs{\Omega_{T}}}\norm{v_t(x_t)+\fL v(x_t)-f(x_t)}_{L^2(\Omega_{T})}^2+\frac{1}{\Abs{\Omega}}\norm{v(x,0)-\nu(x)}_{L^2(\Omega)}^2\\
            &~~~+\frac{1}{\Abs{\partial\Omega}T}\norm{v_t(x_t)-g_t(x_t)}_{L^2(0,T;L^2(\partial\Omega))}^2+\sum_{r=1}^{l}\frac{1}{\Abs{B_{r}^{\kappa}}T}\norm{{_{r}v(x_{r}',t)}-{_{r}g(x'_{r},t)}}_{L^2(0,T;H_{r}^{1}(B_{r}^{\kappa}))}^2,
        \end{aligned}
    \end{equation*}
    then there is a constant $C$ depending only on $\Omega_{T}$, $d$, $A(x_t)$, $\vb(x_t)$ and $c(x_t)$ such that for all $v\in L^2(0,T;H^2(\Omega))$, $v_t\in L^2(0,T;L^2(\Omega))$
    \begin{equation}
        \max_{[0,T]}\norm{u-v}_{L^2(\Omega)}^{2}+ \norm{u-v}_{L^2(0,T;H^1(\Omega))}^{2}+\norm{u_t-v_t}_{L^2(0,T;L^2(\Omega))}^{2}<C\fR_D(v).
    \end{equation}
\end{theorem}
\begin{proof}
     By letting $q(x_t):=f(x_t)-\fL_{t} v(x_t)-v_{t}(x_t)$ in $\Omega_{T}$, $h(x_t):=g(x_t)-v(x_t)$ on $\partial\Omega_{T}$, $p(x)=\nu(x)-v(x,0)$ in $\Omega$ and $w(x_t):=u(x_t)-v(x_t)$ in $\Omega_T$, we have
\begin{equation}
\left\{
\begin{aligned}
    w_t + \fL_t w &= q(x_t), 
    &&\quad x_t \in \Omega_T, \\
    w(x_t) &= h(x_t), 
    &&\quad x_t \in \partial \Omega_T, \\
    w(x, 0) &= p(x), 
    &&\quad x \in \Omega.
\end{aligned}
\right.
\end{equation}

    By Theorem~\ref{thm::ExtensionH1}, there is an extension $E: H^{1}(\partial\Omega)\to H^1(\Omega)$ such that $Eh\lfloor_{\partial\Omega\times \{t\}}=h(x_t)$ , $\norm{Eh}_{H^1(\Omega)}\le C\norm{h}_{H^1(\partial \Omega)}$ and $Eh_{t}\lfloor_{\partial\Omega\times \{t\}}=h_t(x_t)$ , $\norm{Eh_t}_{H^1(\Omega)}\le C\norm{h_t}_{H^1(\partial \Omega)}$ for all $t\in[0,T]$. Further defining $\tilde{w}=w-Eh$, then $\tilde{w}$ is the solution to 
\begin{equation}
\left\{
\begin{aligned}
    \tilde{w}_t + \fL_t \tilde{w} &= q(x_t) - \fL_t Eh(x_t) - Eh_t(x_t),
    &&\quad x_t \in \Omega_T, \\
    \tilde{w}(x_t) &= 0, 
    &&\quad x_t \in \partial \Omega_T, \\
    \tilde{w}(x, 0) &= p(x) - Eh(x,0), 
    &&\quad x \in \Omega.
\end{aligned}
\right.
\end{equation}

By proof of Theorem~\ref{thm::Stability4EllipticProblem}, one can directly obtain that
    \begin{equation}
        \frac{\diff}{\diff{t}}\norm{\tilde{w}}_{L^2(\Omega)}^2+\frac{\lambda}{2}\norm{\tilde{w}}_{H^1(\Omega)}^2\le C(\norm{h}_{H^1(\partial\Omega)}^2+\norm{q}_{L^2(\Omega)}^2+\norm{h_t}_{L^2(\Omega)}^2).
    \end{equation}
    
    Then by standard proof in~\cite{Evans2010Partial}, we can obtain the consequence.
\end{proof}
\section{Error Analysis: A Priori Estimates for \textbf{SSBE}}\label{sec::GeneralizationBound}
In this section, we establish a priori error bounds for the \textbf{SSBE} method when applied to elliptic and parabolic problems, with the aim of connecting the theoretical guarantees for the population loss with the practical performance governed by the empirical loss. While the convergence behavior of the algorithm has been illustrated in previous sections, a rigorous quantification of the \emph{generalization gap}—the discrepancy between the idealized population risk and its empirical counterpart—remains essential for understanding the robustness of the method under finite sampling. 

For the purpose of obtaining explicit and tractable error estimates, we restrict the PDE forcing term to lie in a Barron-type function space as mentioned in Section~\ref{sec::BarronSpace}, which allows us to leverage tools from statistical learning theory to derive bounds that depend polynomially on the sampling complexity. We emphasize that this regularity assumption is introduced solely for the sake of theoretical analysis; in practice, the \textbf{SSBE} algorithm is applicable beyond the Barron-type setting and remains effective for a wide range of PDE problems without this restriction.

We structure our analysis as follows:

\begin{itemize}
\item \textbf{Notations and Assumptions}: To begin our analysis, we first state the assumptions and introduce necessary notations.
\item \textbf{Generalization Error Estimates}: We separately provide explicit \textit{a posteriori} and \textit{a priori} generalization error bounds for both elliptic and parabolic problems.
\item \textbf{Technical Lemmas and Sketch of Proof}: We introduce the key technical lemmas and provide an outline for the proof, highlighting essential insights, together with a conceptual proof diagram as Figure~\ref{fig::Sketch}.
\end{itemize}

\paragraph{Notations and Assumptions}
To begin, we introduce necessary assumptions and notations used throughout this section:
\begin{itemize}
    \item Domain $\Omega \subseteq [0,1]^d$ and $\Omega_T \subseteq [0,1]^{d+1}$ are of class $\fC^{1,0}$.
    \item  Operators $\fL$ and $\fL_t$ are assumed in non-divergence form (Remark~\ref{rmk::NonDivergence}).
    \item Coefficients of PDEs are uniformly bounded by a constant $M$, i.e. $\norm{A(x)}_{L^{\infty}(\Omega)}$, $\norm{\hat{\vb}(x)}_{L^{\infty}(\Omega)}$, $\norm{c(x)}_{L^{\infty}(\Omega)}<M$, where $\norm{\cdot}_{L^{\infty}(\Omega)}$ for vector or matrix valued function means the entry-wise norm.
    \item The forcing terms satisfy $|f|<1$ in $\Omega$ or $\Omega_T$, and boundary/initial data satisfy $|g|,|\nabla g|,|\nu|<1$.
    \item Straighten functions $\{\gamma_r\}_{r=1}^{l}$ are bounded by a constant $\widetilde{M}$, i.e. $\Abs{\gamma_r(x_r')}<\widetilde{M}$ for all $x_r'\in B_r^{\kappa}$ and $r\in \{1,\ldots,l\}.$
    \item Neural networks are two-layer structures, i.e. $u(x,\theta)=\sum_{k=1}^{m}a_k\sigma(\vw^{\T}_kx)$ is a two-layer neural network with activation $\sigma(x)=\frac{1}{6}\operatorname{ReLU}{\color{red}^{3}}(x)=\frac{1}{6}\max^3\{x,0\}$ and $\theta=\operatorname{vec}\{a_k,\vw_k\}_{k=1}^{m}$. $u(x,t,\theta)=\sum_{k=1}^{m}a_k\sigma(\vw^{\T}_kx+w^t_kt)$ is a two-layer neural network with activation $\sigma(x)=\frac{1}{6}\operatorname{ReLU}{\color{red}^{3}}(x)=\frac{1}{6}\max^3\{x,0\}$ and $\theta=\operatorname{vec}\{a_k,\vw_k,w_k^t\}_{k=1}^{m}$.

\end{itemize}

\begin{figure}[htpb]
    \centering
    \includegraphics[scale=0.8]{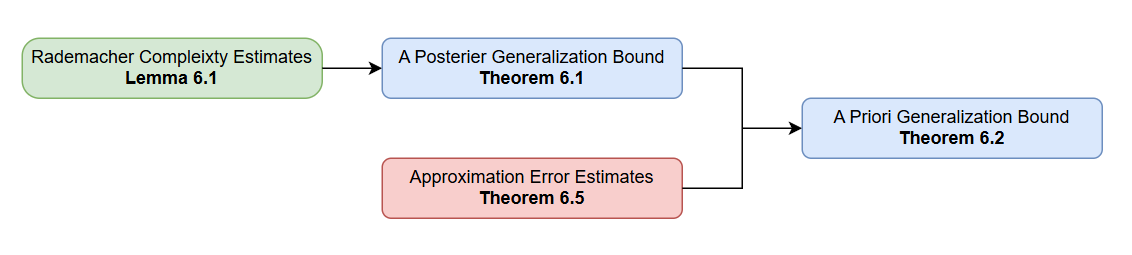}
    \caption{Conceptual Proof Diagram on Elliptic Case}
    \label{fig::Sketch}
\end{figure}

\noindent The \textbf{main consequences} for elliptic and parabolic problems are presented as follows (from Theorem~\ref{thm::APosterierGeneralizationBoundElliptic} to~\ref{thm::APrioriGeneralizationBoundParabolic}):

\paragraph{Generalization Error Estimate for Elliptic Problem}    
$\fR_{D}(\theta)$ is the population loss $\fR_D(v)$ defined in~\eqref{eq::AlgorithmRepre4EllipticPopulation} with objective function $u(x,\theta)$ and $\fR_{S}(\theta)$ is the empirical loss $\fR_{S}(v)$ defined in~\eqref{eq::AlgorithmRepre4EllipticEmpirical} with objective function $u(x,\theta)$ and random sampled set $$S=\{x^i\}_{i=1}^{n_{\Omega}}\bigcup \cup_{i=1}^{l}\{{x_r'}^i\}_{i=1}^{n_r}=:S_{\Omega}\bigcup \cup_{i=1}^{l}S_r.$$

\begin{theorem}[A posterier generalization bound]\label{thm::APosterierGeneralizationBoundElliptic}
For any $\delta\in (0,1)$, with probability at least $1-\delta$ over the choice of random sampled set $\{x^i\}_{i=1}^{n_{\Omega}}\bigcup \cup_{r=1}^{l}\{{x_{r}'}^{i}\}_{i=1}^{n_r}:=S$, for any two layer neural network $u(x,\theta)$, we have
\begin{equation}
    \begin{aligned}
        \Abs{\fR_{D}(\theta)-\fR_{S}(\theta)}&\le 2d\left(72M^2+32(\widetilde{M}+1)^2l\right)\left(d+\ln(\pi(\norm{\theta}_{\fP}+1))+\sqrt{\ln\left(l/\delta\right)}\right)\\
        &~~~\times\left(\frac{(\norm{\theta}_{\fP}+1)^2}{\sqrt{n_{\Omega}}}+\sum_{r=1}^{l}\frac{(\norm{\theta}_{\fP}+1)^2}{\sqrt{n_r}}\right).
    \end{aligned}
\end{equation}
\end{theorem}

\begin{theorem}[A priori generalization bound]\label{thm::APrioriGeneralizationBoundElliptic}
    Suppose that $(f,g)\in \fB(\Omega,\partial\Omega)$ and let
    \begin{equation}
        \theta_{S}=\arg\min_{\norm{\theta}_{\fP}<B}\fR_{S}(\theta),
    \end{equation}
    then with probability at least $1-\delta$ over the choice of random sampled set $\{x^i\}_{i=1}^{n_{\Omega}}\bigcup \cup_{r=1}^{l}\{{x_{r}'}^{i}\}_{i=1}^{n_r}:=S$, we have
    \begin{equation}
        \begin{aligned}
            \fR_{D}(\theta_{S})&\le  \frac{12M+3l d(\widetilde{M}+1)}{m}\norm{(f,g)}_{\fB}^2\\
            &~~~+2d\left(72M^2+32(\widetilde{M}+1)^2l+1\right)\left(d+\ln(\pi^2(4\norm{(f,g)}_{\fB}+2))+\sqrt{\ln\left(2l/\delta\right)}\right)\\
            &~~~\times\left(\frac{4\norm{(f,g)}_{\fB}^{2}+2}{\sqrt{n_{\Omega}}}+\sum_{r=1}^{l}\frac{4\norm{(f,g)}_{\fB}^{2}+2}{\sqrt{n_r}}\right).
        \end{aligned}
    \end{equation}
\end{theorem}
\paragraph{Generalization Error Estimate for Parabolic Problem}  
$\fR_{D}(\theta)$ is the population loss $\fR_D(v)$ defined in~\eqref{eq::AlgorithmRepre4ParabolicPopulation} with objective function $u(x,t,\theta)$ and $\fR_{S}(\theta)$ is the empirical loss $\fR_{S}(v)$ defined in~\eqref{eq::AlgorithmRepre4ParabolicEmpirical} with objective function $u(x,t,\theta)$ and random sampled set $$S=\{x^i_t\}_{i=1}^{n_{\Omega_{T}}}\bigcup \cup_{i=1}^{l}\{{x_r'}^i,t_j\}_{i=1,j=1}^{n_r,n_T}\bigcup\{x_{\partial\Omega}^{i},t_j\}_{i=1,j=1}^{n_{\partial\Omega},n_T}\bigcup \{x^i\}_{i=1}^{n_{\Omega}}=:S_{\Omega_T}\bigcup \cup_{i=1}^{l}S_r\bigcup S_{\partial\Omega}\bigcup S_{\Omega}.$$

\begin{theorem}[A posterier generalization bound]\label{thm::APosterierGeneralizationBoundParabolic}
For any $\delta\in (0,1)$, with probability at least $1-\delta$ over the choice of random sampled set $\{x^i_t\}_{i=1}^{n_{\Omega_{T}}}\bigcup \cup_{i=1}^{l}\{{x_r'}^i,t_j\}_{i=1,j=1}^{n_r,n_T}\bigcup\{x_{\partial\Omega}^{i},t_j\}_{i=1,j=1}^{n_{\partial\Omega},n_T}\bigcup \{x^i\}_{i=1}^{n_{\Omega}}:=S$, for any two layer neural network $u(x,t,\theta)$, we have
\begin{equation}
    \begin{aligned}
        \Abs{\fR_{D}(\theta)-\fR_{S}(\theta)}&\le 2d\left(72M^2+32(\widetilde{M}+1)^2l+64\right)\left(d+\ln(\pi(\norm{\theta}_{\fP}+1))+\sqrt{\ln\left(l/\delta\right)}\right)\\
        &~~~\times\left(\frac{1}{\sqrt{n_{\Omega_T}}}+\sum_{r=1}^{l}\frac{1}{\sqrt{n_rn_T}}+\frac{1}{\sqrt{n_{\partial\Omega}n_T}}+\sum_{r=1}^{l}\frac{1}{\sqrt{n_{\Omega}}}\right)(\norm{\theta}_{\fP}+1)^2.
    \end{aligned}
\end{equation}
\end{theorem}

\begin{theorem}[A priori generalization bound]\label{thm::APrioriGeneralizationBoundParabolic}
    Suppose that $(f,g,\nu)\in \fB(\Omega_{T},\partial\Omega_{T},\Omega)$ and let
    \begin{equation}
        \theta_{S}=\arg\min_{\norm{\theta}_{\fP}<B}\fR_{S}(\theta),
    \end{equation}
    then with probability at least $1-\delta$ over the choice of random sampled set $$\{x^i_t\}_{i=1}^{n_{\Omega_{T}}}\bigcup \cup_{r=1}^{l}\{({x_{r}'}^{i},t_j)\}_{i=1,j=1}^{n_r,n_{T}}\bigcup \{x^{i}\}_{i=1}^{n_{\Omega}}:=S,$$ we have
    \begin{equation}
        \begin{aligned}
            \fR_{D}(\theta_{S})&\le  \frac{12(M+1)+3l d(\widetilde{M}+1)+6}{m}\norm{(f,g,\nu)}_{\fB}^2\\
            &~~~+2d\pi\left(72(M+1)^2+32(\widetilde{M}+1)^2l+65\right)\left(d+1+\sqrt{\ln\left(2l/\delta\right)}\right)(2\norm{(f,g)}_{\fB}+2+B)^3\\
            &~~~\times\left(\frac{1}{\sqrt{n_{\Omega_{T}}}}+\sum_{r=1}^{l}\frac{1}{\sqrt{n_rn_T}}+\frac{1}{\sqrt{n_{\partial\Omega}n_T}}+\frac{1}{\sqrt{n_{\Omega}}}\right).
        \end{aligned}
    \end{equation}
\end{theorem}
\noindent \textbf{Technical Lemma and Sketch of Proof}

Given the technical complexity of the proofs, we first provide formal statements of the key technical lemmas (Lemma~\ref{lem::RademacherComplexity4BVP}, Theorem~\ref{thm::ApproximationError}) followed by a high-level sketch elucidating the core ideas. The conceptual proof diagram of the theorem proving path can be seen in Figure~\ref{fig::Sketch}. Complete technical details for the elliptic case are rigorously developed in Appendix~\ref{sec::Detail4GeneralzationError}, where the arguments rely on the interplay between functional analytic estimates and statistical learning theory. For parabolic equations, analogous generalization bounds can be derived using identical methodology: the Rademacher complexity analysis  remains unchanged, while the approximation error bounds adapt naturally to the time-evolved Sobolev spaces via standard parabolic regularity theory. To avoid redundancy, we omit repetitive derivations for parabolic problems and focus on the elliptic case as a representative framework.

In the following, we define several function spaces for $r\in \{1,\ldots,l\}$ and $\alpha\in \{1,\ldots,d-1\}$
\begin{equation}\label{eq::Def4ClassicFunctionSpace}
    \begin{aligned}
        \fF_{Q}&:=\left\{f(x,\theta):=\sum_{k=1}^{m}a_k\left[\vw_k^{\T}A(x)\vw_k\sigma''(\vw_k^{\T}x)+\hat{\vb}^{\T}(x)\vw_k\sigma'(\vw_k^{\T}x)+c(x)\sigma(\vw_k^{\T}x)\right]\ \big|\ \norm{\theta}_{\fP}<Q,\ x\in \Omega \right\}\\
        _{r}\fG_{Q}&:=\left\{g(x'_r,\theta):=\sum_{k=1}^{m}a_k\sigma( {_{r}\vw_{k}^{\T}}(x_r-\fB_{r}))\ \big|\ \norm{\theta}_{\fP}<Q,\ _{r}\vw_{k}^{\T}=\vw_{k}^{\T}\fO_{r}^{-1},\ x_{r_d}=\gamma_r(x_r'),\ x_r'\in B(0,\alpha)\right\}\\
        D_{x_{r_{\alpha}}} {_{r}\fG_{Q}}&:=\left\{D_{x_{r_{\alpha}}}g(x_r',\theta):=\sum_{k=1}^{m}a_k\left(_{r}\vw_{k,\alpha}+{_{r}\vw_{k,d}}D_{x_{r_{\alpha}}}\gamma_r(x_r')\right)\sigma'( {_{r}\vw_{k}^{\T}}(x_r-\fB_{r}))\ \big|\ \norm{\theta}_{\fP}<Q\right\}
        \end{aligned}
\end{equation}

\begin{lemma}[Rademacher complexity estimates]\label{lem::RademacherComplexity4BVP}
    We provide the Rademacher complexity of $\fF_{Q}$, $ _{r}\fG_{Q}$ and $D_{x_{r_{\alpha}}} {_{r}\fG_{Q}}$ respectively.
    \begin{enumerate}
        \item The Rademacher complexity of $\fF_{Q}$ over a set of $n_{\Omega}$ samples of $\Omega$, denoted as $S_{\Omega}=\{x^{i}\}_{i=1}^{n_{\Omega}}$, has an upper bound 
        \begin{equation}
            \operatorname{Rad}(\fF_{Q})\le \frac{4MQd^2}{\sqrt{n_{\Omega}}}.
        \end{equation}
        \item The Rademacher complexity of $ _{r}\fG_{Q}$ and $D_{x_{r_{\alpha}}} {_{r}\fG_{Q}}$ over a set of $n_r$ samples of $B_{r}^{\kappa}$, denoted as $S_r=\{{x_{r}'}^{i}\}_{i=1}^{n_r}$, has an upper bound
        \begin{equation}
            \begin{aligned}
                 \operatorname{Rad}({_{r}\fG_{Q}})&\le \frac{1}{3}\frac{Q}{\sqrt{n_r}},\\
                 \operatorname{Rad}(D_{x_{r_{\alpha}}} {_{r}\fG_{Q}})&\le \frac{(\widetilde{M}+1)Q}{\sqrt{n_r}},
            \end{aligned}
        \end{equation}
        for $r\in \{1,\ldots,l\}$ and $\alpha\in \{1,\ldots, d-1\}$.
    \end{enumerate}
\end{lemma}

\begin{theorem}[Approximation error estimates]\label{thm::ApproximationError}
    For any $(f,g)\in \fB(\Omega,\partial\Omega)$, there is a two layer neural network parameter $\tilde{\theta}=:\{\tilde{a}_k,\tilde{\vw}_k\}_{k=1}^{m}$ such that
    \begin{equation}
        \fR_{D}(\tilde{\theta})\le \frac{12M+3ld(\widetilde{M}+1)(\widetilde{M}+1)}{m}\norm{(f,g)}_{\fB}^2,
    \end{equation}
    with $\norm{\tilde{\theta}}_{\fP}\le 2\norm{(f,g)}_{\fB}$.
\end{theorem}

\noindent \textbf{Sketch of Proof}: The objective is to establish explicit bounds between the population loss and empirical loss for solutions of elliptic/parabolic PDEs under the \textbf{SSBE} framework.
\begin{itemize}
    \item \textbf{Control of Rademacher Complexity (Lemma~\ref{lem::RademacherComplexity4BVP})}
    
    Lemma~\ref{lem::RademacherComplexity4BVP} (Rademacher complexity estimates) indicates that if $(f,g)$ belongs to Barron pair, then the Rademacher complexity of the hypothesis class $\fF_Q$, ${_{r}\fG_{Q}}$, $D_{x_{\alpha}}{_{r}\fG_Q}$ satisfies
    \begin{equation*}
        \operatorname{Rad}(\cdot)\le Q \operatorname{O}(\frac{1}{\sqrt{n_{\Omega}}},\frac{1}{\sqrt{n_{r}}}).
    \end{equation*}
    This quantifies the complexity of the hypothesis class can be uniformly bounded by $\norm{\theta}_{\fP}$ and data size, which is essential for generalization estimates. 
    \item \textbf{Approximation Error Analysis (Theorem~\ref{thm::ApproximationError} )}

    Theorem~\ref{thm::ApproximationError} indicates that if $(f,g)$ belongs to Barron pair, for fixed network width $m$, there exists a two-layer network parameter $\tilde{\theta}$ such that 
    \begin{equation*}
        \fR_{D}(\tilde{\theta})\le \operatorname{O}(\frac{1}{m}),
    \end{equation*}
    demonstrates that the population loss can be made arbitrarily small by increasing $m$ (network capacity).
    \item \textbf{A Posteriori Bound (Theorem~\ref{thm::APosterierGeneralizationBoundElliptic})}
    \begin{enumerate}
        \item Apply the generalization inequality derived in~\cite{Shwartz2014Understanding}.
        \item Substitute the Rademacher bound from Lemma~\ref{lem::RademacherComplexity4BVP} into the generalization gap. Hence yields
    \end{enumerate}
    \begin{equation*}
         \Abs{\fR_{D}(\theta)-\fR_{S}(\theta)}\le \operatorname{O}(\frac{1}{\sqrt{n_{\Omega}}},\frac{1}{\sqrt{n_{r}}},\norm{\theta}_{\fP}) .
    \end{equation*}
    \item \textbf{A Priori Generalization Bound (Theorem~\ref{thm::APrioriGeneralizationBoundElliptic})}

    Decompose the population loss into approximation error and generalization error, and apply Theorem~\ref{thm::ApproximationError} and Theorem~\ref{thm::APosterierGeneralizationBoundElliptic}, yields
    \begin{equation*}
        \fR_{D}(\theta_S)\le \operatorname{O}(\frac{1}{\sqrt{n_{\Omega}}},\frac{1}{\sqrt{n_r}},\frac{1}{m}, B, \norm{(f,g)}_{\fB}).
    \end{equation*}
\end{itemize}
\section{Numerical Experiments}
\label{sec:example}

In this section, we present our proposed approach through a series of representative numerical examples. Specifically, we examine the Poisson equation, heat equation, nonlinear Poisson equation, and high-dimensional Poisson equation under various network architectures and parameter settings. \textcolor{black}{We further test SSBE on PDEs defined over irregular domains, where it consistently achieves lower relative $L^2$ and $H^1$ errors than standard PINNs, demonstrating its robustness to complex geometries.} The results show that our method achieves substantial improvements in both $L^2$ and $H^1$ convergence. Unless otherwise stated, we employ a standard fully connected neural network (FNN) architecture with layer sizes $[d,100,100,100,1]$ ($d$ is dimension) and the activation function $\tanh(x)$. All models are trained using the Adam optimizer, with the learning rate decaying from $1\times 10^{-3}$ to $1\times 10^{-6}$, and training continues until the loss converges. All experiments are conducted on a single RTX 3070Ti GPU, with identical training configurations for both the PINN and SSBE methods to ensure a fair comparison. Complete implementation details and parameter settings are available at   \url{https://github.com/CChenck/H1Boundary}.

\subsection{Poisson Equation}
Let $\Omega = B(0,1) \subset \mathbb{R}^2$ be the unit ball, $\mathcal{L} = -\Delta$ the Laplacian, and $f(x_1, x_2) = 4$, $g(x_1, x_2) = 0$. The PDE problem is then given by  
\begin{equation}\label{eq::Poi_Eqn}
\left\{
\begin{aligned}
    -\Delta u &= 4, 
    && (x_1, x_2) \in B(0,1), \\
    u &= 0, 
    && (x_1, x_2) \in \partial B(0,1).
\end{aligned}
\right.
\end{equation}

\begin{figure}[htpb!]
    \centering
    \includegraphics[scale=0.4]{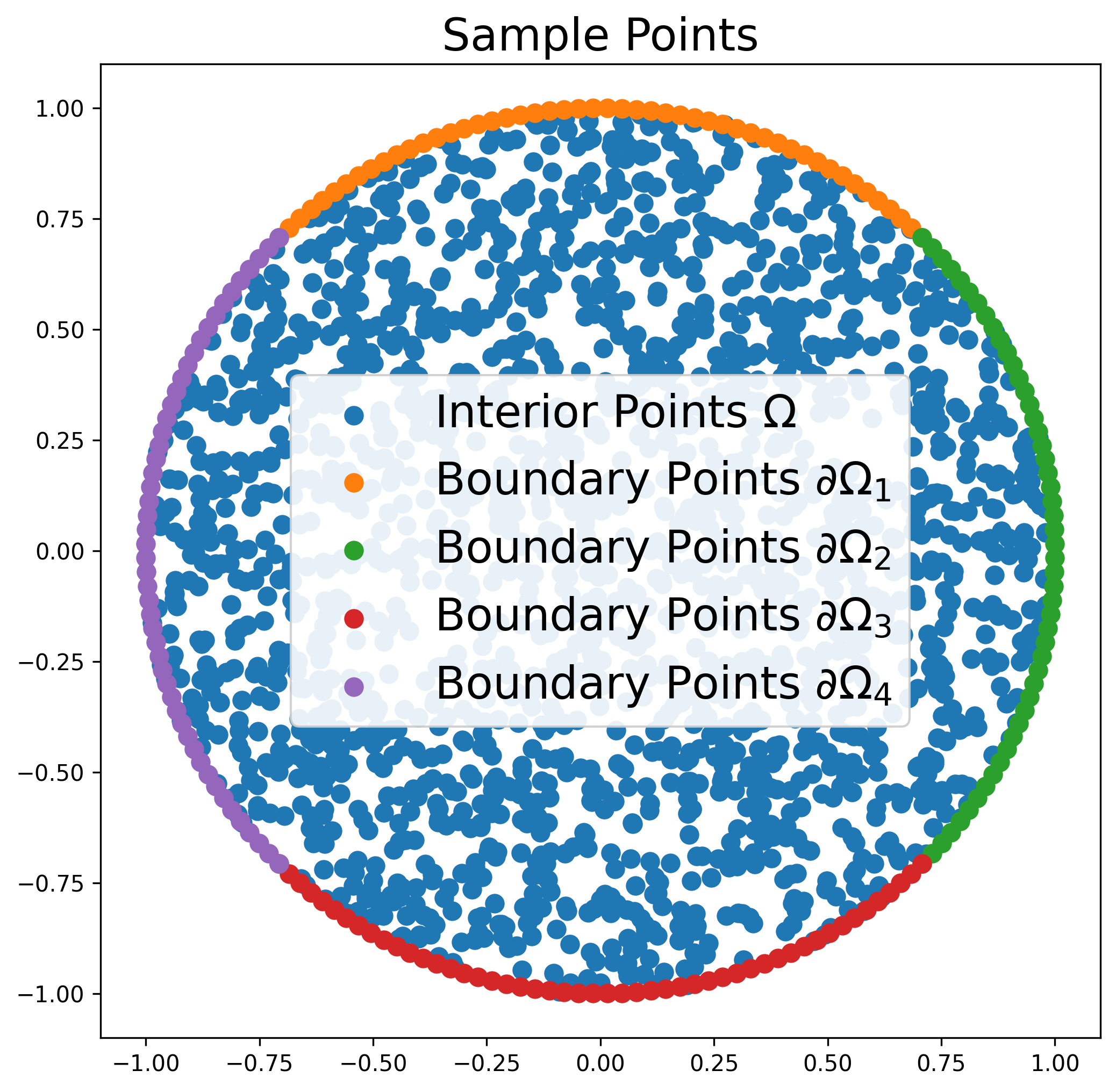}
    \caption{\textbf{Poisson equation:} Randomly sampled training points for the Poisson equation.}
    \label{Fig.Poi_sample}
\end{figure}

This PDE admits the unique ground truth $u(x_1, x_2) = 1 - \left( x_1^2 + x_2^2 \right)$. The randomly sampled training points are illustrated in Figure~\ref{Fig.Poi_sample}, with $2{,}500$ points sampled in the interior of $\Omega$ and $200$ points sampled on the boundary $\partial\Omega$, and the corresponding SSBE boundary loss is explicitly defined in Equation~\eqref{Eqn.Poi_HB}. To mitigate potential singularities when computing the SSBE boundary loss, the sampling points are partitioned into four subsets. To assess performance, we compare the proposed SSBE method with the standard PINN approach under various neural network architectures. The training curves of the relative $L^2$ and relative $H^1$ errors for one representative architecture are shown in Figure~\ref{Fig.Poi_curve}, while the quantitative results in Table~\ref{tab:poi_table} clearly indicate that SSBE consistently yields lower relative $L^2$ and $H^1$ errors than PINN across all tested configurations.
    \begin{equation}
    \begin{aligned}
      L_{HB}  =& \int_{\partial \Omega_1
        }\left(\frac{\diff}{\diff{x_1}}u_\theta(x_1,\sqrt{1-x_1^2})-\frac{\diff}{\diff{ x_1}}g(x_1,\sqrt{1-    x_1^2})\right)^2\sqrt{1+\left(\frac{\diff}{\diff{x_1}}\sqrt{1-x_1^2}\right)^2}\diff{x_1} \\
        +& \int_{\partial \Omega_2
        }\left(\frac{\diff}{\diff{x_2}}u_\theta(\sqrt{1-x_2^2},x_2)-\frac{\diff}{\diff{x_2}}g(\sqrt{1-x_2^2},x_2)\right)^2\sqrt{1+\left(\frac{\diff}{\diff{x_2}}\sqrt{1-x_2^2}\right)^2}\diff{x_2}  \\
        +&         \int_{\partial \Omega_3
        }\left(\frac{\diff}{\diff{x_1}}u_\theta(x_1,-\sqrt{1-x_1^2})-\frac{\diff}{\diff{x_1}}g(x_1,-\sqrt{1-    x_1^2})\right)^2\sqrt{1+\left(\frac{\diff}{\diff{x_1}}(-\sqrt{1-x_1^2})\right)^2}\diff{x_1} \\
      +& \int_{\partial \Omega_4
        }\left(\frac{\diff}{\diff{x_2}}u_\theta(-\sqrt{1-x_2^2},x_2)-\frac{\diff}{\diff{x_2}}g(-\sqrt{1-x_2^2},x_2)\right)^2\sqrt{1+\left(\frac{\diff}{\diff{x_2}}(-\sqrt{1-x_2^2})\right)^2}\diff{x_2}.
    \end{aligned}
    \label{Eqn.Poi_HB}
    \end{equation}

\begin{table}[h]
    \centering
    \begin{tabular}{|c|c|c|c|c|}
        \hline 
        Architecture & \multicolumn{2}{c|}{Orginal PINN} & \multicolumn{2}{c|}{SSBE} \\ 
        \hline 
        & Re $L^2$ & Re $H^1$ & Re $L^2$ & Re $H^1$ \\ 
        \hline 
        30 units / 3 hidden layers & $1.08\mathrm{e\text{-}04}$ & $4.05\mathrm{e\text{-}04}$ & $3.53\mathrm{e\text{-}06}$ & $1.81\mathrm{e\text{-}05}$ \\ 
        \hline 
        50 units / 3 hidden layers & $2.47\mathrm{e\text{-}05}$ & $8.75\mathrm{e\text{-}05}$ & $3.18\mathrm{e\text{-}06}$ & $1.09\mathrm{e\text{-}05}$ \\ 
        \hline 
        100 units / 3 hidden layers & $2.21\mathrm{e\text{-}05}$ & $8.19\mathrm{e\text{-}05}$ & $\boldsymbol{2.15\mathrm{e\text{-}06}}$ & $\boldsymbol{1.14\mathrm{e\text{-}05}}$ \\ 
        \hline 
        30 units / 5 hidden layers & $2.36\mathrm{e\text{-}04}$ & $7.56\mathrm{e\text{-}04}$ & $2.53\mathrm{e\text{-}05}$ & $9.92\mathrm{e\text{-}05}$ \\ 
        \hline 
        50 units / 5 hidden layers & $2.09\mathrm{e\text{-}04}$ & $6.73\mathrm{e\text{-}04}$ & $6.26\mathrm{e\text{-}06}$ & $2.53\mathrm{e\text{-}05}$ \\ 
        \hline 
        100 units / 5 hidden layers & $7.66\mathrm{e\text{-}05}$ & $2.78\mathrm{e\text{-}04}$ & $5.05\mathrm{e\text{-}06}$ & $2.33\mathrm{e\text{-}05}$ \\ 
        \hline 
        30 units / 7 hidden layers & $2.54\mathrm{e\text{-}04}$ & $8.16\mathrm{e\text{-}04}$ & $8.00\mathrm{e\text{-}06}$ & $2.99\mathrm{e\text{-}05}$ \\ 
        \hline 
        50 units / 7 hidden layers & $1.22\mathrm{e\text{-}04}$ & $3.98\mathrm{e\text{-}04}$ & $9.71\mathrm{e\text{-}06}$ & $3.87\mathrm{e\text{-}05}$ \\ 
        \hline 
        100 units / 7 hidden layers & $1.18\mathrm{e\text{-}04}$ & $3.47\mathrm{e\text{-}04}$ & $7.94\mathrm{e\text{-}06}$ & $3.33\mathrm{e\text{-}05}$ \\ 
        \hline 
    \end{tabular}
    \caption{\textbf{Poisson equation:} Relative $L^2$ and $H^1$ errors between the predicted and exact solutions $u(x_1,x_2)$ for the original PINN and the proposed SSBE method, evaluated across different neural network architectures obtained by varying the number of hidden layers and the number of neurons per layer.
}
    \label{tab:poi_table}
\end{table}

\begin{figure}[htpb]
    \centering
    \includegraphics[scale=0.35]{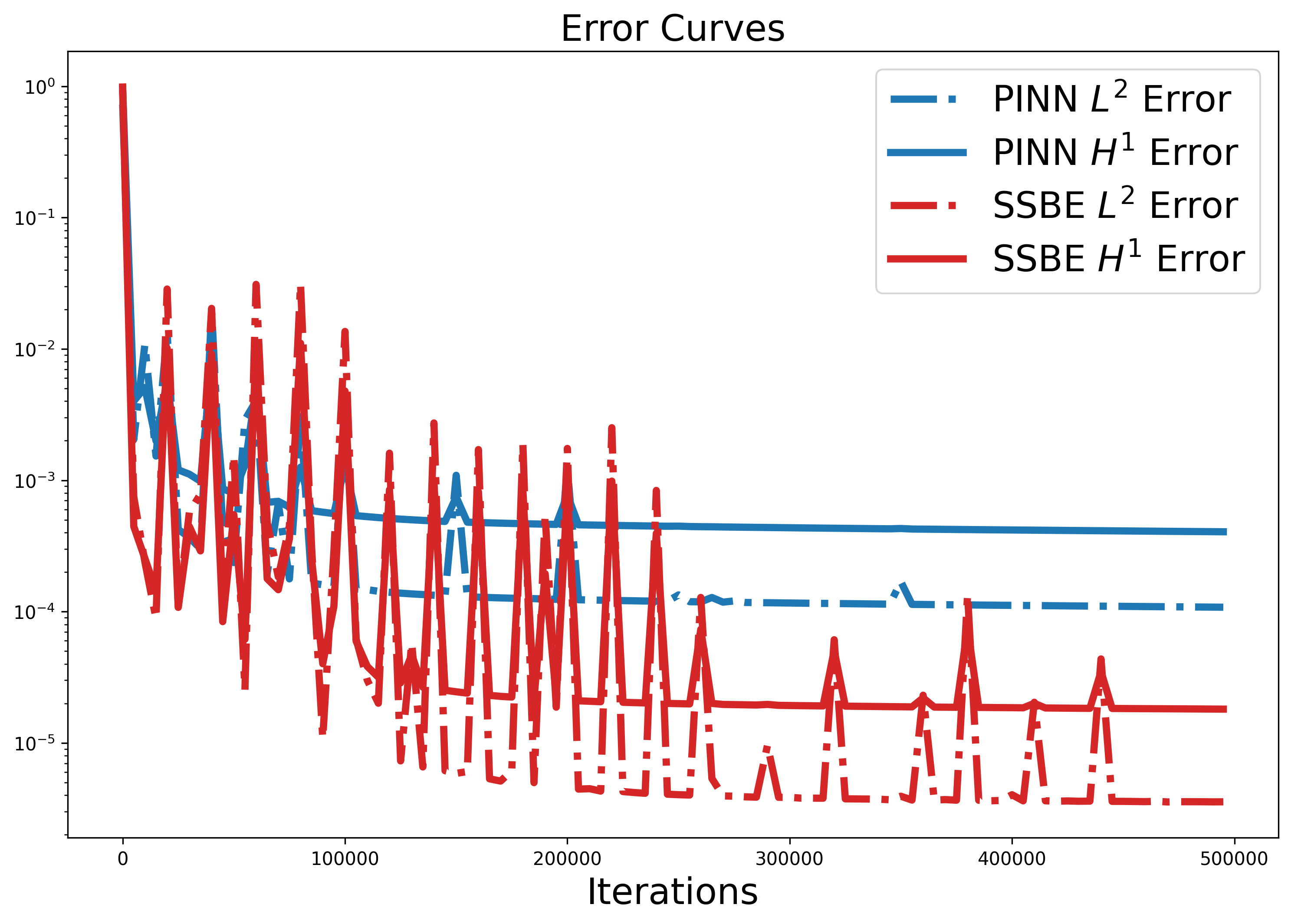}
    \caption{\textbf{Poisson equation:} Error curves comparing PINN and SSBE.
}
    \label{Fig.Poi_curve}
\end{figure}

\subsection{Heat Equation}
Let $\Omega = [-1,1] \times [-1,1] \subset \mathbb{R}^2$ be the square domain, $\mathcal{L} = -\Delta$ the Laplacian, and $f(x_1, x_2) = g(x_1, x_2) = 0$. The PDE problem is given by  
\begin{equation}\label{Poi_eqn}
\left\{
\begin{aligned}
    \frac{\partial u}{\partial t} - \Delta u &= 0, 
    && (x_1, x_2) \in \Omega, \quad t \in [0,1], \\
    u(x_1, x_2, t) &= 0, 
    && (x_1, x_2) \in \partial \Omega, \quad t \in [0,1], \\
    u(x_1, x_2, 0) &= \sin(\pi x_1) \sin(\pi x_2), 
    && (x_1, x_2) \in \Omega.
\end{aligned}
\right.
\end{equation}
This PDE admits the unique ground truth $u(x_1, x_2, t) = \sin(\pi x_1)\sin(\pi x_2)e^{-2\pi^2 t}$. In this experiment, the loss function is defined as  
\begin{equation}
    \text{Loss} = \lambda_1 L_{\text{res}} + \lambda_2 L_{\text{ini}} + \lambda_3 L_{LB} + \lambda_4 L_{HB},
\end{equation}
where $L_{\text{res}}$ denotes the PDE residual loss, $L_{\text{ini}}$ the initial condition loss, $L_{LB}$ the $L^2$-based boundary loss, and $L_{HB}$ the proposed SSBE boundary loss defined in Equation~\eqref{Eqn.Heat_HB}. The sampling points are illustrated in Figure~\ref{Fig.Heat_sample}. The sampling strategy is as follows: $5{,}000$ points are randomly sampled in the interior of $\Omega$, $2{,}000$ points are randomly sampled on the boundary $\partial\Omega$, and an additional $2{,}000$ points are used for the initial condition. Numerical experiments are conducted under various choices of the parameters $\lambda_i$, with the corresponding results summarized in Table~\ref{tab:heat_eqn}. Across all tested settings, SSBE consistently yields lower relative $L^2$ and $H^1$ errors than the original PINN, indicating improved accuracy and enhanced convergence. Figure~\ref{Fig.Heat_com} presents the exact solution, the PINN solution, and the SSBE solution at $t = 0$, $0.1$, and $0.3$ for a representative setting, while Figure~\ref{Fig.Heat_error} shows the corresponding pointwise errors, where SSBE exhibits uniformly lower errors across all time steps.

\begin{figure}[htpb]
    \centering
    \includegraphics[scale=0.4]{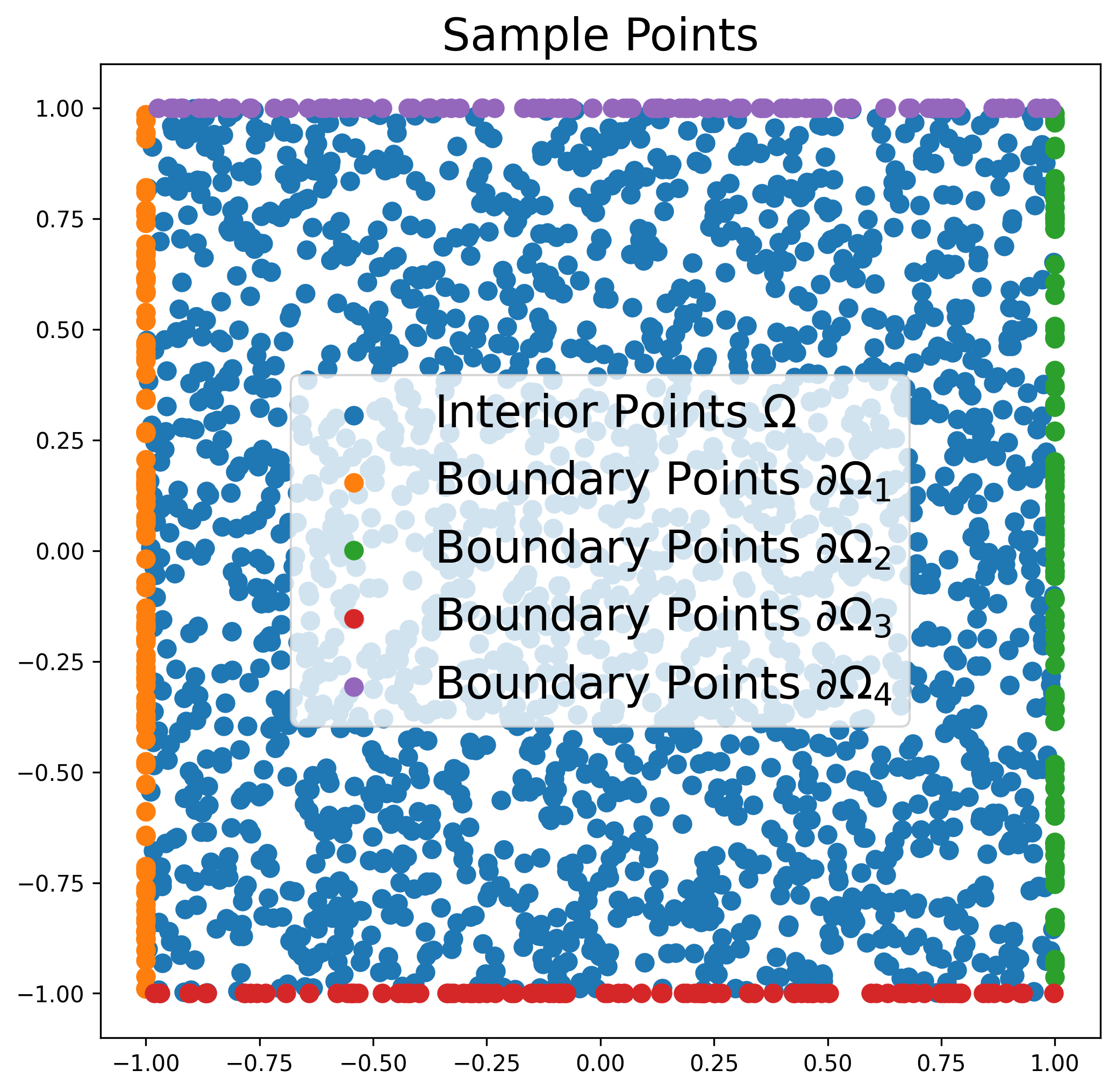}
    \caption{\textbf{Heat equation:} Randomly sampled training points for the heat equation.}
    \label{Fig.Heat_sample}
\end{figure}

    \begin{equation}
    \begin{aligned}
      L_{HB}  =& \int_{\partial \Omega_1
        }\left(\frac{\diff}{\diff{x_1}}u_\theta(-1,x_2)-\frac{\diff}{\diff{ x_1}}g(-1,x_2)\right)^2\diff{x_2}
        + \int_{\partial \Omega_2
        }\left(\frac{\diff}{\diff{x_1}}u_\theta(1,x_2)-\frac{\diff}{\diff{x_1}}g(1,x_2)\right)^2\diff{x_2}  \\
        +&         \int_{\partial \Omega_3
        }\left(\frac{\diff}{\diff{x_2}}u_\theta(x_1,-1)-\frac{\diff}{\diff{x_2}}g(x_1,-1)\right)^2\diff{x_1} 
      + \int_{\partial \Omega_4
        }\left(\frac{\diff}{\diff{x_2}}u_\theta(x_1,1)-\frac{\diff}{\diff{x_2}}g(x_1,1)\right)^2\diff{x_1}.
    \end{aligned}
    \label{Eqn.Heat_HB}
    \end{equation}

\begin{table}[htpb]
    \centering
    \begin{tabular}{|c|c|c|c|c|c|c|c|}
        \hline 
        \multicolumn{4}{|c|}{Parameters} & \multicolumn{2}{c|}{Original PINN} & \multicolumn{2}{c|}{SSBE} \\ 
        \hline
        $\lambda_1$ & $\lambda_2$ & $\lambda_3$ & $\lambda_4$ & Re $L^2$ & Re $H^1$ & Re $L^2$ & Re $H^1$ \\ 
        \hline 
        1 & 1 & 1 & 1 & $5.98 \mathrm{e\text{-}03}$ & $1.50 \mathrm{e\text{-}02}$ & $1.98 \mathrm{e\text{-}03}$ & $6.80 \mathrm{e\text{-}03}$ \\ 
        \hline 
        1 & 1 & 10 & 1 & $4.59 \mathrm{e\text{-}03}$ & $1.50 \mathrm{e\text{-}02}$ & $1.98 \mathrm{e\text{-}03}$ & $6.80 \mathrm{e\text{-}03}$ \\ 
        \hline 
        1 & 1 & 0.1 & 1 & $2.31\mathrm{e\text{-}02}$ & $2.42\mathrm{e\text{-}02}$ & $2.23\mathrm{e\text{-}03}$ & $5.10\mathrm{e\text{-}03}$ \\ 
        \hline 
        1 & 10 & 10 & 1 & $3.81\mathrm{e\text{-}03}$ & $9.70\mathrm{e\text{-}03}$ & $2.71\mathrm{e\text{-}03}$ & $6.50\mathrm{e\text{-}03}$ \\ 
        \hline 
        1 & 0.1 & 0.1 & 1 & $3.68 \mathrm{e\text{-}02}$ & $4.96 \mathrm{e\text{-}02}$ & $3.23 \mathrm{e\text{-}03}$ & $5.50 \mathrm{e\text{-}03}$ \\ 
        \hline
        1 & 10 & 10 & 0.1 & - & - & $1.44\mathrm{e\text{-}03}$ & $5.00\mathrm{e\text{-}03}$ \\ 
        \hline 
        1 & 10 & 10 & 0.01 & - & - & $\boldsymbol{1.10\mathrm{e\text{-}03}}$ & $\boldsymbol{4.40\mathrm{e\text{-}03}}$ \\ 
        \hline
        1 & 10 & 10 & 0.001 & - & - & $2.23\mathrm{e\text{-}03}$ & $5.70\mathrm{e\text{-}03}$ \\ 
        \hline
    \end{tabular}
    \caption{\textbf{Heat equation:} Comparison between the original PINN ($\lambda_4 = 0$) and the proposed SSBE method under various $\lambda_i$ parameter settings. Test points are taken at 20 uniformly spaced times over $t \in [0,1]$, with $50\times 50$ uniformly sampled spatial points in $(x,y)$ at each time instance.
}
    \label{tab:heat_eqn}
\end{table}

\begin{figure}[!ht]
    \centering
    \setcounter {subfigure} 0(a.1){
    \includegraphics[scale=0.28]{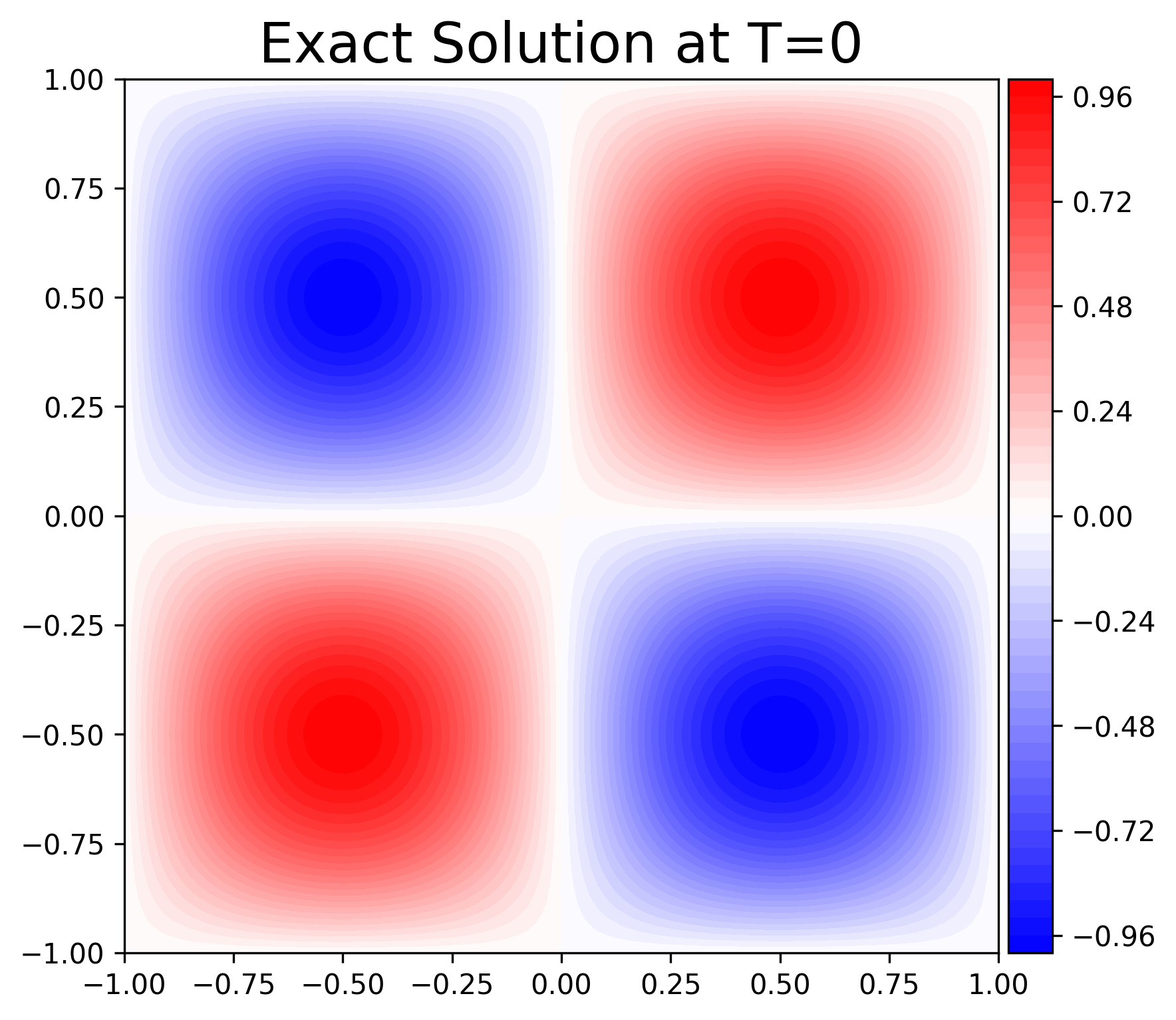}}
    \setcounter {subfigure} 0(a.2){
    \includegraphics[scale=0.28]{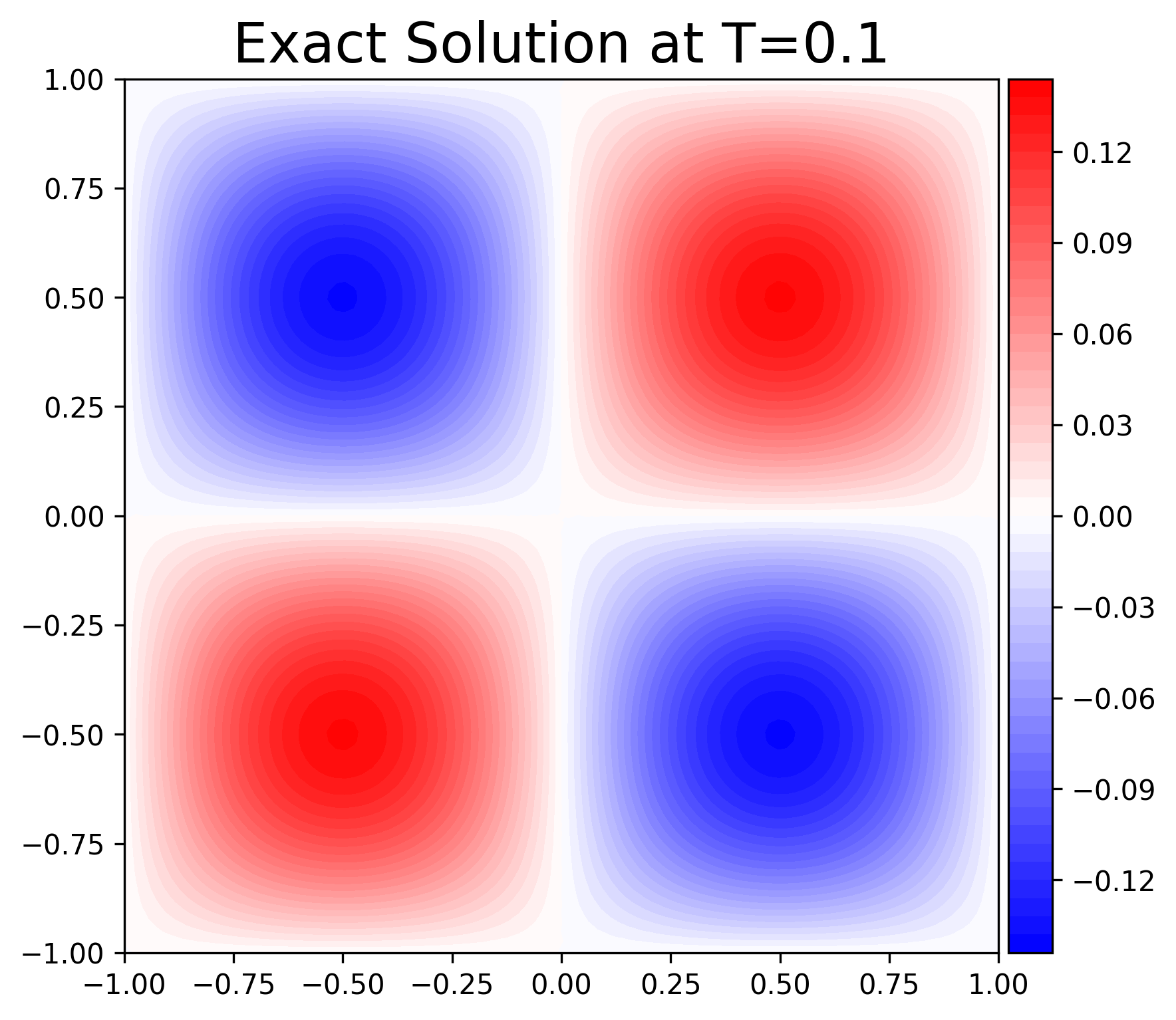}}
    \setcounter {subfigure} 0(a.3){
    \includegraphics[scale=0.28]{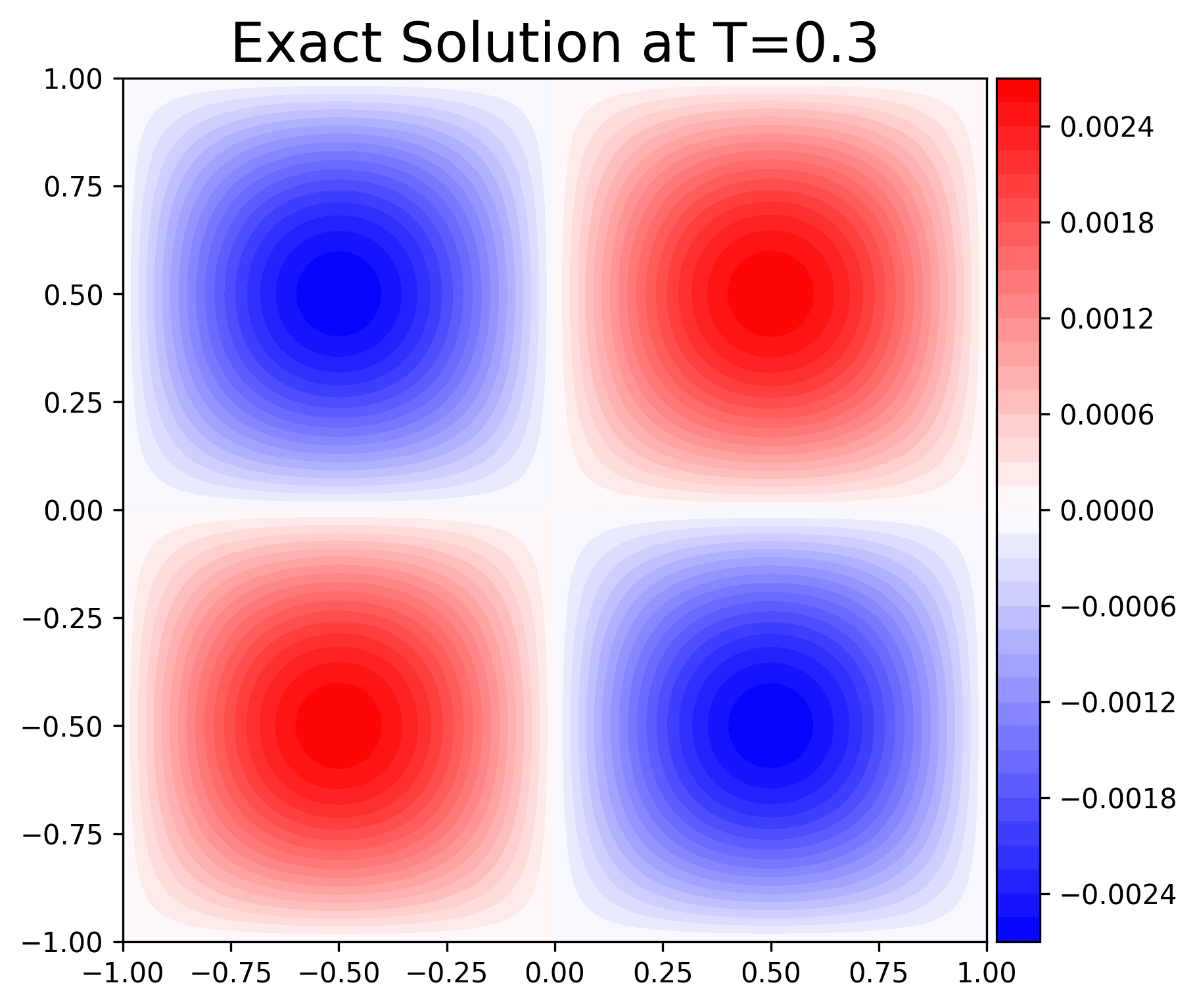}}\\
    \setcounter {subfigure} 0(b.1){
    \includegraphics[scale=0.28]{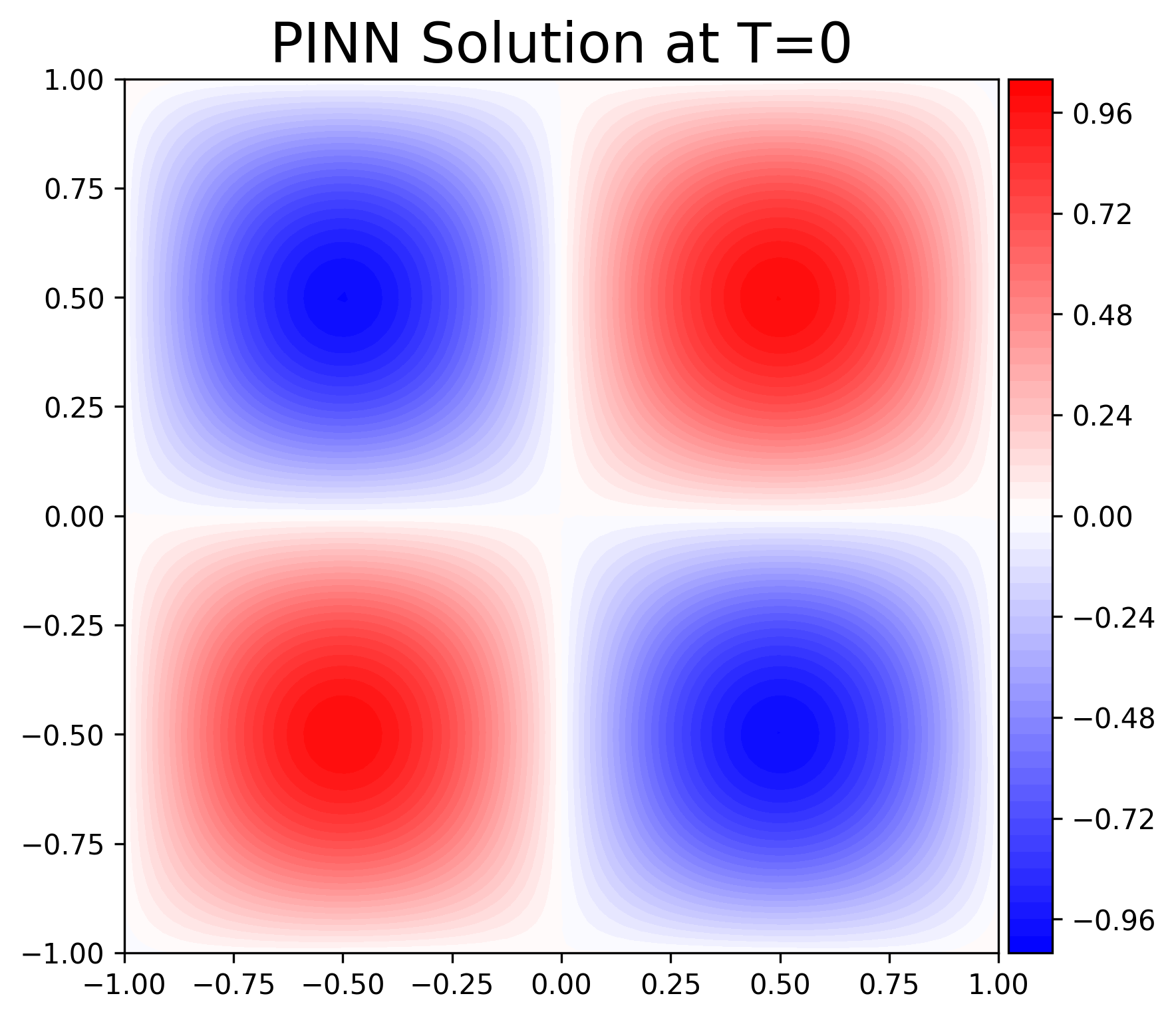}}
    \setcounter {subfigure} 0(b.2){
    \includegraphics[scale=0.28]{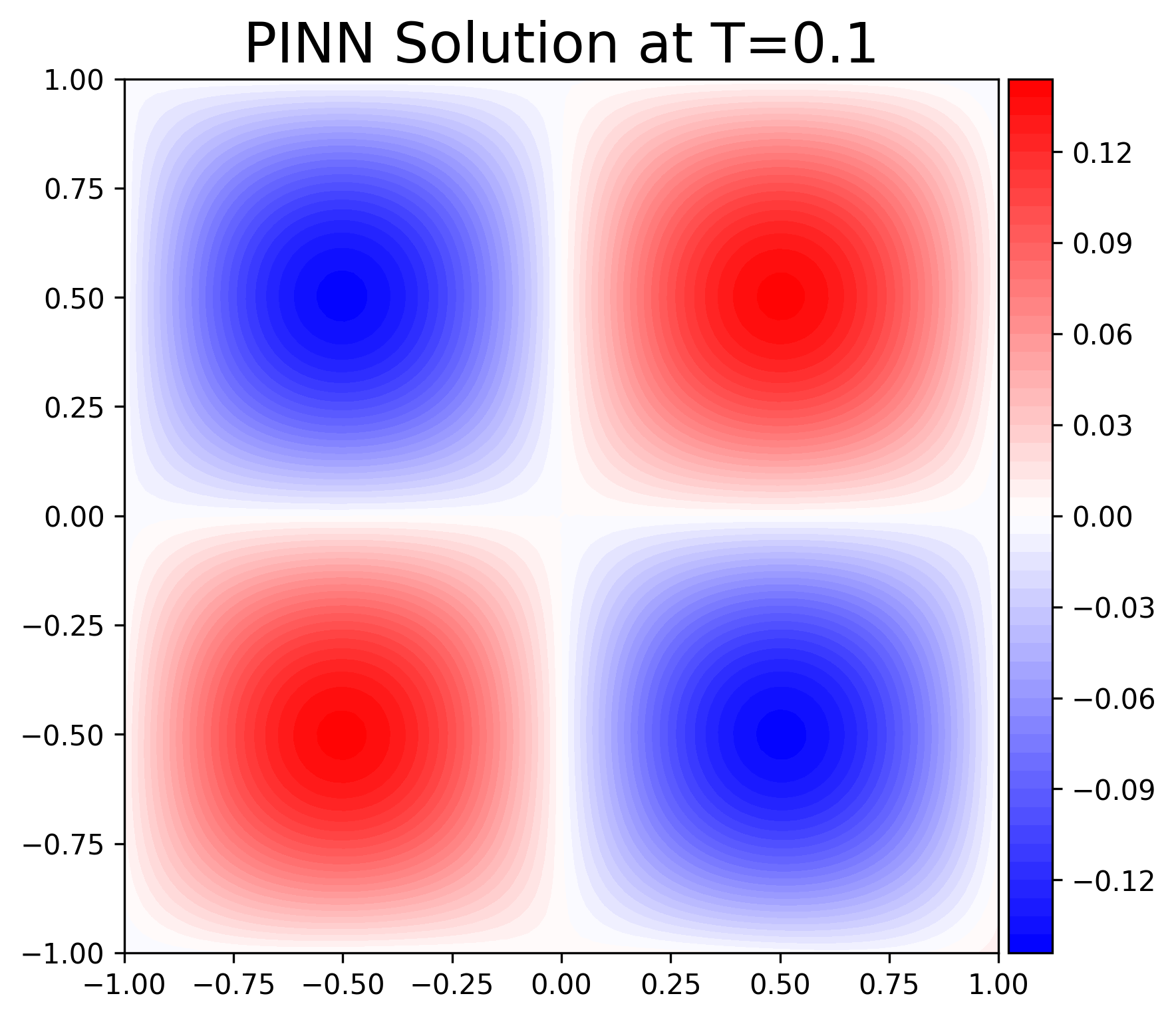}}
    \setcounter {subfigure} 0(b.3){
    \includegraphics[scale=0.28]{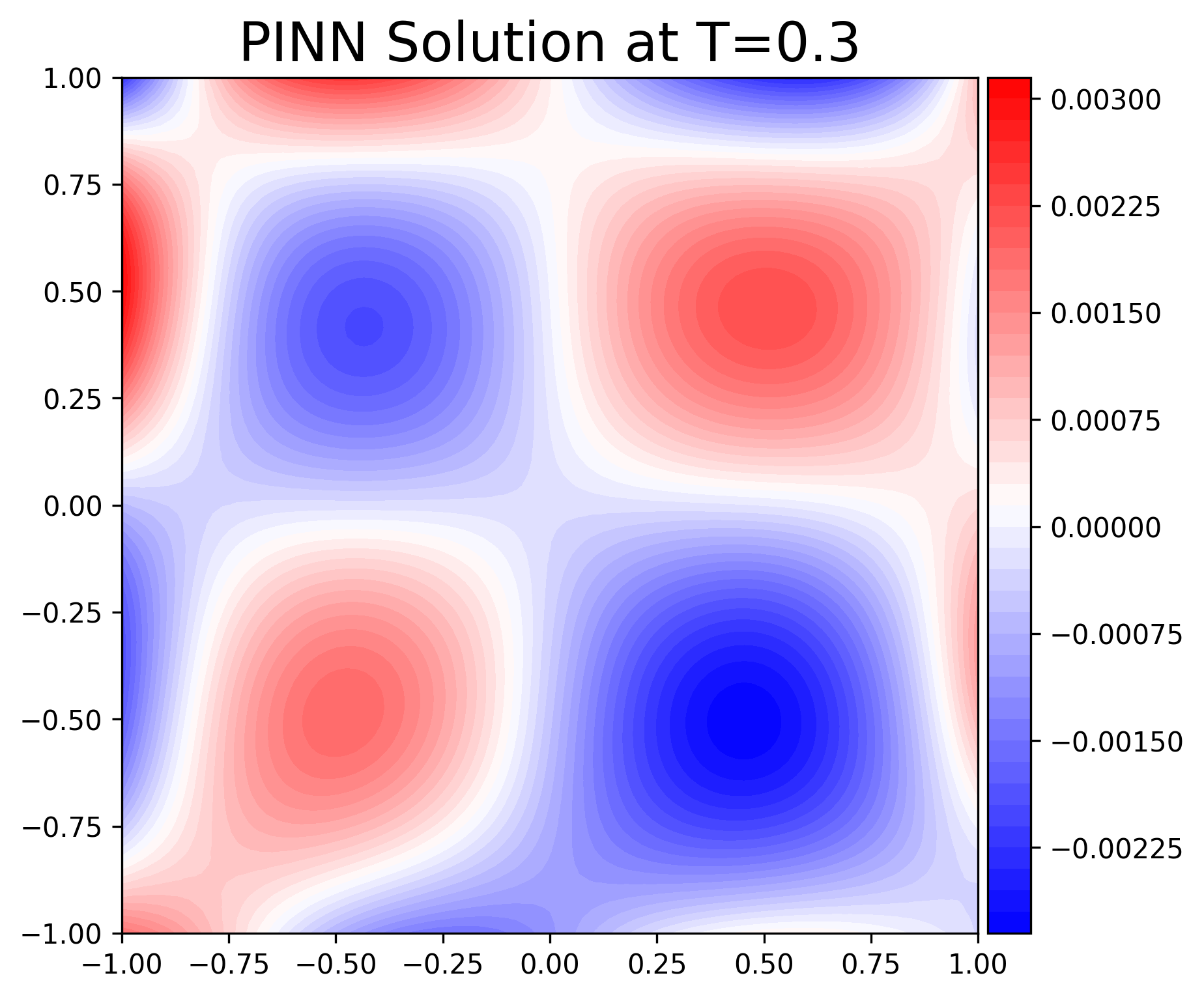}} \\
        \setcounter {subfigure} 0(c.1){
    \includegraphics[scale=0.28]{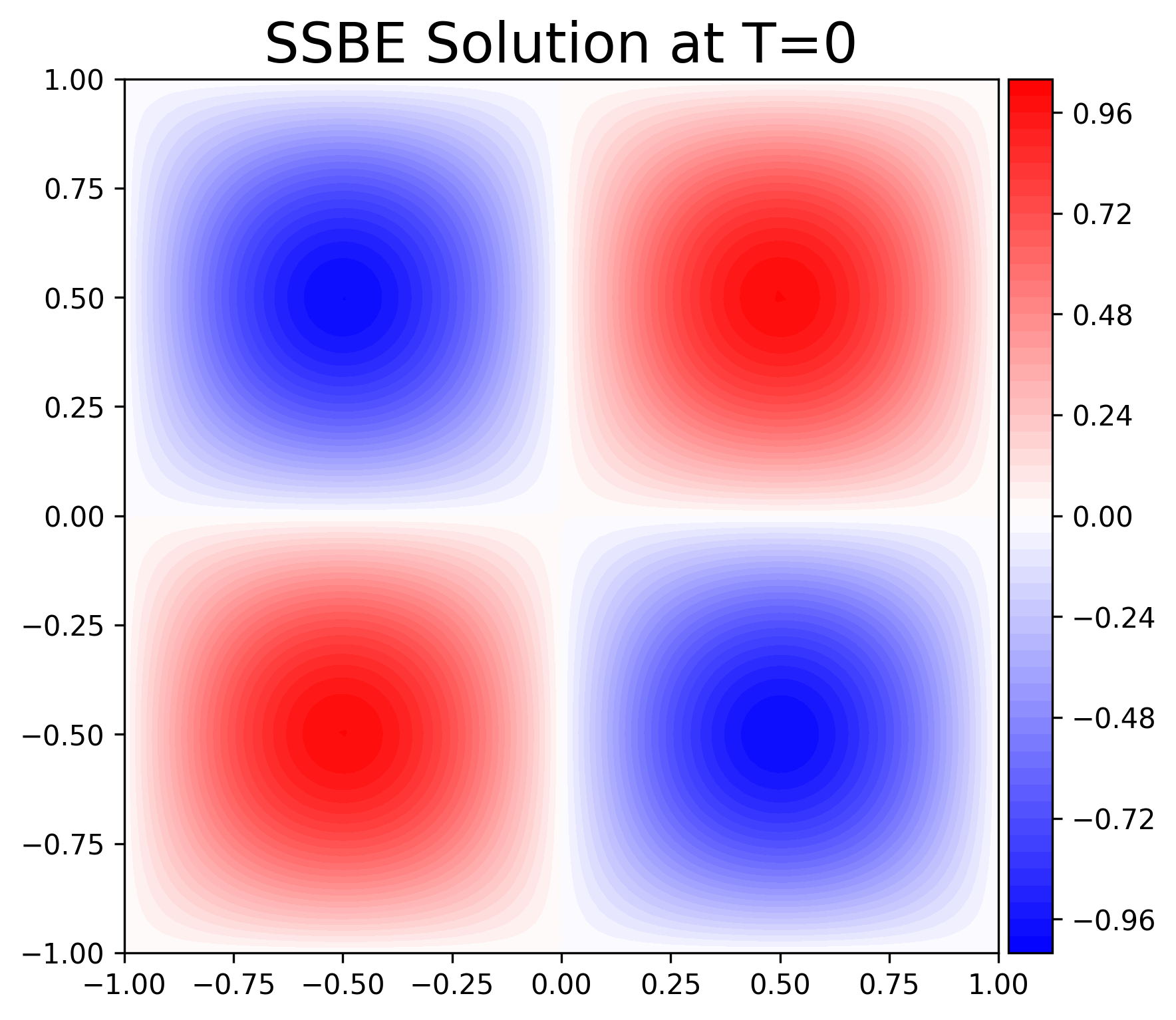}}
    \setcounter {subfigure} 0(c.2){
    \includegraphics[scale=0.28]{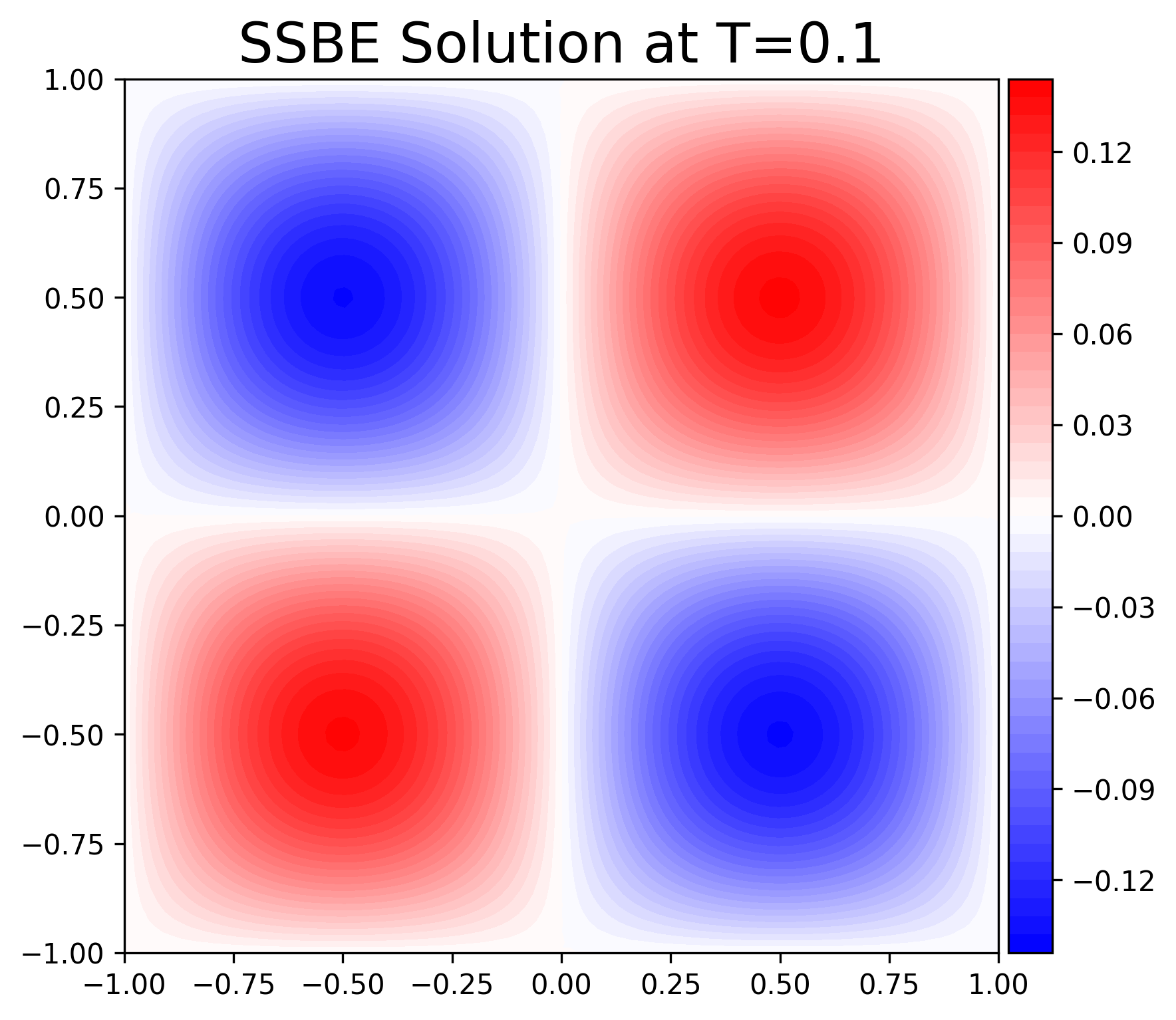}}
    \setcounter {subfigure} 0(c.3){
    \includegraphics[scale=0.28]{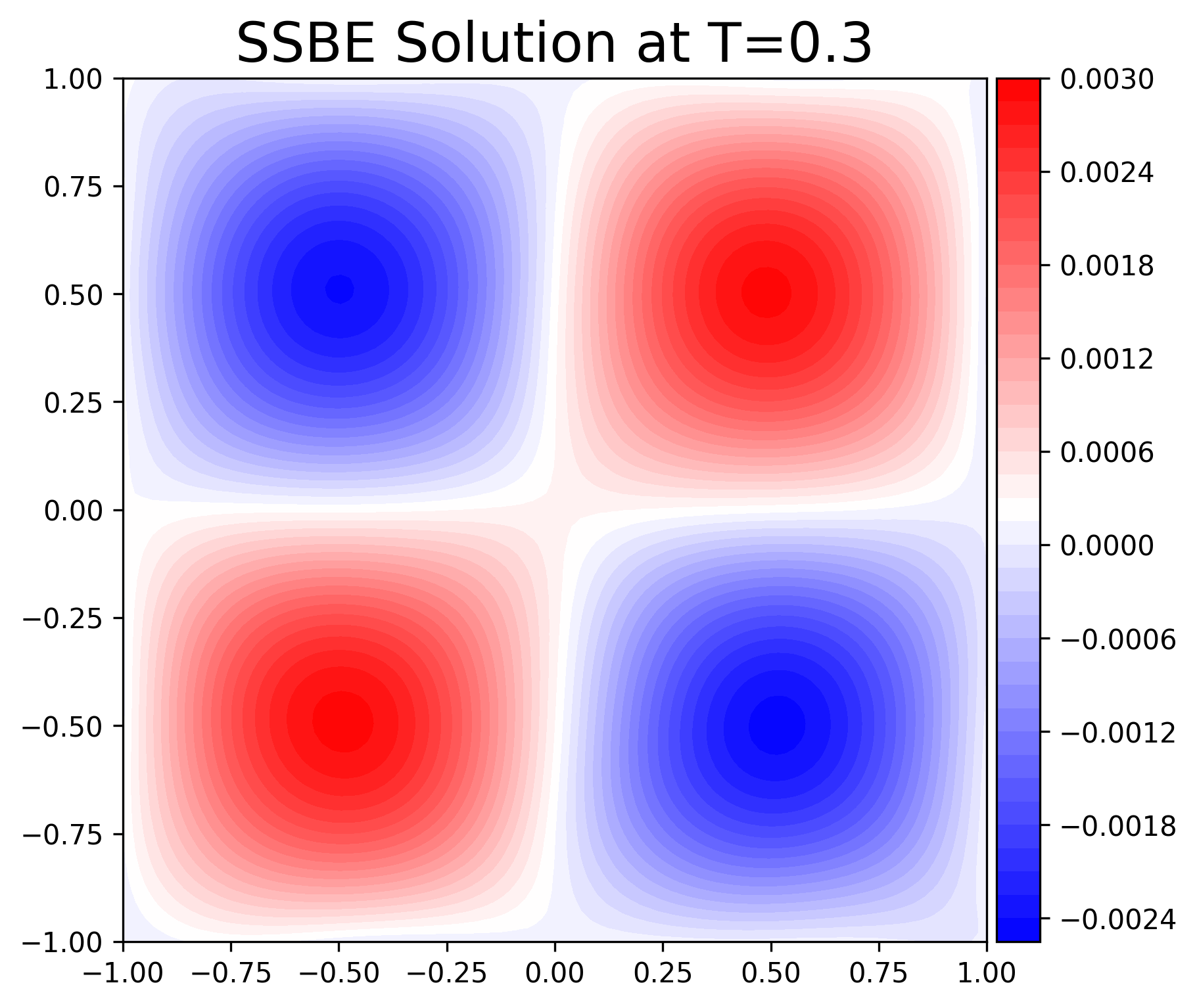}} 
    
    \caption{\textbf{Heat equation:} First row: exact solution; second row: prediction of the PINN model; third row: prediction of the SSBE model. Results are shown for the setting $\lambda_1 = \lambda_2 = \lambda_3 = \lambda_4 = 1$.
}
    \label{Fig.Heat_com}
\end{figure}

\begin{figure}[!ht]
    \centering
    \setcounter {subfigure} 0(a.1){
    \includegraphics[scale=0.27]{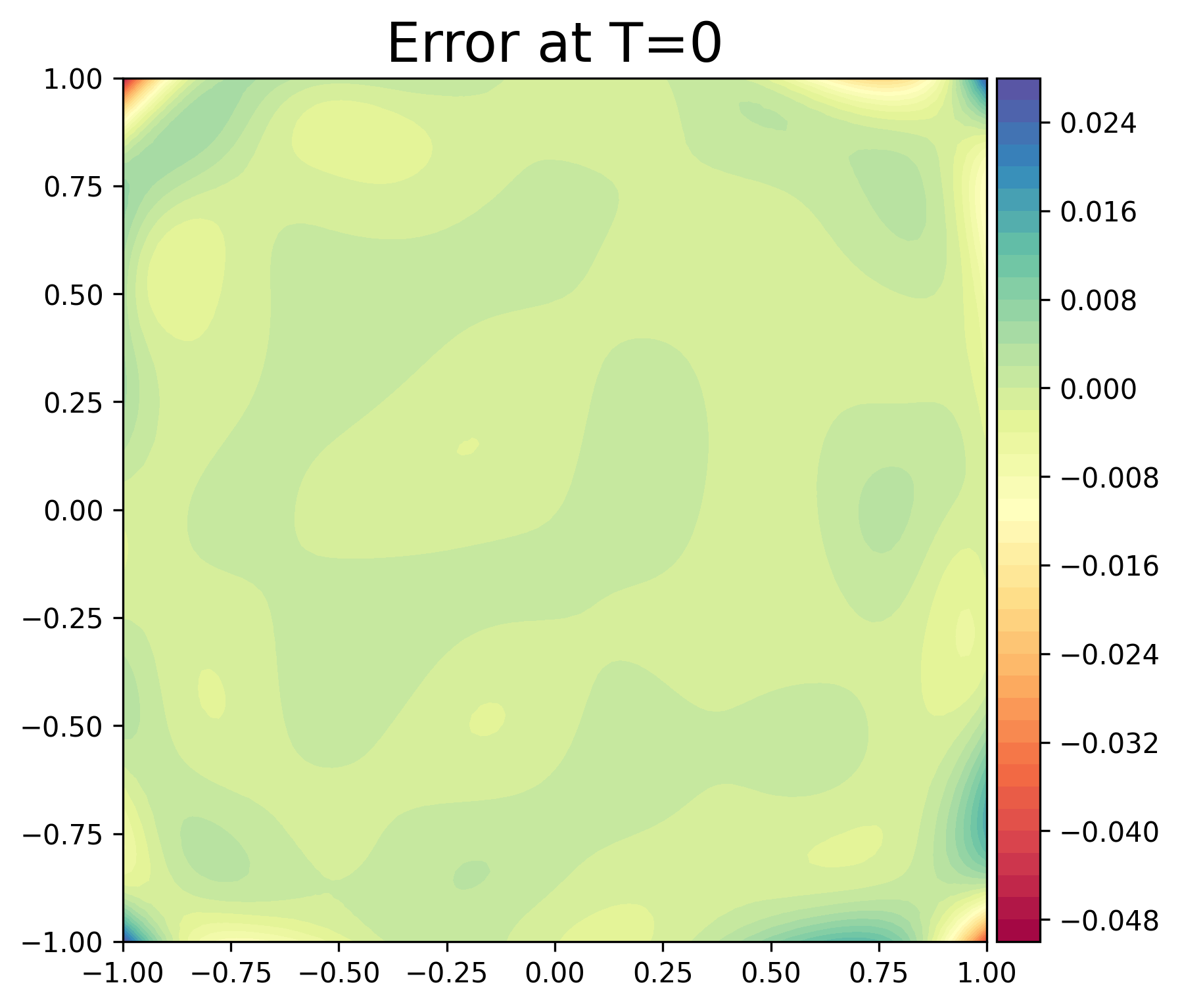}}
    \setcounter {subfigure} 0(a.2){
    \includegraphics[scale=0.27]{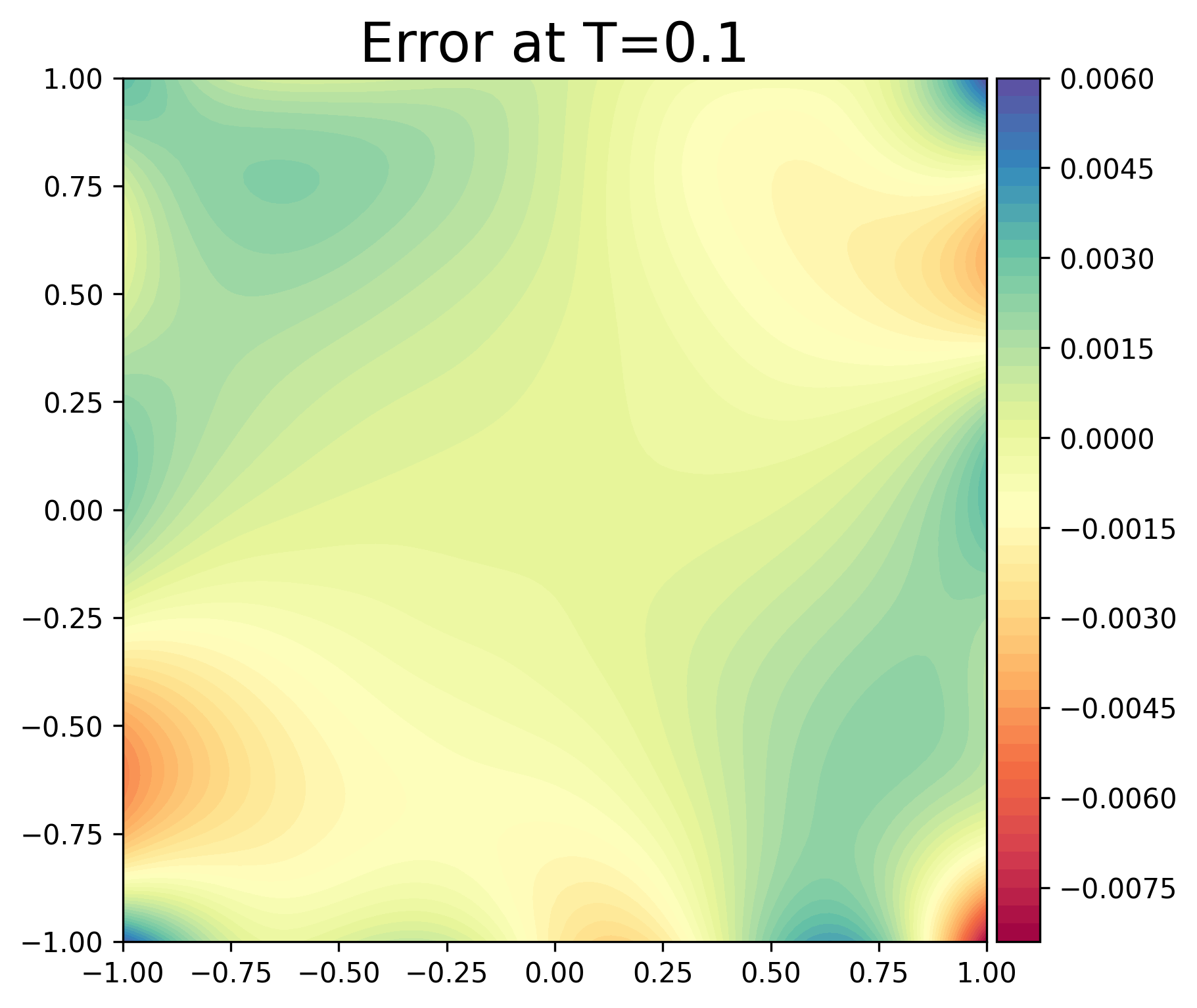}}
    \setcounter {subfigure} 0(a.3){
    \includegraphics[scale=0.27]{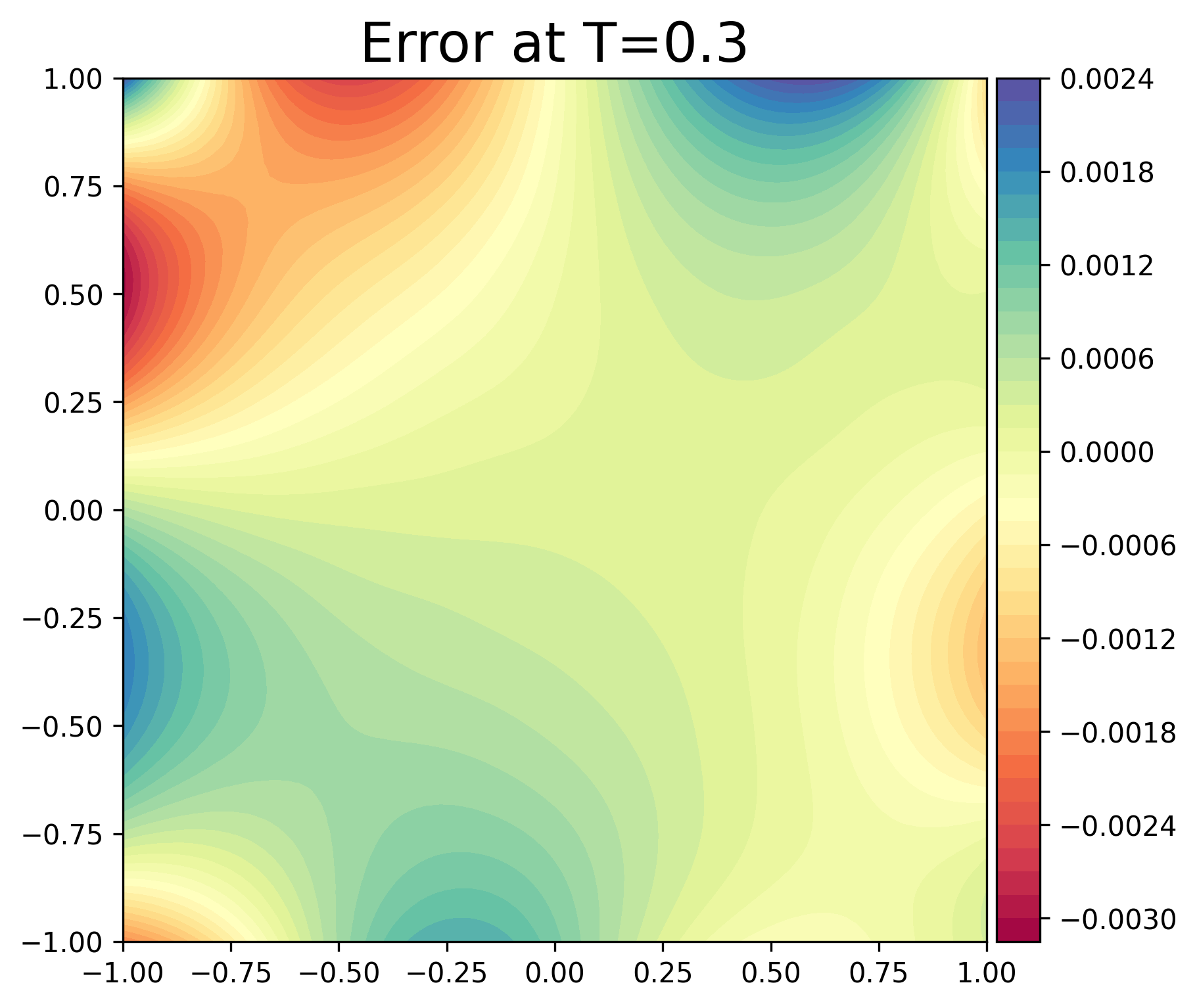}}\\
    \setcounter {subfigure} 0(b.1){
    \includegraphics[scale=0.27]{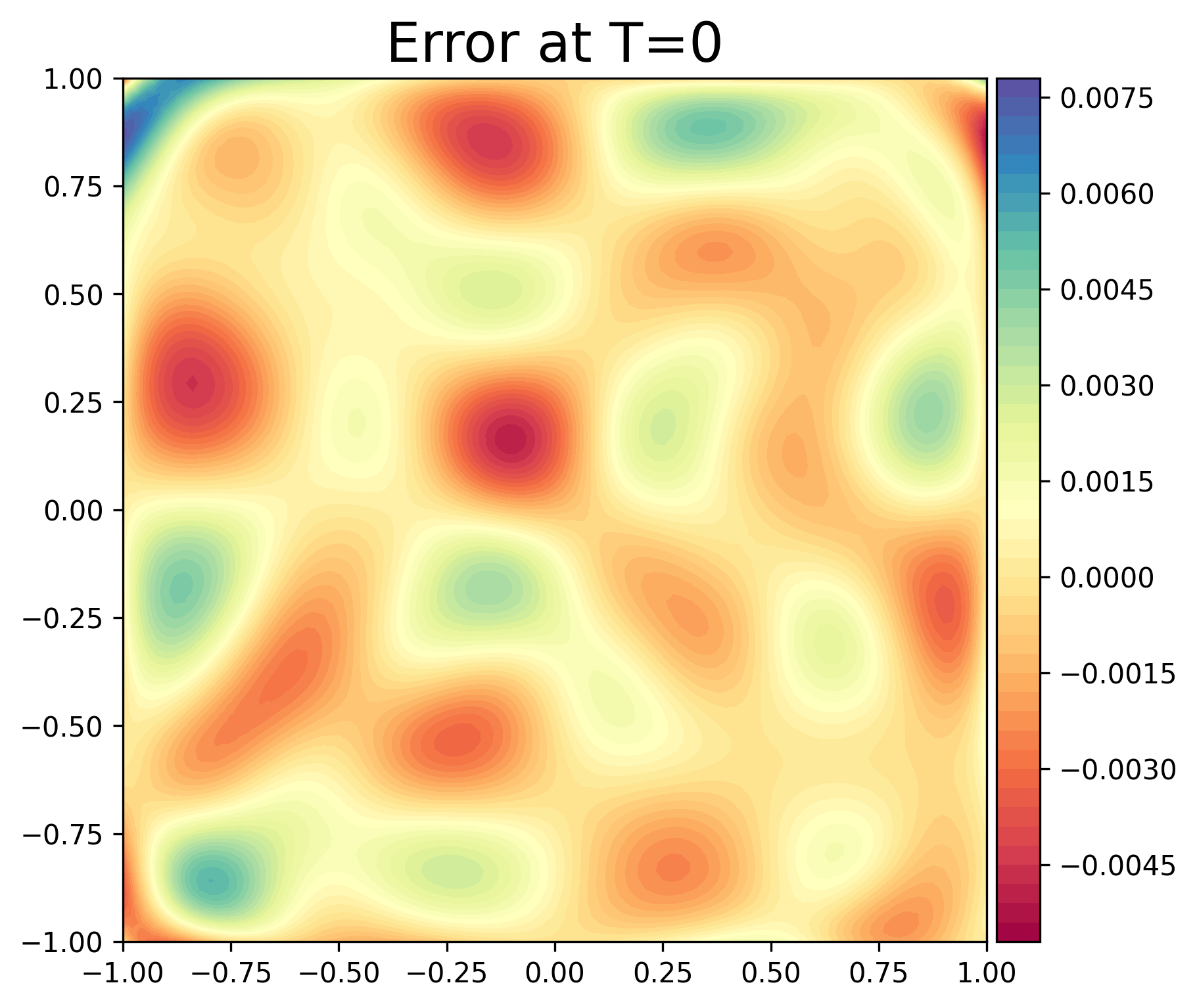}}
    \setcounter {subfigure} 0(b.2){
    \includegraphics[scale=0.27]{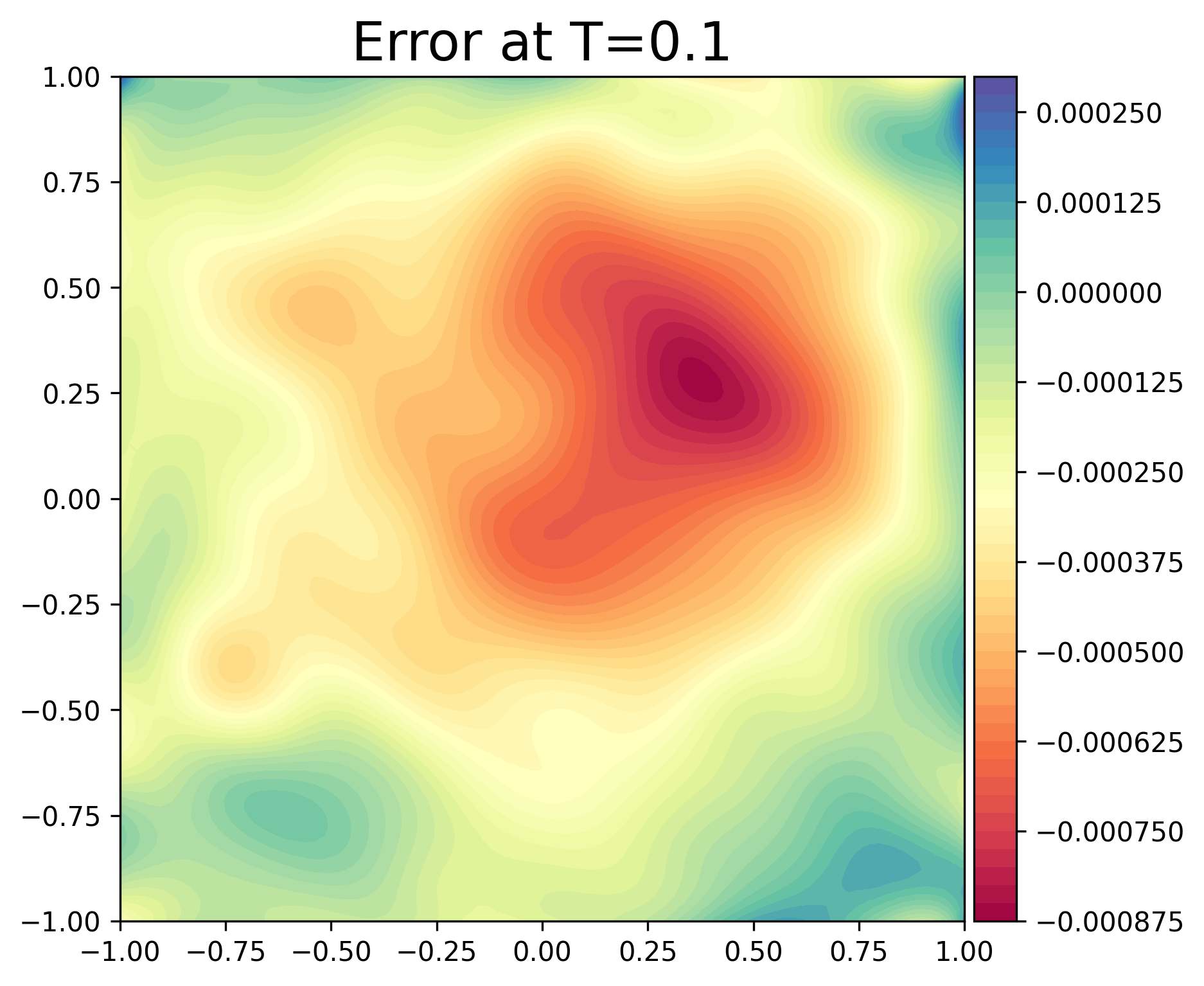}}
    \setcounter {subfigure} 0(b.3){
    \includegraphics[scale=0.27]{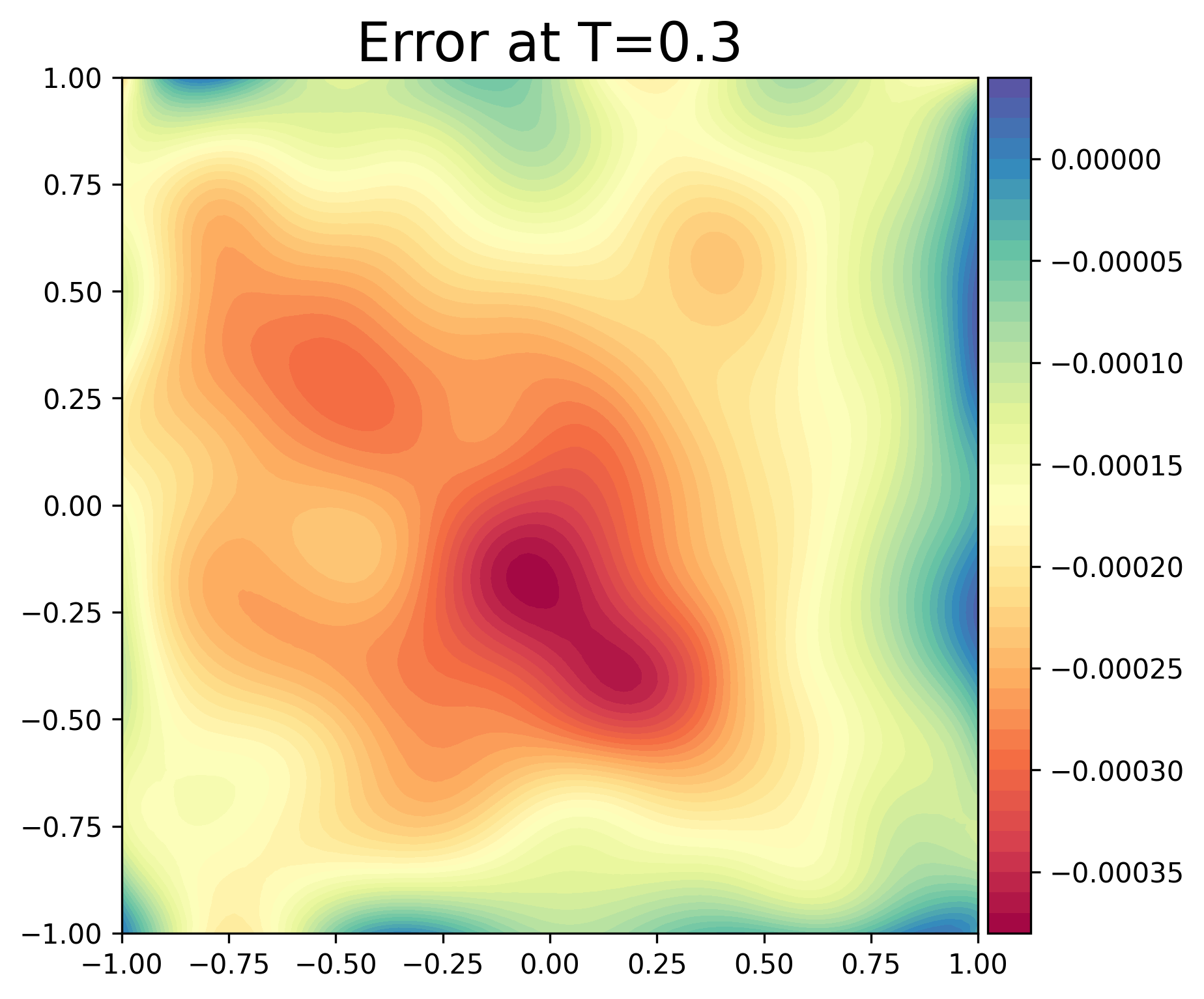}} 
    
    \caption{\textbf{Heat equation:} Pointwise error. First row: PINN model; second row: SSBE model, under the parameter setting $\lambda_1 = \lambda_2 = \lambda_3 = \lambda_4 = 1$.
}
    \label{Fig.Heat_error}
\end{figure}

\newpage
\subsection{Nonlinear Elliptic Equation}
In this example, we consider a nonlinear PDE to demonstrate that the proposed SSBE method is effective not only for linear problems but also for nonlinear settings. Let $\Omega = [-1,1] \times [-1,1] \subset \mathbb{R}^2$, $\mathcal{L} = -\Delta$ be the Laplacian, and $f(x_1, x_2)$ and $g(x_1, x_2) = 0$. The PDE is given by  
\begin{equation}\label{eq::CounterExample}
\left\{
\begin{aligned}
    -\alpha \Delta u + \beta u + \gamma u^k &= f(x_1, x_2), 
    && (x_1, x_2) \in \Omega, \\
    u(x_1, x_2) &= 0, 
    && (x_1, x_2) \in \partial \Omega.
\end{aligned}
\right.
\end{equation}
with the reference ground truth as $u(x_1, x_2) = \sin(\pi x_1)\sin(\pi x_2)$. In this experiment, the sampling strategy and the definition of $L_{HB}$ are identical to those in the heat equation example. We compare the original PINN and the proposed SSBE method under various PDE parameter settings. The corresponding numerical results are summarized in Table~\ref{tab:non_results}, showing that SSBE consistently achieves lower relative $L^2$ and $H^1$ errors across all tested configurations.

\begin{table}[htpb]
    \centering
    \begin{tabular}{|c|c|c|c|c|c|c|c|}
        \hline 
        \multicolumn{4}{|c|}{Parameters} & \multicolumn{2}{c|}{Original PINN} & \multicolumn{2}{c|}{SSBE} \\ 
        \hline
        $\alpha$ & $\beta$ & $\gamma$ & $k$ & Re $L^2$ & Re $H^1$ & Re $L^2$ & Re $H^1$ \\ 
        \hline 
        1 & 0 & 1 & 3 & $3.36 \mathrm{e\text{-}04}$& $1.04 \mathrm{e\text{-}03}$ & $1.78 \mathrm{e\text{-}05}$ & $3.54 \mathrm{e\text{-}05}$ \\ 
        \hline 
        1 & 0 & 1 & 5 & $4.82\mathrm{e\text{-}04}$ &  $1.58\mathrm{e\text{-}03}$& $2.01\mathrm{e\text{-}05}$ & $5.84\mathrm{e\text{-}05}$ \\ 
        \hline 
        1 & 1 & 1 & 3 & $2.45\mathrm{e\text{-}04}$ & $8.40\mathrm{e\text{-}04}$ & $1.78\mathrm{e\text{-}05}$ & $3.99\mathrm{e\text{-}05}$ \\ 
        \hline 
        1 & 5 & 5 & 3 & $1.05\mathrm{e\text{-}03}$ & $2.89\mathrm{e\text{-}03}$ & $2.06\mathrm{e\text{-}05}$ & $5.44\mathrm{e\text{-}05}$ \\ 
        \hline 
        0.1 & 5 & 5 & 3 &$4.99\mathrm{e\text{-}04}$  &  $1.86\mathrm{e\text{-}03}$ & $4.54\mathrm{e\text{-}05}$ & $2.37\mathrm{e\text{-}04}$ \\ 
        \hline
    \end{tabular}
    \caption{\textbf{Nonlinear elliptic equation:} Comparison between the original PINN and the proposed SSBE method under different parameter settings.}
    \label{tab:non_results}
\end{table}

\subsection{High Dimension Poisson Equation}
Neural networks are capable of mitigating the curse of dimensionality in solving high-dimensional problems. In this example, we illustrate this advantage by comparing the performance of the original PINN and the proposed SSBE method on the following high-dimensional Poisson equation:  
\begin{equation}
\left\{
\begin{aligned}
    -\Delta u &= f, 
    && \mathbf{x} \in \Omega, \\
    u &= g, 
    && \mathbf{x} \in \partial \Omega.
\end{aligned}
\right.
\end{equation}
where $\Omega = [-1,1]^d$ right-side terms
\begin{equation*}
    f(\mathbf{x}) = \frac{1}{d}\left(\sin \left(\frac{1}{d} \sum_{i=1}^d x_i\right) - 2\right),\ \mathbf{x}\in\Omega,\quad g(\mathbf{x})=\left(\frac{1}{d} \sum_{i=1}^d x_i\right)^2 + \sin \left(\frac{1}{d} \sum_{i=1}^d x_i\right),\ x\in\partial\Omega,
\end{equation*}
which admits the unique ground truth 
\[
u(\mathbf{x}) = \left(\frac{1}{d} \sum_{i=1}^d x_i\right)^2 + \sin \left(\frac{1}{d} \sum_{i=1}^d x_i\right).
\]
The SSBE boundary loss in the $d$-dimensional setting is defined as  
\begin{equation}
\begin{aligned}
L_{HB} = \sum_{i=1}^{d} \bigg( 
    &\int_{\partial \Omega_{1i}} \left( 
        \frac{\partial}{\partial x_i} u_\theta(x_i=1, x_1,\dots, x_{i-1}, x_{i+1},\dots, x_d) 
        - \frac{\partial}{\partial x_i} g(x_i=1, x_1,\dots, x_{i-1}, x_{i+1},\dots, x_d) 
    \right)^2 \diff S \\
  +\ &\int_{\partial \Omega_{2i}} \left( 
        \frac{\partial}{\partial x_i} u_\theta(x_i=-1, x_1,\dots, x_{i-1}, x_{i+1},\dots, x_d) 
        - \frac{\partial}{\partial x_i} g(x_i=-1, x_1,\dots, x_{i-1}, x_{i+1},\dots, x_d) 
    \right)^2 \diff S 
\bigg),
\end{aligned}
\label{Eqn.High_Poi_HB}
\end{equation}
with the sampling strategy as follows: $10{,}000$ points are uniformly sampled in the interior of $\Omega$, and $1{,}000 \times d$ points are uniformly sampled on the boundary $\partial\Omega$. The problem dimensions considered are $d = 3$, $10$, $20$, and $50$. The numerical results, presented in Table~\ref{tab:high_d_poi_table}, show that the proposed SSBE loss achieves accurate $H^1$ approximations even in high-dimensional settings, consistently outperforming the original PINN.

\begin{table}[htpb]
    \centering
    \begin{tabular}{|c|c|c|c|c|}
        \hline 
        Dimension & \multicolumn{2}{c|}{Orginal PINN} & \multicolumn{2}{c|}{SSBE} \\ 
        \hline 
        $d$ & Re $L^2$ & Re $H^1$ & Re $L^2$ & Re $H^1$ \\ 
        \hline 
        3 & $2.53\mathrm{e\text{-}04}$ & $7.11\mathrm{e\text{-}04}$ & $6.41\mathrm{e\text{-}05}$ & $8.87\mathrm{e\text{-}05}$ \\ 
        \hline 
        10 & $7.82\mathrm{e\text{-}04}$ & $1.49\mathrm{e\text{-}03}$ & $3.60\mathrm{e\text{-}04}$ & $6.29\mathrm{e\text{-}04}$ \\ 
        \hline 
        20 & $1.44\mathrm{e\text{-}03}$ & $2.43\mathrm{e\text{-}03}$ & $2.83\mathrm{e\text{-}04}$ & $5.50\mathrm{e\text{-}04}$ \\ 
        \hline 
        50 & $1.54\mathrm{e\text{-}03}$ & $2.42\mathrm{e\text{-}03}$ & $4.38\mathrm{e\text{-}04}$ & $4.73\mathrm{e\text{-}04}$ \\ 
        \hline 
    \end{tabular}
    \caption{\textbf{High-dimensional Poisson equation:} Relative $L^2$ and $H^1$ errors between the predicted and exact solutions $u(\mathbf{x})$ for the original PINN and the proposed SSBE method across different problem dimensions.}
    \label{tab:high_d_poi_table}
\end{table}

\begin{color}{black}
\subsection{Irregular Domain}

\subsubsection{Irregular domain: intrinsic boundary parameterization}
To further examine the robustness of the proposed SSBE-PINN on complex geometries, we consider the Poisson equation on a two-dimensional domain with a smooth oscillatory boundary. Specifically, we study the problem
\begin{equation}
\left\{
\begin{aligned}
    -\Delta u &= f,
    && \mathbf{x} \in \Omega, \\
    u &= g,
    && \mathbf{x} \in \partial\Omega.
\end{aligned}
\right.
\end{equation}
 the right-side term can be defined as following
\begin{equation*}
    f(x,y) = -2e^{x+y}, (x,y)\in\Omega \qquad g = e^{x+y}|_{\partial\Omega}.
\end{equation*}
The domains $\Omega$ are defined as a collection of irregular domains with boundaries which can be represented in polar form as
\begin{equation*}
    \Omega = \{(x,y): r < R(\omega), \omega\in[0,2\pi)\},\qquad \partial\Omega=\{(x,y): r = R(\omega), \omega\in[0,2\pi)\},  
\end{equation*}
where $R(\theta)$ is a smooth, oscillatory radius function. And it can be derived that this equation admits the unique ground truth
\begin{equation*}
    u(x,y) = e^{x+y}, \quad (x,y)\in\Omega. 
\end{equation*}

For such geometries, by defining the polar coordinate
\begin{equation*}
    \gamma(\omega)=\bigl(R(\omega)\cos\omega,\; R(\omega)\sin\omega\bigr),
\end{equation*}
the proposed SSBE enforces boundary conditions in the Sobolev $H^1(\partial\Omega)$ sense, penalizing not only the boundary value but also its intrinsic tangential derivative with respect to $\theta$:
\begin{equation*}
    L_{LB}+L_{HB}
= 
\int_{0}^{2\pi}\left(
  |u(\gamma(\omega))-g(\gamma(\omega))|^2
  +
  \left|\partial_{\omega}u(\gamma(\omega))
       -\partial_{\omega}g(\gamma(\omega))\right|^2
\right)|\gamma'(\omega)|\diff{\omega}.
\end{equation*}

We test the SSBE framework on various intrinsic boundary parametrization $R(\omega)$. The numerical results are summarized in Table~\ref{tab:ir_domain_poi_table} and illustrated in Figure~\ref{fig.ir_poi}. Across all domains, SSBE consistently outperforms the standard PINN with $L^2(\partial\Omega)$ enforcement, achieving significantly lower relative errors in both $L^2(\Omega)$ and $H^1(\Omega)$ norms.

\begin{table}[htpb]
    \centering
    \begin{tabular}{|c|c|c|c|c|}
        \hline 
        Domain Shape & \multicolumn{2}{c|}{Orginal PINN} & \multicolumn{2}{c|}{SSBE} \\ 
        \hline 
         $\omega\in[0,2\pi)$& Re $L^2$ & Re $H^1$ & Re $L^2$ & Re $H^1$ \\ 
        \hline 
        Flowers $R(\omega)= 1 + 0.3\cos(5\omega)$ & $8.312\mathrm{e\text{-}05}$ & $4.035\mathrm{e\text{-}04}$ & $8.589\mathrm{e\text{-}06}$ & $3.987\mathrm{e\text{-}05}$ \\ 
        \hline 
        Heart $R(\omega)= 0.6 + 0.6\cos(\omega)$ & $3.465\mathrm{e\text{-}05}$ & $2.073\mathrm{e\text{-}04}$ & $2.711\mathrm{e\text{-}06}$ & $5.053\mathrm{e\text{-}05}$ \\ 
        \hline 
        Lips $R(\omega)= 1+ 0.25\cos(2\omega) + 0.1\cos(6\omega)$ & $5.533\mathrm{e\text{-}05}$ & $2.601\mathrm{e\text{-}04}$ & $1.004\mathrm{e\text{-}05}$ & $4.810\mathrm{e\text{-}05}$ \\ 
        \hline 
    \end{tabular}
    \caption{\textcolor{black}{\textbf{Irregular Domain Poisson equation:} Relative $L^2$ and $H^1$ errors between the predicted and exact solutions $u(\mathbf{x})$ for the original PINN and the proposed SSBE method across shapes.}}
    \label{tab:ir_domain_poi_table}
\end{table}

\begin{figure}[!ht]
    \centering
    \setcounter {subfigure} 0(a.1){
    \includegraphics[scale=0.24]{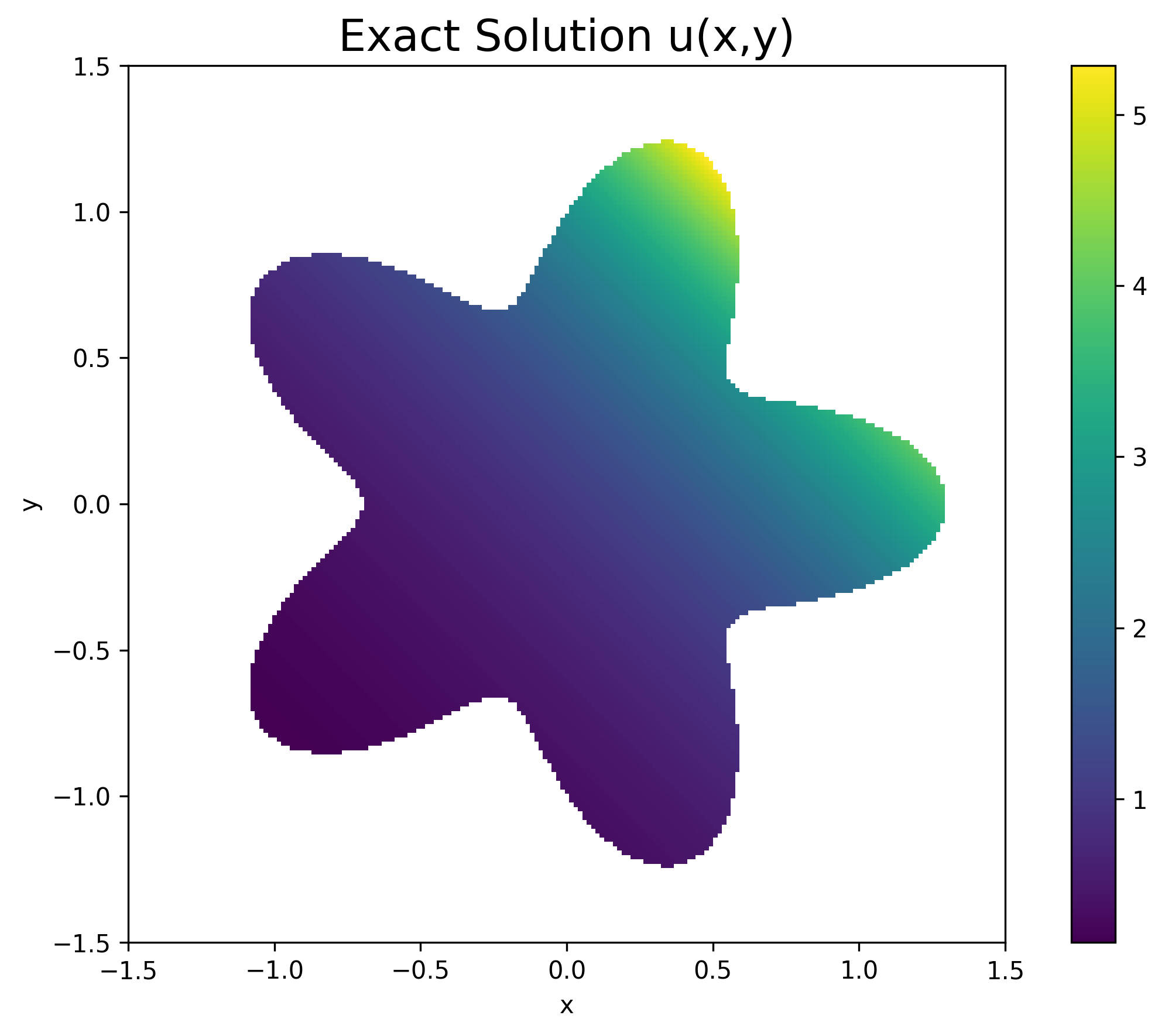}}
    \setcounter {subfigure} 0(a.2){
    \includegraphics[scale=0.24]{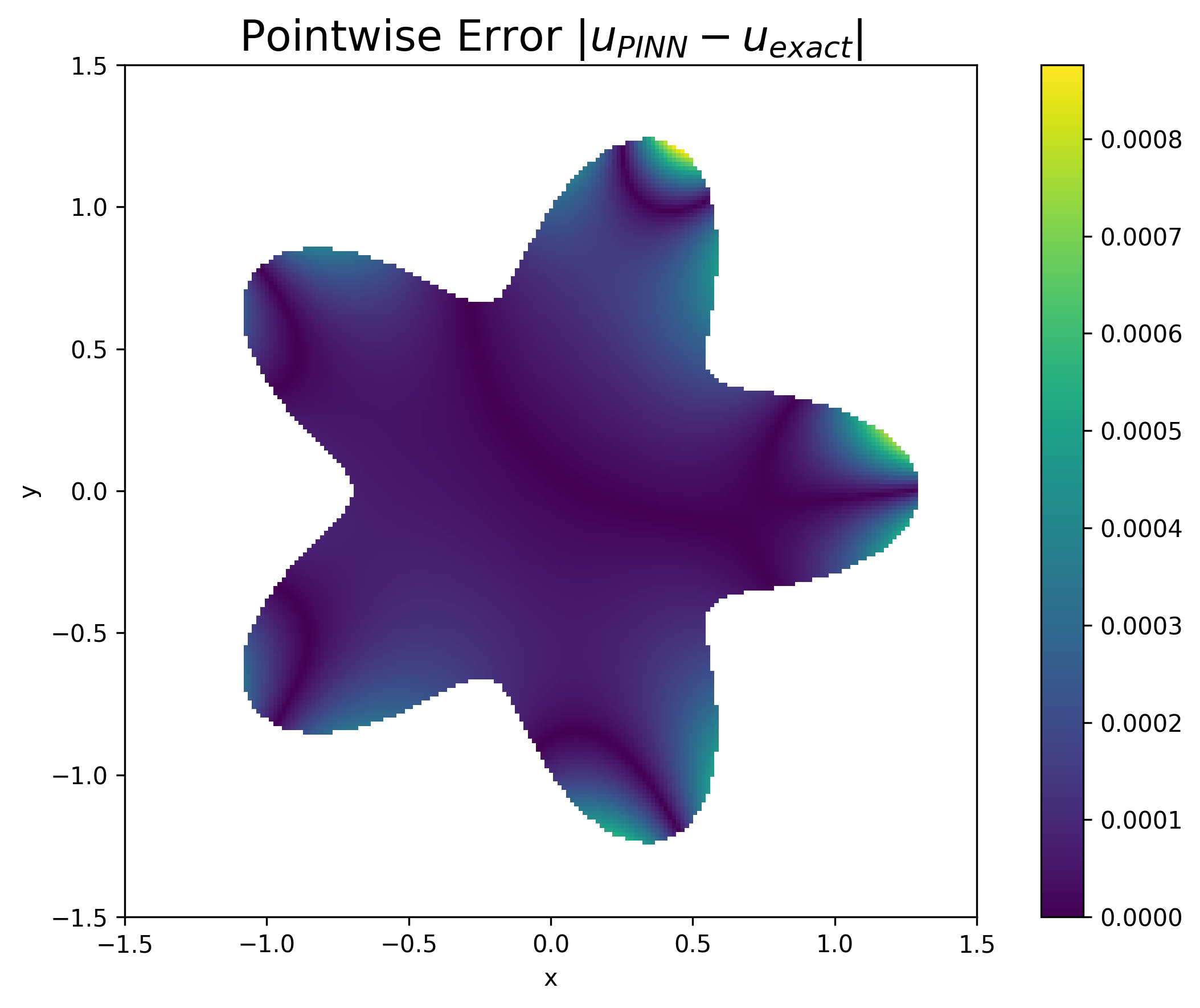}}
    \setcounter {subfigure} 0(a.3){
    \includegraphics[scale=0.24]{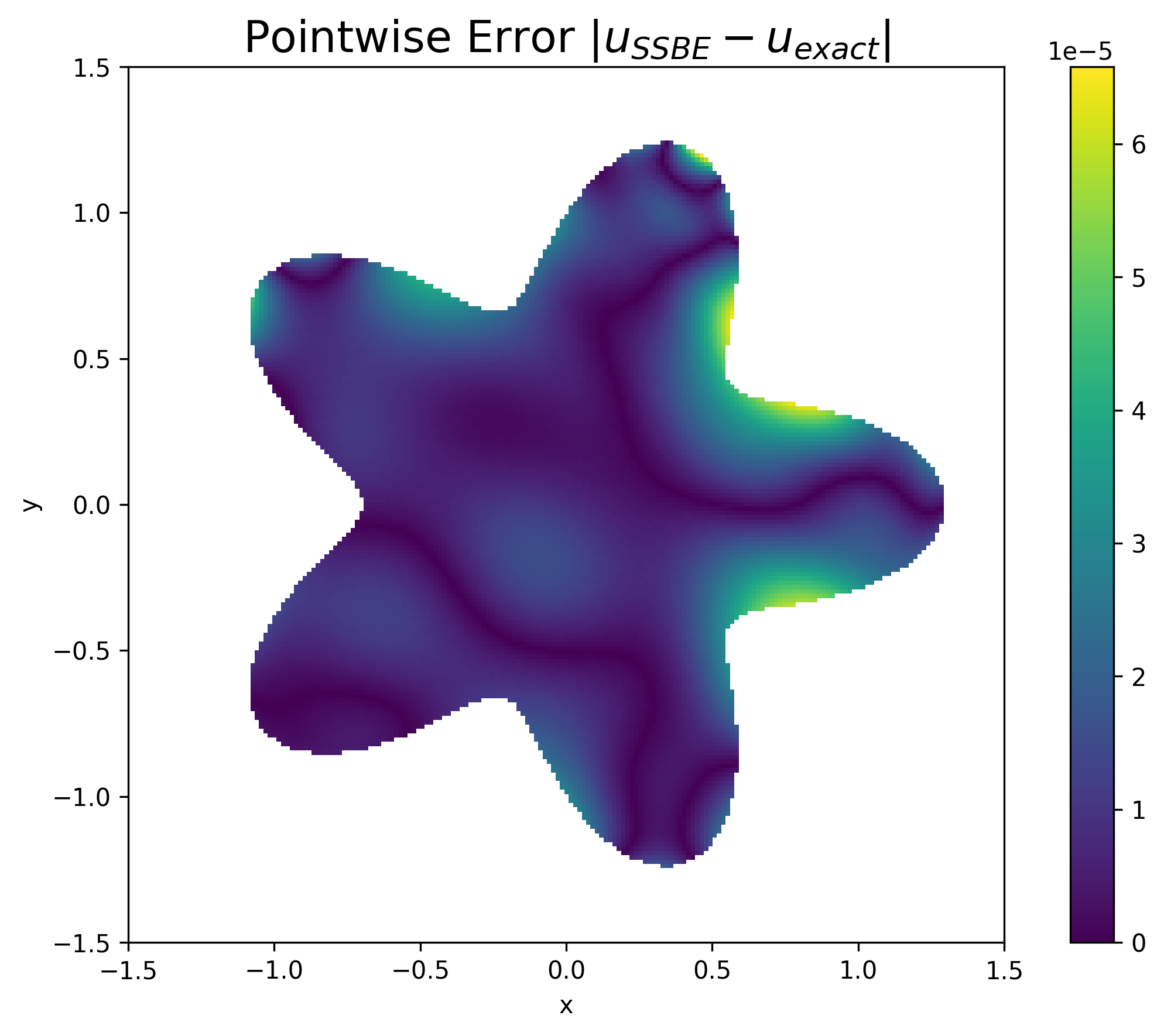}}\\
    \setcounter {subfigure} 0(b.1){
    \includegraphics[scale=0.24]{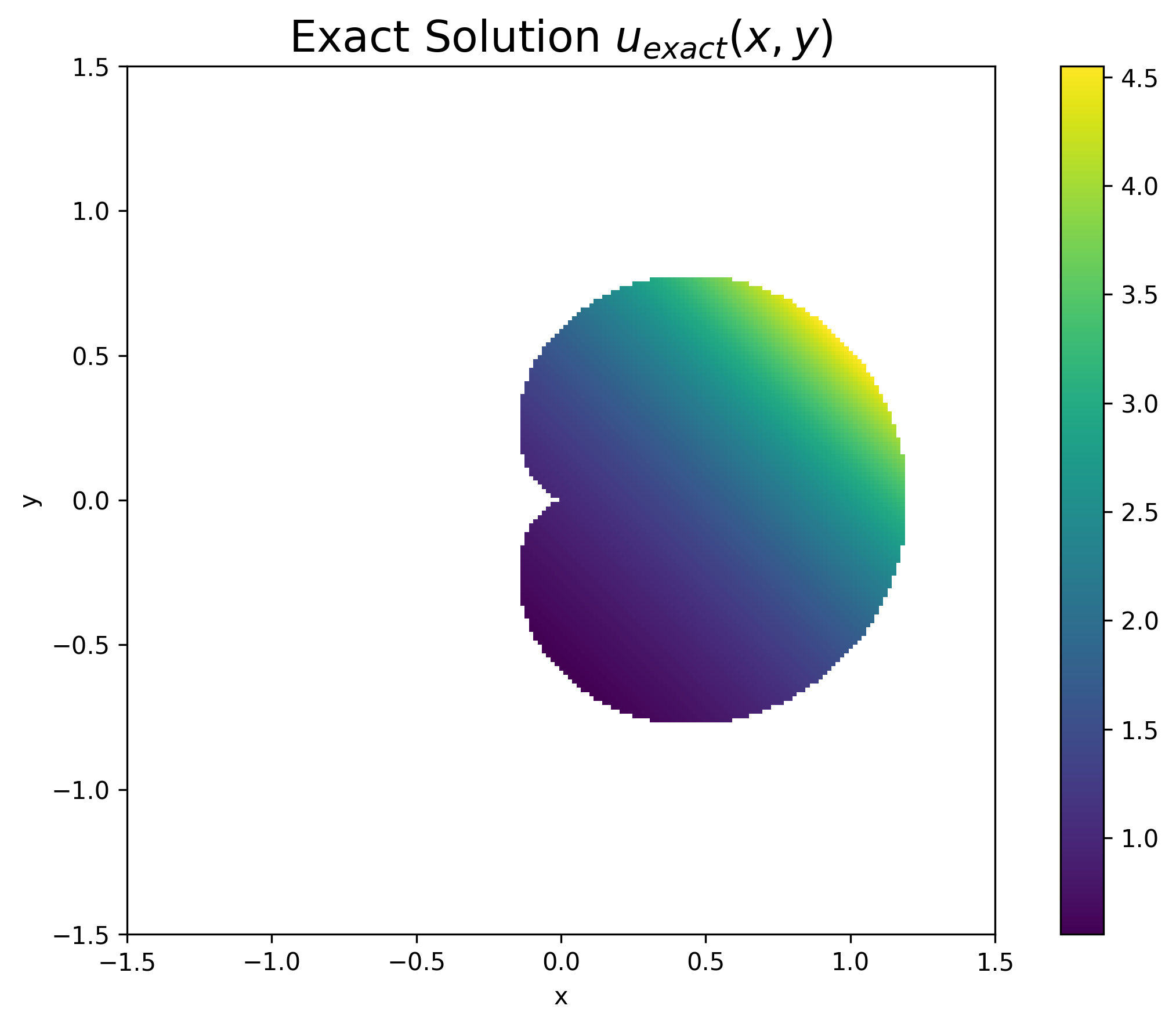}}
    \setcounter {subfigure} 0(b.2){
    \includegraphics[scale=0.24]{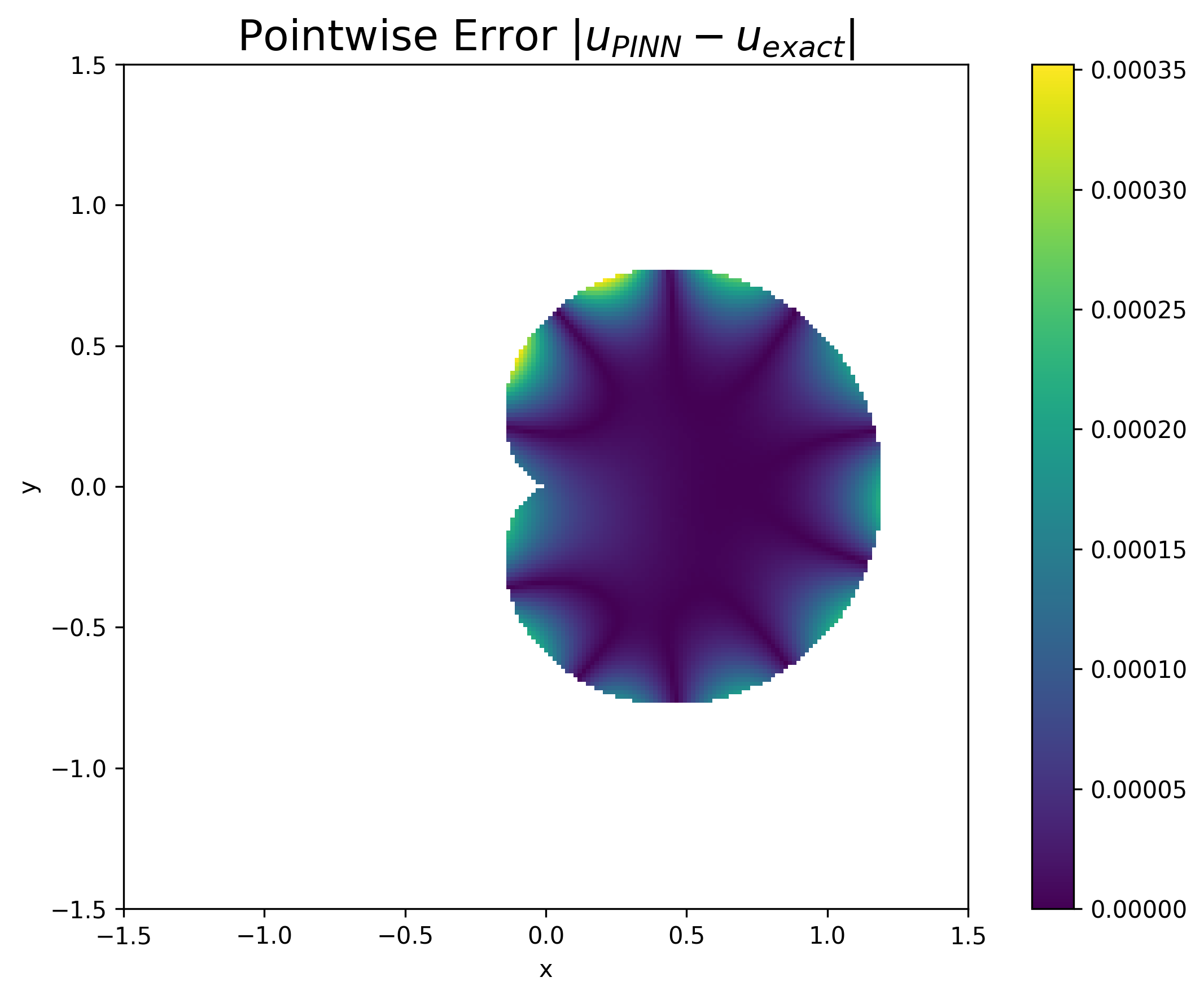}}
    \setcounter {subfigure} 0(b.3){
    \includegraphics[scale=0.24]{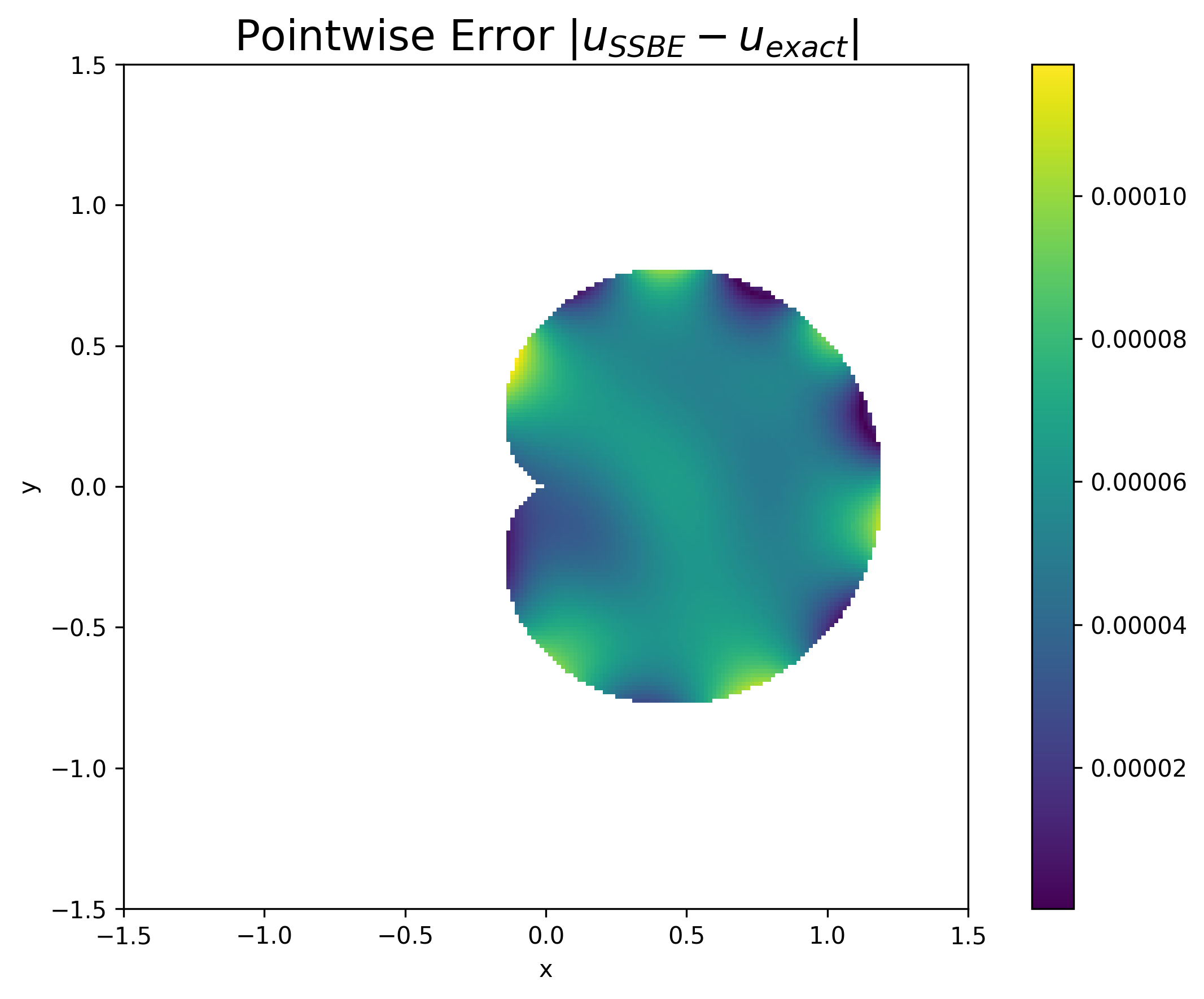}} \\
    \setcounter {subfigure} 0(c.1){
    \includegraphics[scale=0.24]{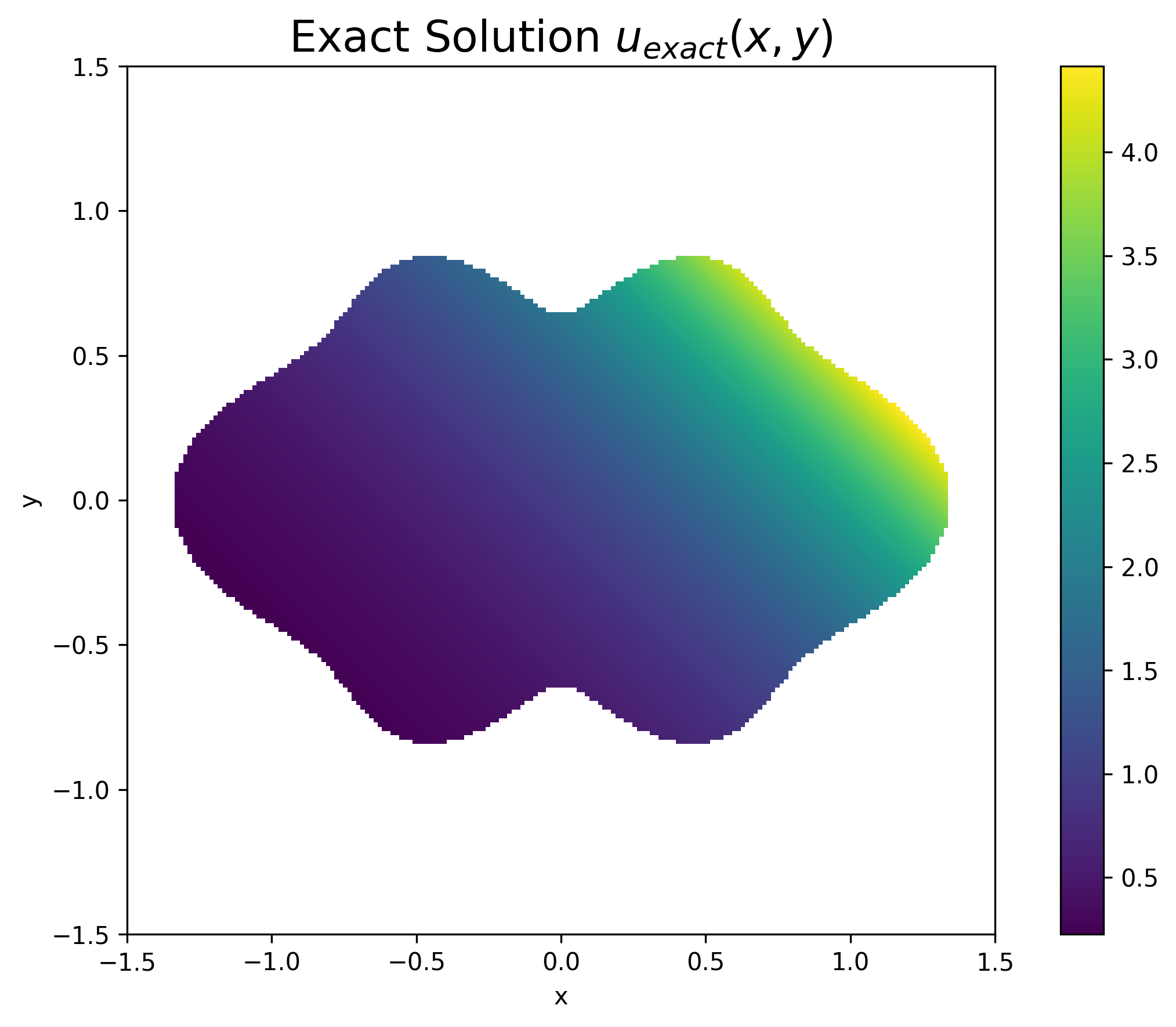}}
    \setcounter {subfigure} 0(c.2){
    \includegraphics[scale=0.24]{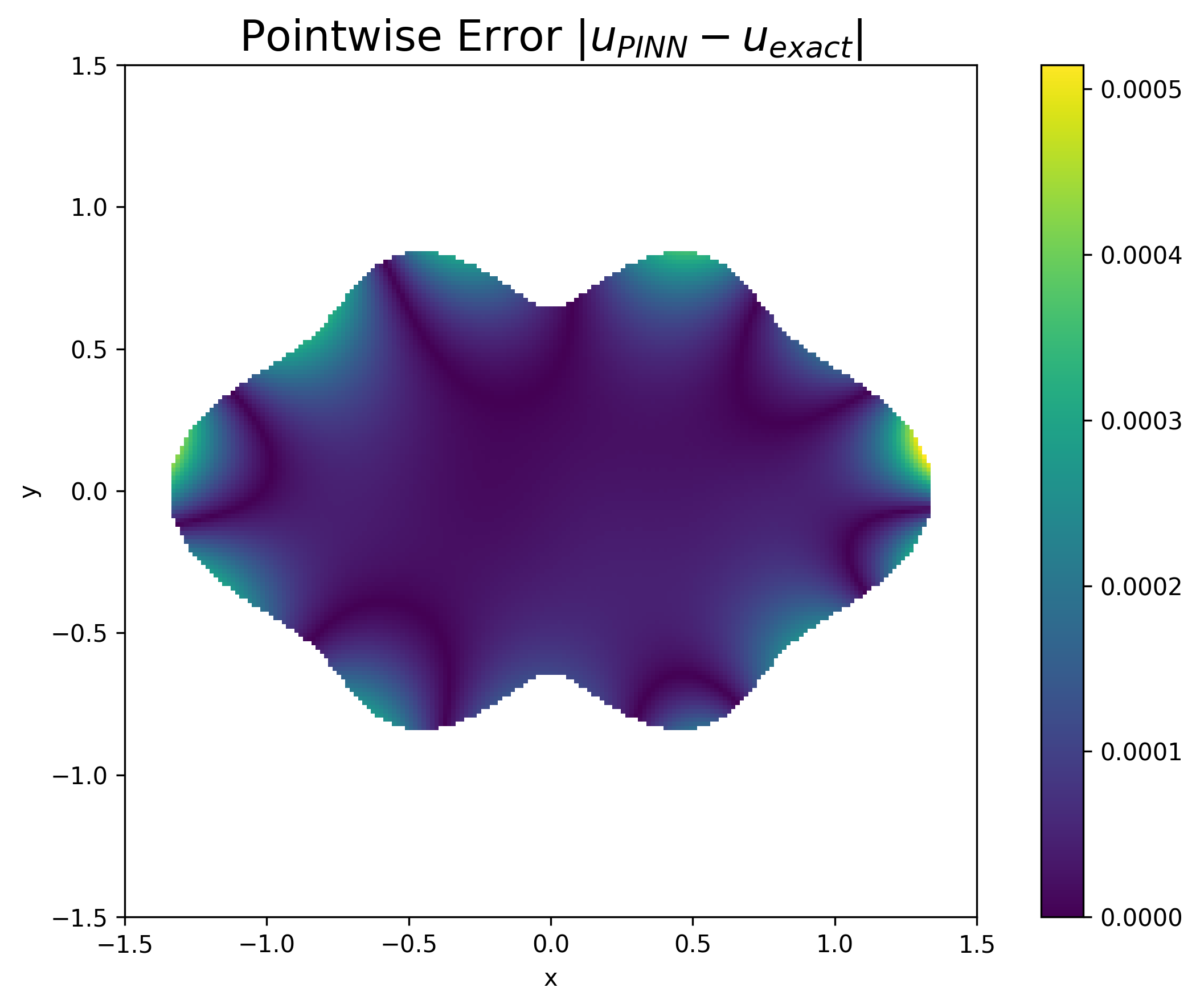}}
    \setcounter {subfigure} 0(c.3){
    \includegraphics[scale=0.24]{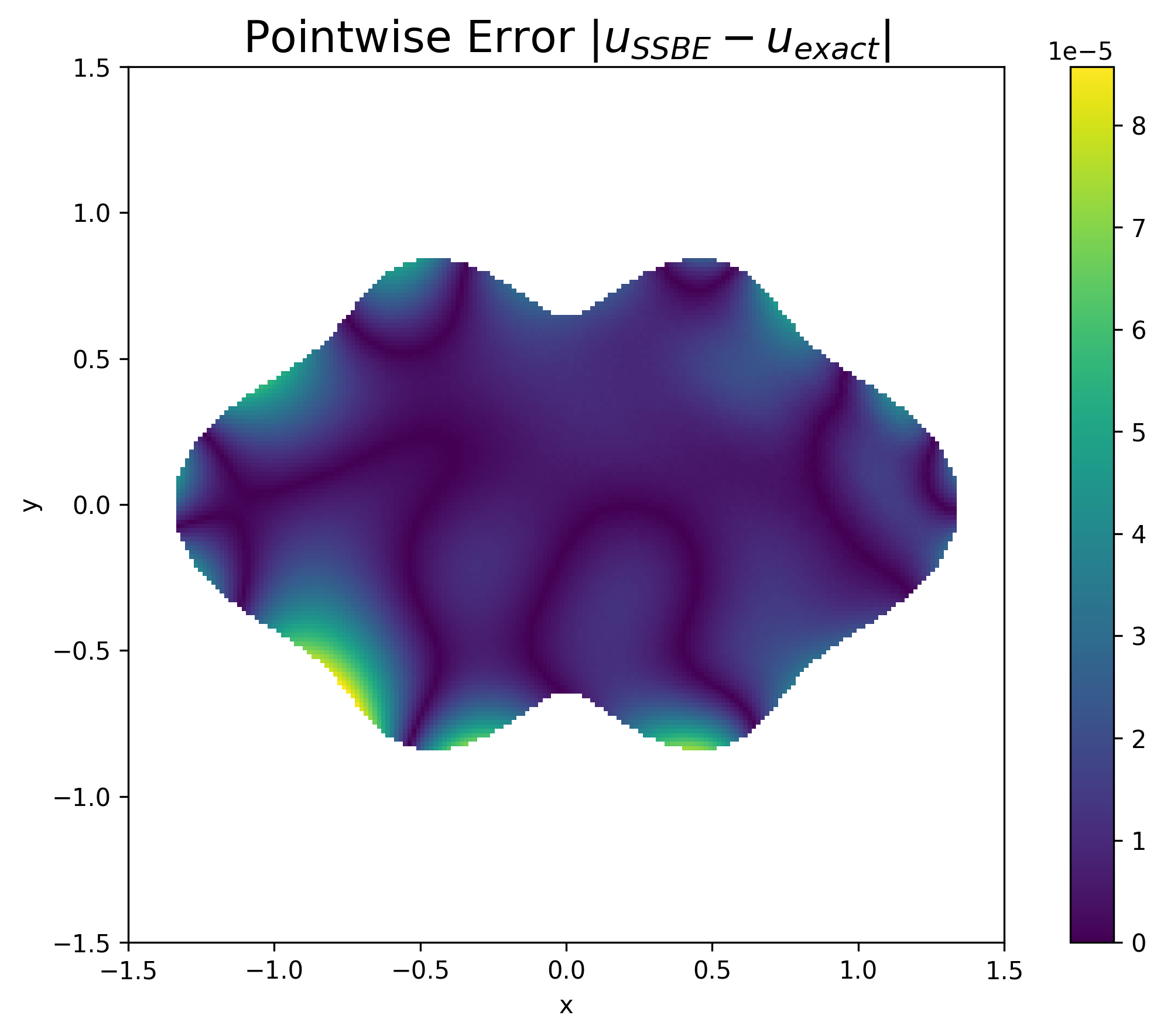}} 
    
\caption{\textcolor{black}{\textbf{Poisson equation on irregular domains.} 
For each domain ((a) flower, (b) heart, and (c) lips), the first column shows the exact solution, the second column shows the pointwise error of the standard PINN, and the third column shows the pointwise error of the proposed SSBE-PINN.}}
    \label{fig.ir_poi}
\end{figure}

%%多孔介质问题
%% 带传播系数的poisson equation
\subsubsection{Irregular domain: porous medium}
We then consider a Poisson problem posed on a two-dimensional porous domain. Specially, we consider the Poisson equation with diffusion coefficient $k(\mathbf{x})=1+x+y$ as
\begin{equation}
\left\{
\begin{aligned}
    -\nabla \cdot (k(\mathbf{x})\nabla u(\mathbf{x}))  &= f,
    && \mathbf{x} \in \Omega, \\
    u &= g,
    && \mathbf{x} \in \partial\Omega.
\end{aligned}
\right.
\end{equation}
with the right-side terms
\begin{equation*}
\begin{aligned}
    f(x,y) &= \frac{\pi^2}{2}(1+x+y)\sin{\left(\frac{\pi}{2}x\right)}\sin{\left(\frac{\pi}{2}y\right)}-\frac{\pi}{2}\sin{\left(\frac{\pi}{2}(x+y)\right)}, (x,y)\in\Omega,\\
     g(x,y) &= \sin\left(\frac{\pi}{2}x\right)\sin\left(\frac{\pi}{2}y\right)|_{\partial\Omega}.
\end{aligned}
\end{equation*}
We study the above equation defined on several porous medium cases by removing holes from the two-dimension square $[-2,2]^2$. In Case~1, the computational domain $\Omega$ is obtained by removing three solid 
inclusions from the square $[-2,2]^2$, namely
\begin{equation}
    \begin{aligned}
        &\Omega := [-2,2]^2 \setminus 
        \left(\mathrm{Ellipse}_1 \cup \mathrm{Ellipse}_2 \cup \mathrm{Ellipse}_3\right), \\
        &\mathrm{Ellipse}_1 : 4x^2 + 9y^2 = 1, \\
        &\mathrm{Ellipse}_2 : (x+1)^2 + (y-1)^2 = \tfrac{1}{4}, \\
        &\mathrm{Ellipse}_3 : (x-1)^2 + (y+1)^2 = \tfrac{1}{4}, \\
        & \partial \Omega:= \partial [-2,2]^2\cup \partial \mathrm{Ellipse}_1 \cup \partial \mathrm{Ellipse}_2 \cup \partial \mathrm{Ellipse}_3.
    \end{aligned}
\end{equation}
In Cases~2 and~3, we further increase the complexity of the porous medium by 
removing a larger number of holes of varying sizes, leading to more highly 
perforated domains. And it can be derived that this equation admits the unique ground truth
\begin{equation*}
    u(x,y)=\sin\left(\frac{\pi}{2}x\right)\sin\left(\frac{\pi}{2}y\right),\quad (x,y)\in\Omega.
\end{equation*}

The numerical results are summarized in Table~\ref{tab:porous medium} and illustrated in Figure~\ref{Fig.porous medium}. Across the considered cases, SSBE consistently outperforms the standard PINN with $L^2(\partial\Omega)$ penalty, achieving lower relative errors in both the $L^2(\Omega)$ and $H^1(\Omega)$ norms.
\begin{table}[htpb]
    \centering
    \begin{tabular}{|c|c|c|c|c|}
        \hline 
        Domain  & \multicolumn{2}{c|}{Orginal PINN} & \multicolumn{2}{c|}{SSBE} \\ 
        \hline 
          & Re $L^2$  & Re $H^1$ & Re $L^2$ & Re $H^1$ \\ 
        \hline 
        Case1 & $1.3857\mathrm{e\text{-}03}$ & $2.1237\mathrm{e\text{-}03}$ & $2.0496\mathrm{e\text{-}04}$ & $4.2578\mathrm{e\text{-}04}$ \\ 
        \hline 
        Case2 & $4.0662\mathrm{e\text{-}04}$ & $1.2502\mathrm{e\text{-}03}$ & $1.6125\mathrm{e\text{-}04}$ & $4.3977\mathrm{e\text{-}04}$ \\ 
        \hline 
        Case3 & $6.9132\mathrm{e\text{-}04}$ & $1.9500\mathrm{e\text{-}03}$ & $1.5940\mathrm{e\text{-}04}$ & $4.5886\mathrm{e\text{-}04}$ \\ 
        \hline 
    \end{tabular}
    \caption{\textcolor{black}{\textbf{Porous medium:} Relative $L^2$ and $H^1$ errors between the predicted and exact solutions $u(\mathbf{x})$ for the original PINN and the proposed SSBE method across shapes.}}
    \label{tab:porous medium}
\end{table}

\begin{figure}[!ht]
    \centering
    \setcounter {subfigure} 0(a.1){
    \includegraphics[scale=0.24]{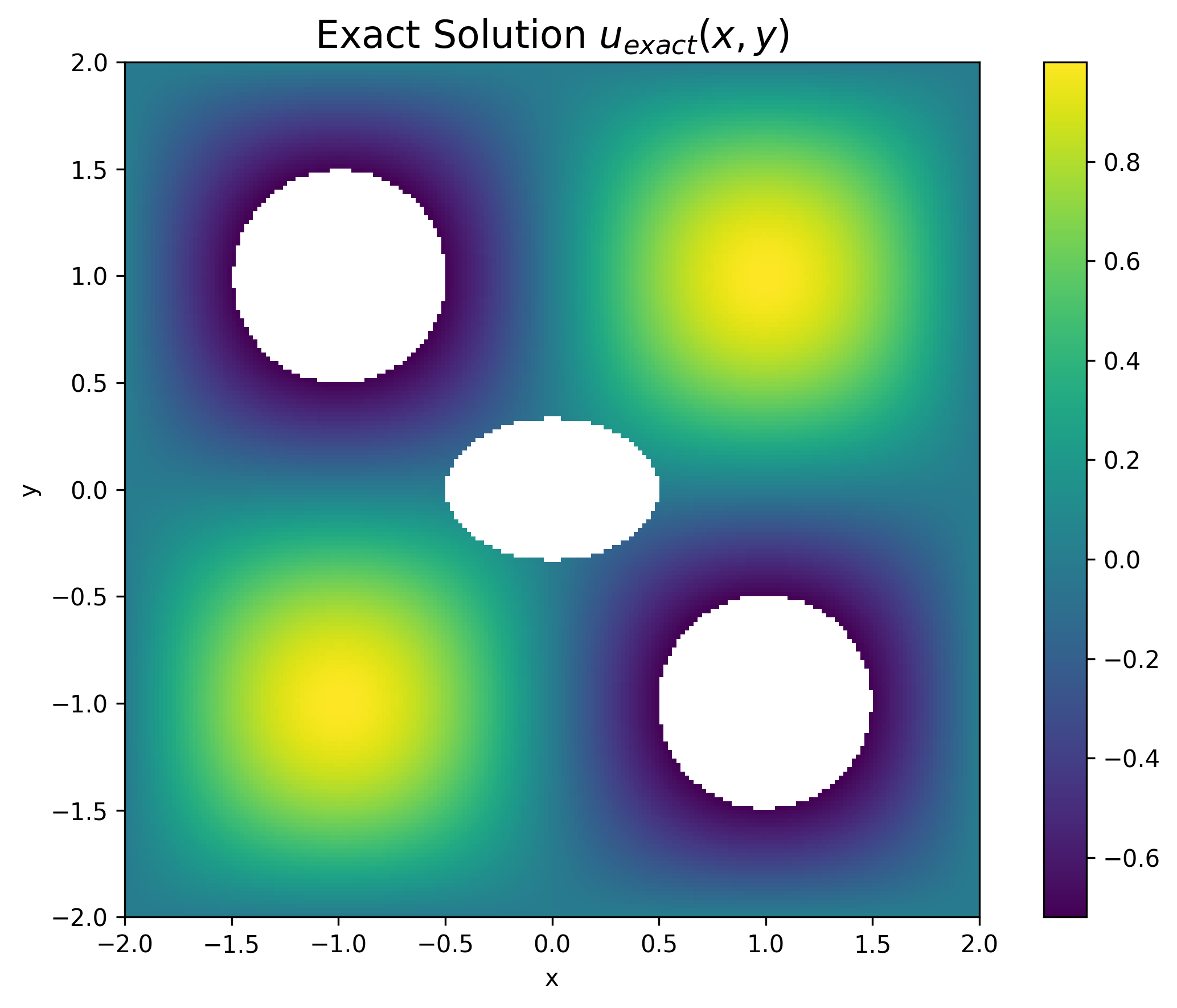}}
    \setcounter {subfigure} 0(a.2){
    \includegraphics[scale=0.24]{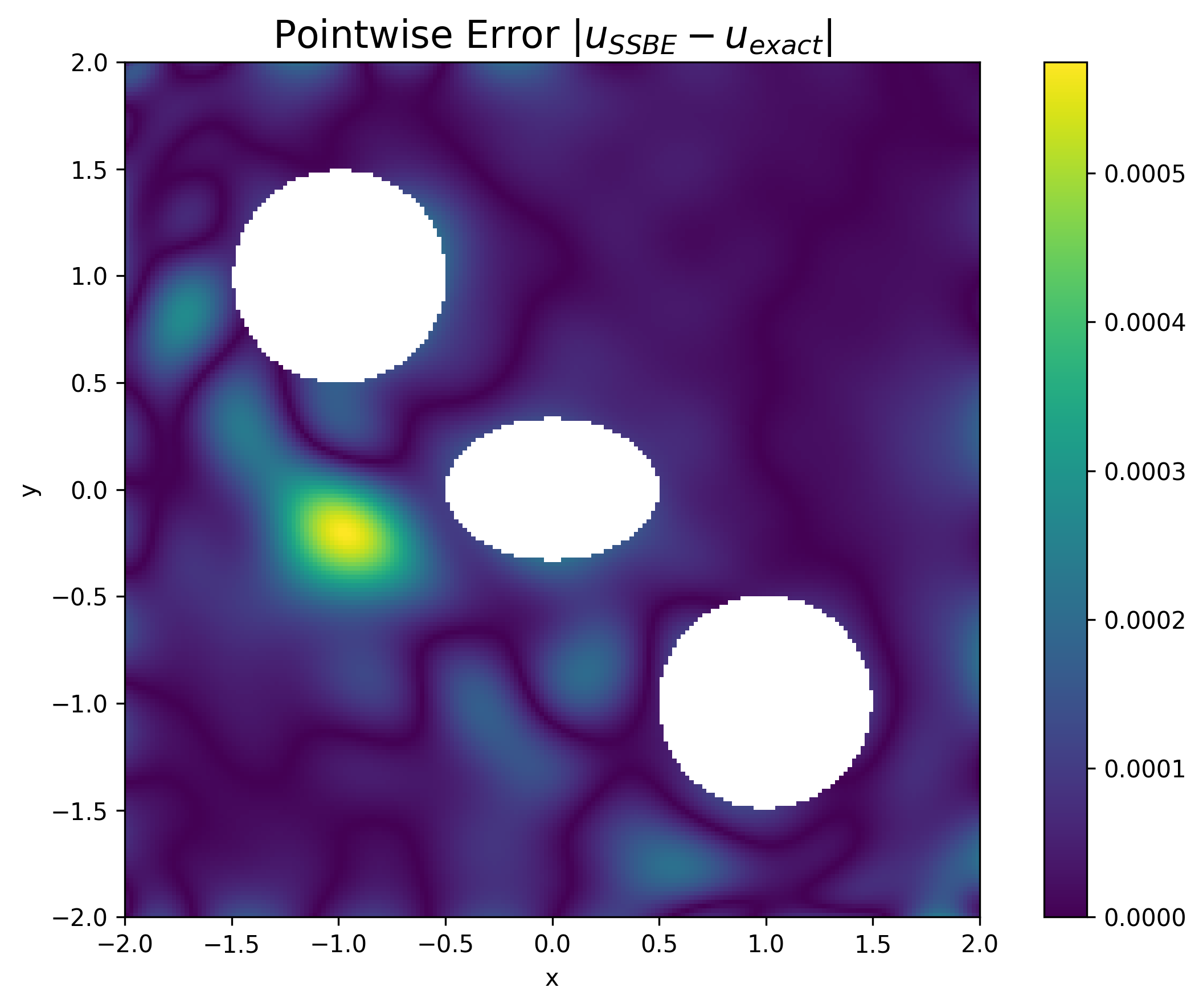}}
    \setcounter {subfigure} 0(a.3){
    \includegraphics[scale=0.24]{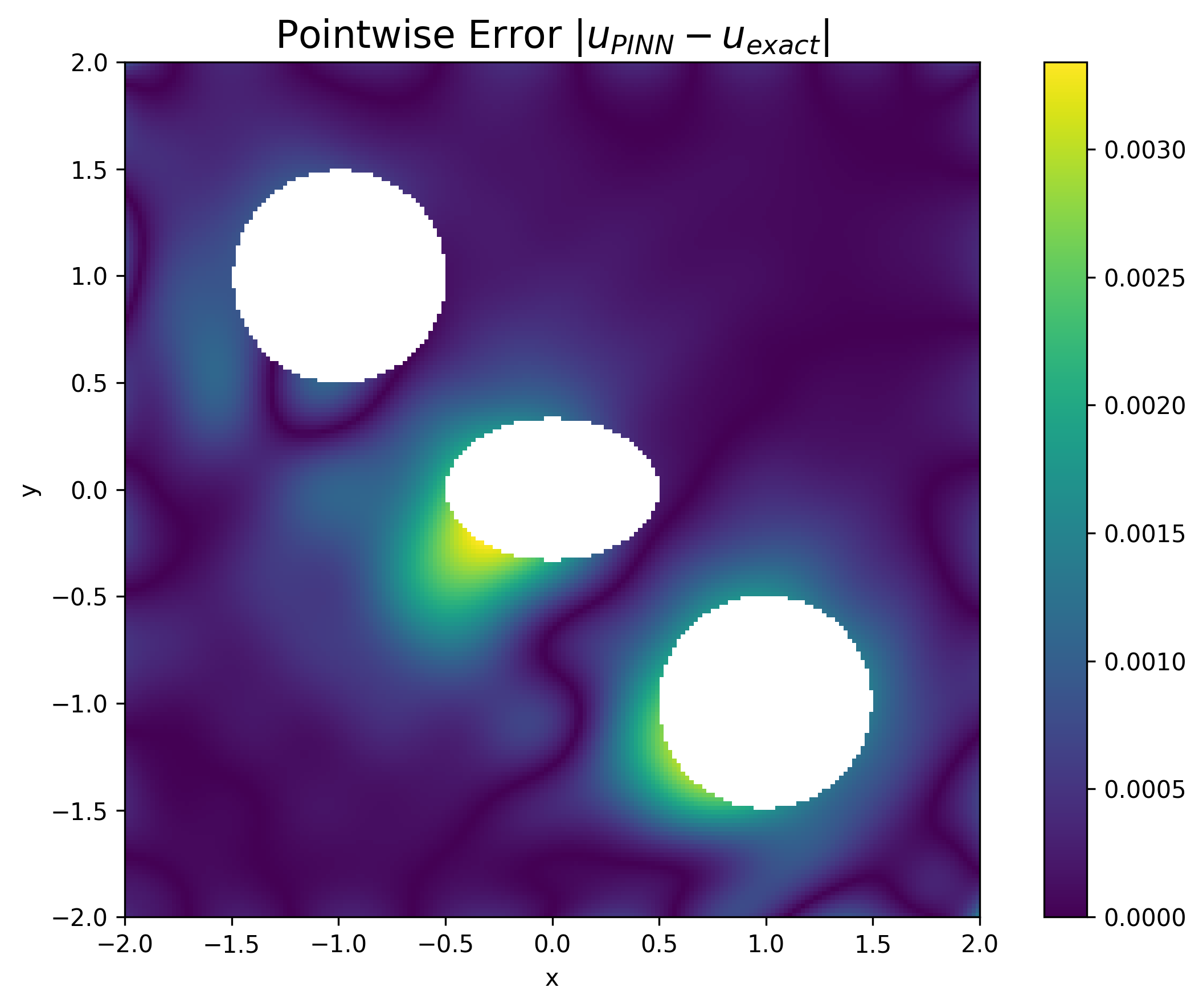}}\\
    \setcounter {subfigure} 0(b.1){
    \includegraphics[scale=0.24]{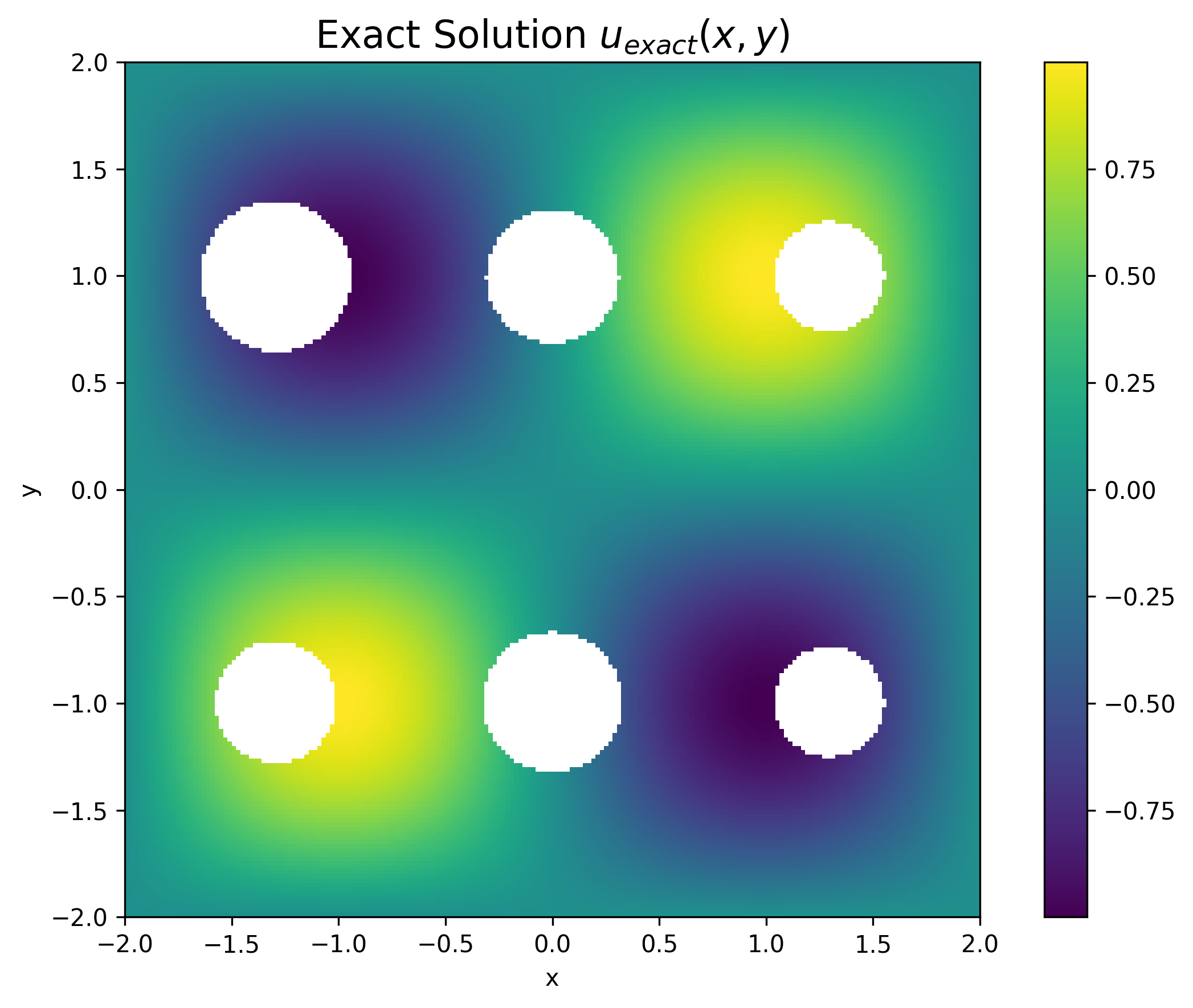}}
    \setcounter {subfigure} 0(b.2){
    \includegraphics[scale=0.24]{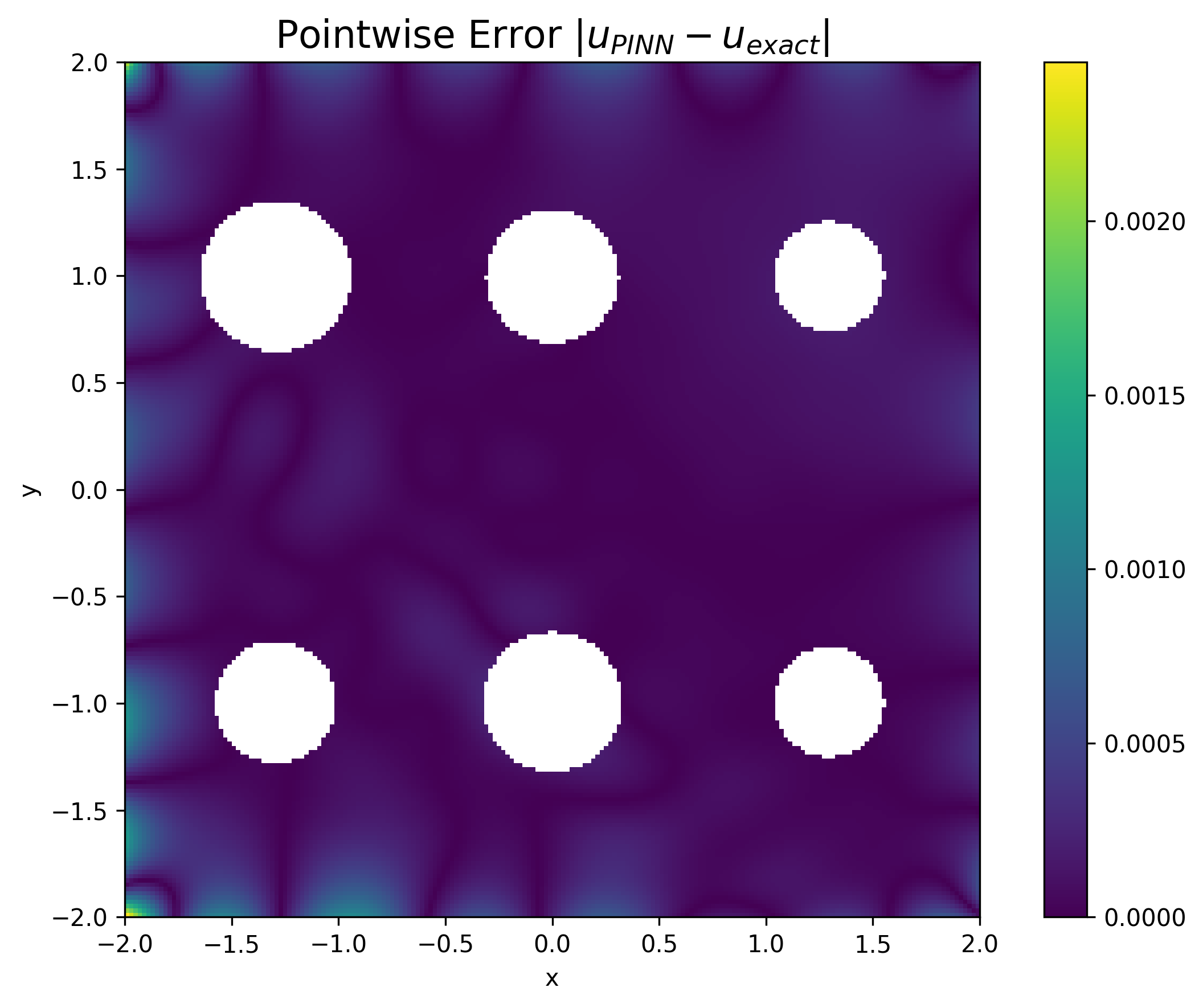}}
    \setcounter {subfigure} 0(b.3){
    \includegraphics[scale=0.24]{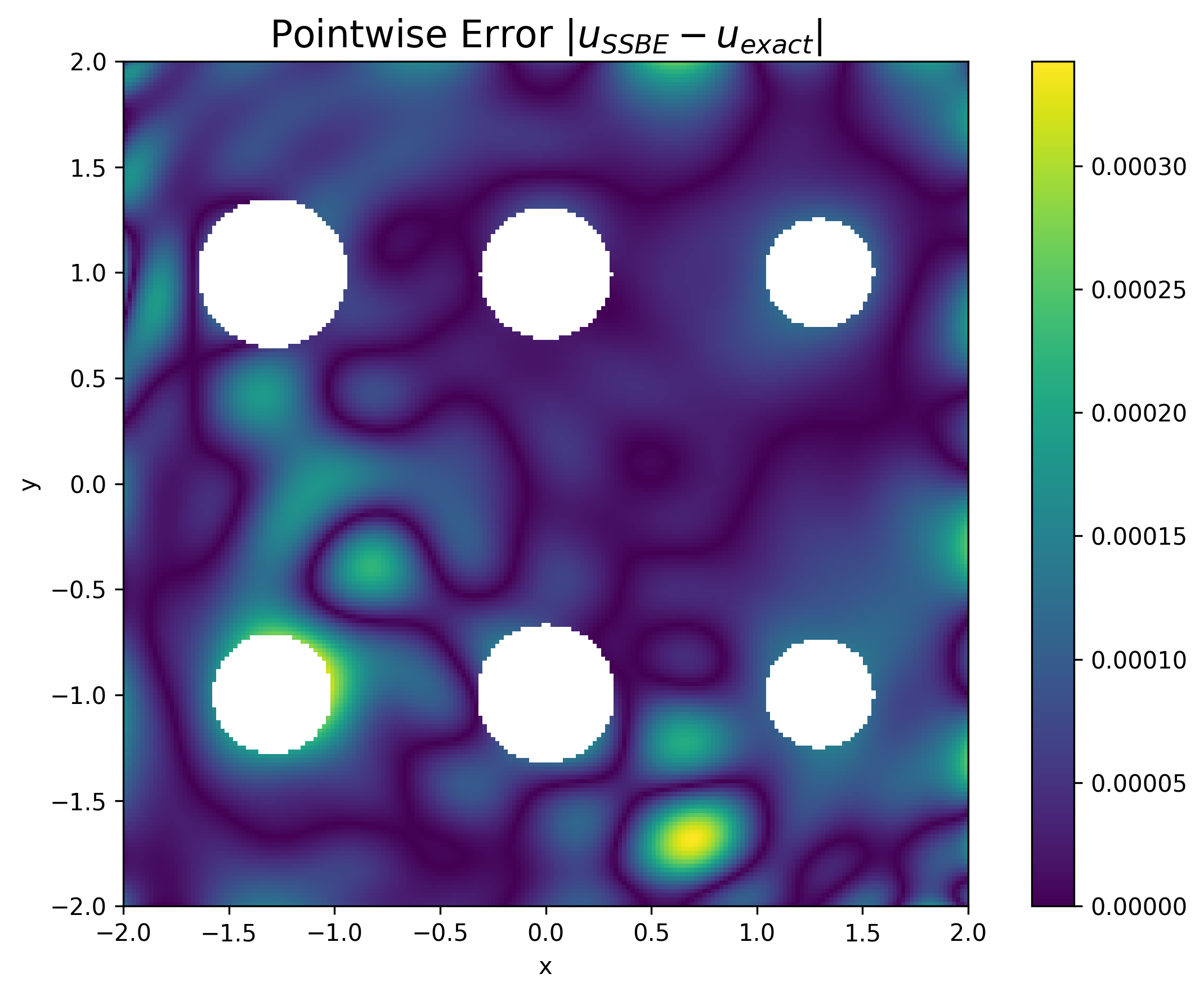}} \\
    \setcounter {subfigure} 0(c.1){
    \includegraphics[scale=0.24]{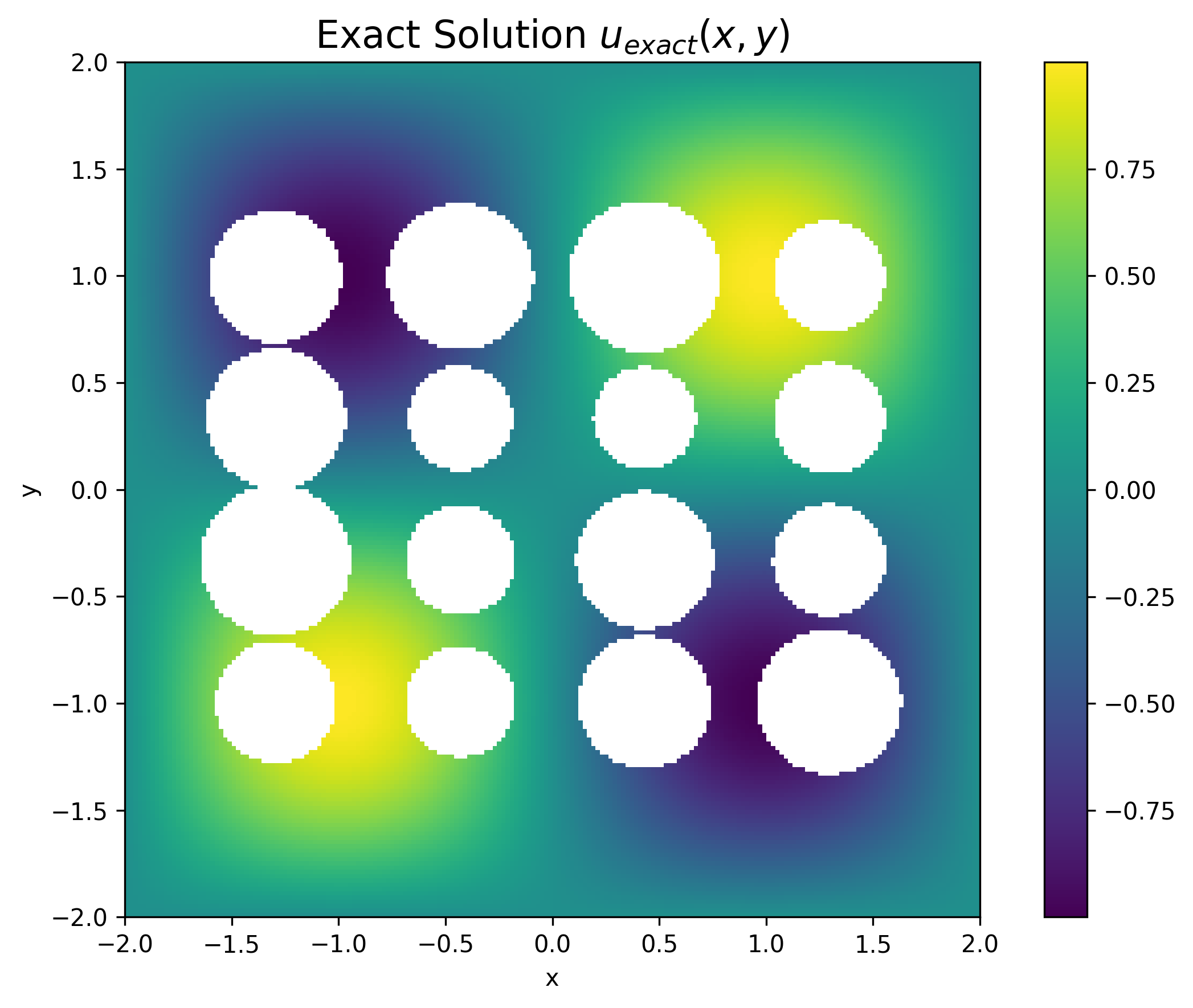}}
    \setcounter {subfigure} 0(c.2){
    \includegraphics[scale=0.24]{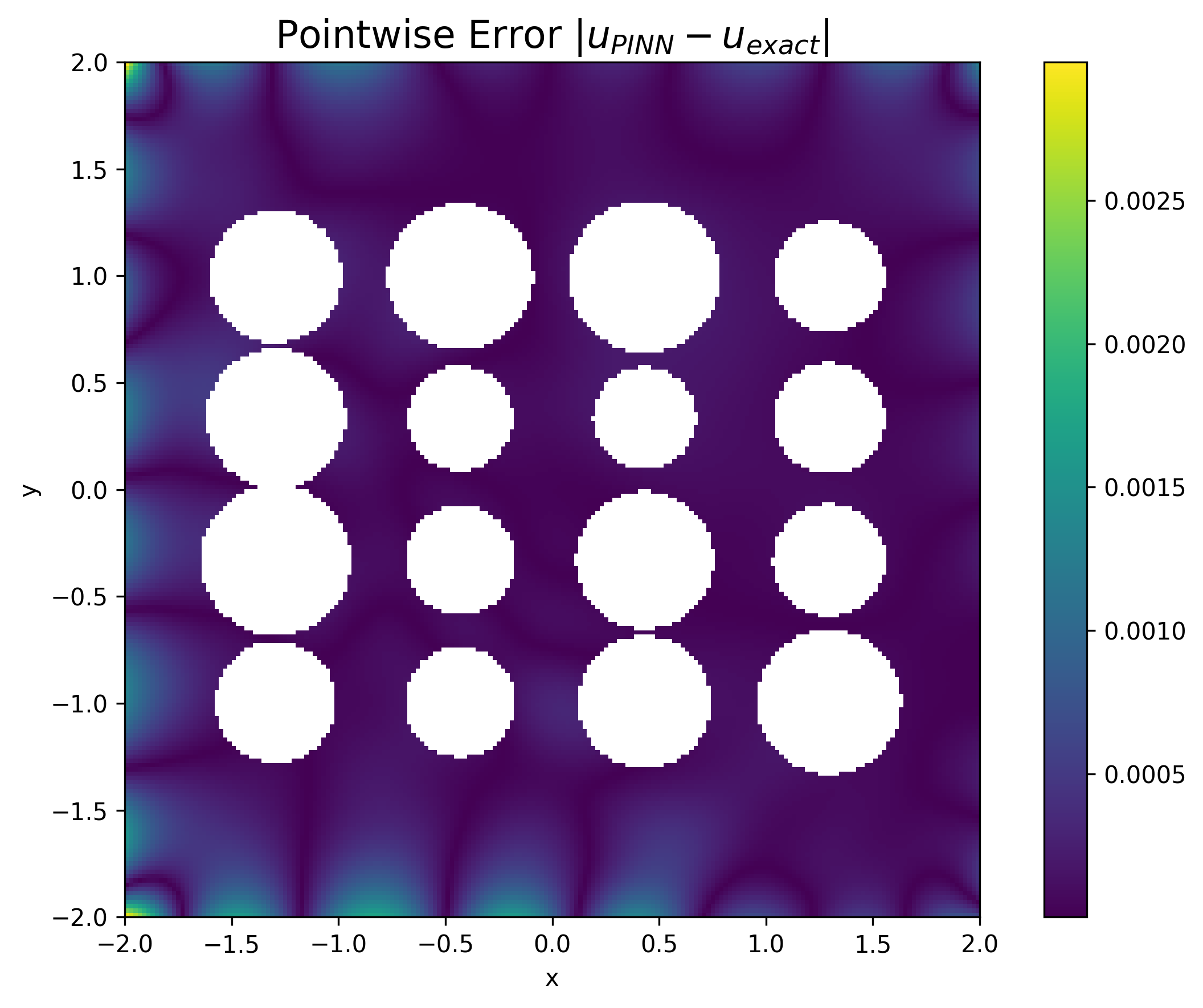}}
    \setcounter {subfigure} 0(c.3){
    \includegraphics[scale=0.24]{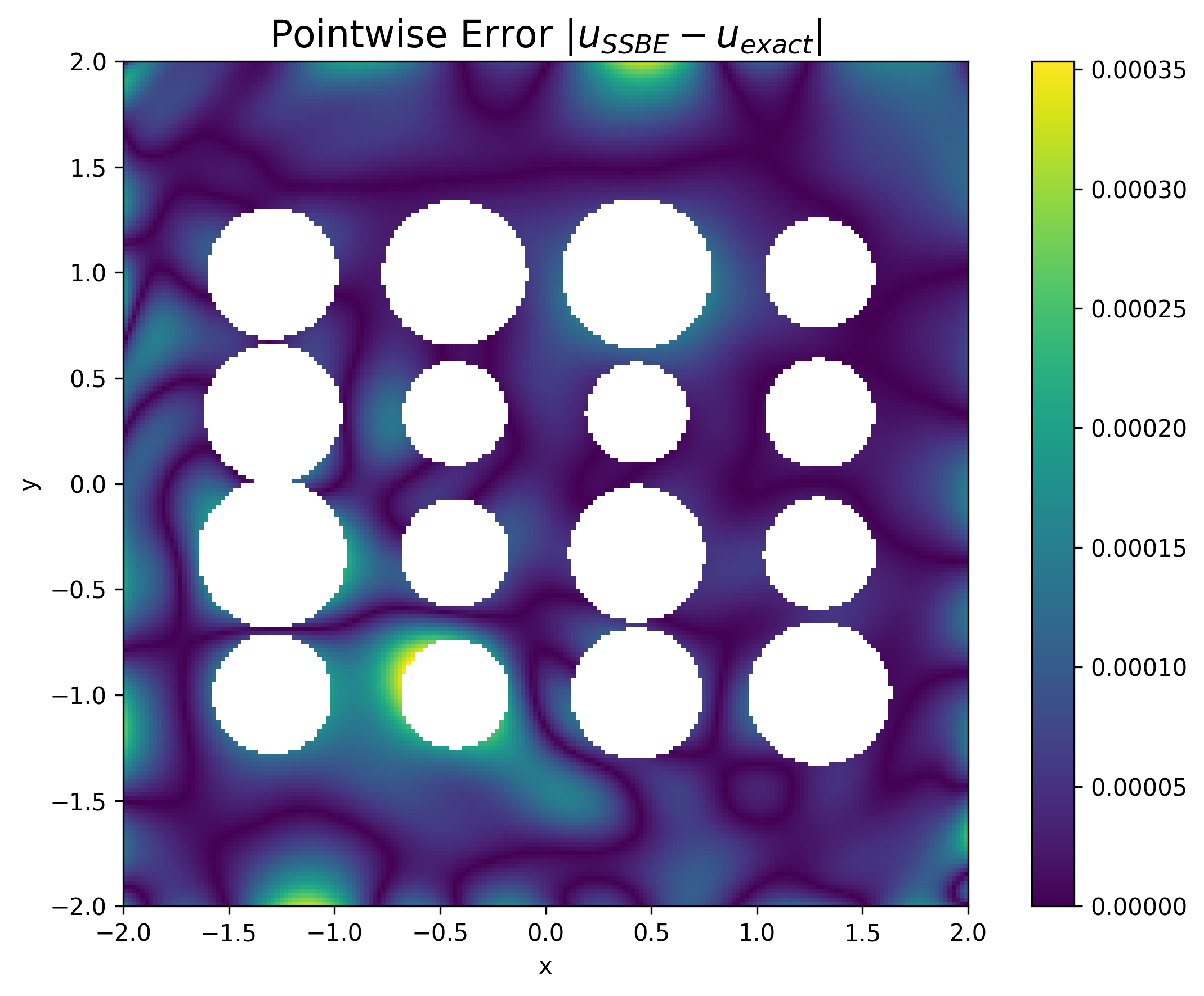}} 
    
\caption{ 
\textcolor{black}{\textbf{Porous medium.} For each domain ((a) case1, (b) case2, and (c) case3), the first column shows the exact solution, the second column shows the pointwise error of the standard PINN, and the third column shows the pointwise error of the proposed SSBE-PINN.}}
    \label{Fig.porous medium}
\end{figure}
\end{color}
\section{Conclusion} \label{sec:conclusion}
In this paper, we present and analyze a novel method, \textit{Sobolev-Stable Boundary Enforcement} (SSBE), designed to address the limitations of Physics-Informed Neural Networks (PINNs) in achieving robust and accurate \( H^1 \) convergence when solving partial differential equations (PDEs). Through rigorous theoretical analysis and extensive numerical experimentation, we establish several key findings:

Firstly, conventional PINNs often fail to achieve accurate convergence in the \( H^1 \) norm due to inadequately defined boundary loss terms, especially when minor boundary inaccuracies occur. Our analysis in Section~\ref{sec::CounterExample} illustrates this critical limitation explicitly, highlighting the need for improved boundary enforcement strategies.

Secondly, the proposed SSBE method significantly enhances convergence properties by explicitly integrating boundary regularity conditions into the neural network training process. By redefining the boundary loss in the Sobolev space, SSBE effectively mitigates numerical instabilities and improves both solution accuracy and smoothness. Theoretically, we demonstrate the robustness of SSBE through detailed stability analyses in Section~\ref{sec::Stability}, ensuring that \( H^1 \) errors remain explicitly bounded by the proposed loss function.

Thirdly, our quantitative generalization error estimates in Section~\ref{sec::GeneralizationBound} rigorously quantify the relationship between population and empirical losses, providing clear insights into the method’s reliability under practical, finite-sample scenarios. The derived \textit{a priori} error bounds indicate that the SSBE approach maintains robust convergence even when data samples are limited, thereby extending the method's applicability to real-world complex PDE problems.

Moreover, extensive numerical experiments presented in Section~\ref{sec:example} validate our theoretical predictions. SSBE consistently achieves lower relative \( L^2 \) and \( H^1 \) errors compared to conventional PINNs across various neural network architectures and PDE types, including Poisson, heat, and nonlinear elliptic equations and \textcolor{black}{over irregular domains such as intrinsic boundary parametrization and porous medium}. These results demonstrate that our method significantly enhances both the accuracy and stability of PDE solutions.

In summary, the SSBE framework represents a substantial advancement in neural network-based PDE solving techniques. By explicitly incorporating boundary regularity constraints, our approach ensures robust and accurate \( H^1 \) convergence, addressing critical limitations of existing PINN methodologies. Future research directions include extending SSBE to more complex PDE systems, exploring its applicability to hyperbolic equations, and further optimizing neural network architectures to maximize computational efficiency and accuracy. Additionally, the exploration of adaptive boundary segmentation methods to enhance the precision and efficiency of gradient computations presents an exciting avenue for future work.
\section*{Acknowledgement}
The authors gratefully acknowledge the reviewers for their valuable and insightful comments that significantly improved this work. This work of T.L. is sponsored by the National Key R\&D Program of China Grant No. 2022YFA1008200 (T. L.), we also thank Shanghai Institute for Mathematics and Interdisciplinary Sciences (SIMIS) for their financial support. This research was funded by SIMIS under grant number SIMIS-ID-2025-ST (T. L.). The work of Y.X. was supported by the Project of Hetao Shenzhen-HKUST Innovation Cooperation Zone HZQB-KCZYB-2020083.
	
\newpage
\appendix

\section{Embedding Theorem: $W^{k,p}(\partial\Omega)\to W^{1,q}(\Omega)$}\label{sec::Extension}
\setcounter{equation}{0}
\setcounter{lemma}{0}
\setcounter{prop}{0}
\setcounter{remark}{0}
\setcounter{theorem}{0}
\renewcommand{\theequation}{A\arabic{equation}}
\renewcommand{\thelemma}{A\arabic{lemma}}
\renewcommand{\theprop}{A\arabic{prop}}
\renewcommand{\thetheorem}{A\arabic{theorem}}
\renewcommand{\theremark}{A\arabic{remark}}

\begin{lemma}\label{lem::ExtensionCtrol}
    For a given function $G(x,t) \in L_{\mathrm{loc}}^{1}\left(\sR^{d-1}\times \sR^+\right), g \in L^{p}\left(\sR^{d-1}\right)$ with $ p>1, d \geq 2$ satisfying
    \begin{equation}\label{eq::ExtensionDimension}
        \Abs{G(x, t)} \leq C_1 \frac{1}{m(\sB^{d}(x, t))}\int_{\sB^{d}(x, t)}\Abs{g(y)} \diff{y}, \quad \text { for }(x, t) \in \sR^{d-1}\times \sR^+,
    \end{equation}
    where $\sB^{d-1}(x,t)$ means the ball in $\sR^{d-1}$ with center at $x$ and radius $t$, then
    \begin{equation}\label{eq::ExtensionControl}        
        \norm{G}_{L^{q}\left(\sR^{d-1}\times\sR^+\right)} \leq C_2\norm{g}_{L^{p}\left(\sR^{d-1}\right)}, \quad \text { where } q=p+\frac{1}{d-1},    
    \end{equation}
    where $C_2$ is a   constant depending on $C_1,\ d,\ p.$
    \end{lemma}
\begin{proof}
    Since the right-hand side of the given inequality~\eqref{eq::ExtensionDimension} is bounded, up to the constant $C$, by the maximal function $\mathcal{M} g(x):=\sup_{x\in B}\frac{1}{\Abs{B}}\int_{B}\abs{g(x)}\diff{x}$, we have
    \begin{equation}\label{eq::ControlInBall}
        \int_{0}^{r}\abs{G(x, t)}^{q} d t \leq Cr(\mathcal{M} g)^{q}(x), \quad \text { for } r>0.
    \end{equation}
    On the other hand, the inequality~\eqref{eq::ExtensionDimension}
    \begin{equation}
        \Abs{G(x, t)} \leq C t^{1-d}\left(\int_{B^{d-1}(x, t)}\Abs{g(y)}^{p} \diff{y}\right)^{1 / p} \leq C t^{1-d}\norm{g}_{L^p(\Omega)},
    \end{equation}
    yields
    \begin{equation}\label{eq::ControlOutBall}
        \int_{r}^{\infty}\Abs{G(x, t)}^{q} d t \leq C\norm{g}_{L^p(\Omega)}^{q} r^{1-(d-1)q}, \quad \text { for } r>0.
    \end{equation}
    Choose $r>0$ such that
    \begin{equation*}
        r(\mathcal{M} g)^{q}(x)=\norm{g}_{L^p(\Omega)}^{q} r^{1-(d-1)q}, \quad \text { i.e., } \quad r=\left(\frac{\norm{g}_{L^p(\Omega)}}{\mathcal{M} g(x)}\right)^{\frac{1}{d-1}}.
    \end{equation*}
    Combining~\eqref{eq::ControlInBall} and~\eqref{eq::ControlOutBall} lead to
    \begin{equation*}
    \int_{0}^{\infty}\abs{G(x, t)}^{q} \diff{t} \leq 2C\norm{g}_{L^p(\Omega)}^{\frac{1}{d-1}}(\mathcal{M} g(x))^{q-\frac{1}{d-1}},
    \end{equation*}
     by letting $q = p+\frac{1}{d-1}$and Fubini's theorem along with boundedness of the maximal operator in $L^{p}, p>1$ i.e. $\norm{\mathcal{M}g}_{L^p(\Omega)}\le C\norm{g}_{L^p(\Omega)}$ yields
    \begin{equation*}
        \int_{\sR_{+}^{n+1}}\Abs{G(x, t)}^{q} \diff{t} \diff{x} \leq C\norm{g}_{L^p(\Omega)}^{1 / n}\norm{g}_{L^p(\Omega)}^{p}=C\norm{g}_{p}^{q} .
    \end{equation*}
    
\end{proof}

\begin{prop}\label{lem::ExtensionEmbedding}
    For $d \geq 2$ and $p>1$, there is a bounded linear extension operator
    \begin{equation*}
        E: W^{1, p}\left(\sR^{d-1}\right) \rightarrow W^{1, q}\left(\sR^{d-1}\times \sR^+\right), \quad \text { where } q=p+\frac{1}{d-1} .
    \end{equation*}
\end{prop}
\begin{proof}
     Let $\varphi \in C_0^{\infty}(\sB^{d-1})$ satisfying $\int_{\sB^{d-1}}\varphi\diff{x} =1,$ $\sB^{d-1}\subset \sR^{d-1}$ is the unit ball, and define $\varphi_t=t^{1-d}\varphi(x/t).$ For $g\in L_{\mathrm{loc}}^1(\sR^{d-1}),$ define a linear operator $E$ as follow:
     \begin{equation}\label{eq::DefExtension}
         (Eg)(x,t)= (f*\varphi_t)(x)= \int_{\sB^{d-1}}g(x-ty)\varphi(y)\diff{y},\ (x,t)\in \sR^{d-1}\times\sR^+,
     \end{equation}
     proposition of convolution guarantees that: $(Eg)\in C_c^{\infty}(\sR^{d-1}\times\sR^+)$ , and by change of variables
     \begin{equation}\label{eq::Control4Extension}
         \abs{(Eg)(x,t)}= \abs{\int_{\sB^{d-1}(x,t)}g(y)\varphi_t(x-y)\D y}\le C_3\frac{1}{m(\sB^{d-1}(x,t))}\int_{\sB^{d-1}(x,t)}\abs{g(y)}\D y,
     \end{equation}
     where $C_3$ is a constant that depends only on $\varphi$ and $d$.
     
     By Lemma~\ref{lem::ExtensionCtrol}, for any $g\in L^{p}(\sR^{d-1}),$ there is 
     \begin{equation}
         \norm{Eg}_{L^q(\sR^{d-1}\times\sR^+)}\le C \norm{g}_{L^p(\sR^{d-1})}.
     \end{equation}
     Moreover, for any $g\in W^{1,p}(\sR^{d-1}),$ there is
     \begin{equation}\label{eq::ZeroOrderControl}
         \Abs{\nabla Eg(x,t)}=\Abs{\nabla_{(x,t)}(Eg)(x,t)}\le C\frac{1}{m(\sB^{d-1}(x,t))} \int_{\sB^{d-1}(x,t)}\abs{\nabla g(y)}\D y.
     \end{equation}
     This is due to
     \begin{equation}\label{eq::FirstOrderControl}
         \begin{aligned}
              \Abs{D_{\alpha}(Eg)(x,t)}&= \abs{\int_{\sB^{d-1}}(D_{\alpha} g)(x-ty)\varphi(y)\D y} = \abs{\int_{\sB^{d-1}(x,t)}(D_{\alpha} g)(y)\varphi_t(x-y)\D y}\\ 
              &\le C\frac{1}{m(\sB^{d-1}(x,t))}\int_{\sB^{d-1}(x,t)}\abs{(D_{\alpha}g)(y)}\D y,\quad \text{for $\alpha\in \{1,\ldots,d-1\}$}, ~\\
              \abs{D_t (Eg)(x,t)}&= \abs{\int_{\sB^{d-1}}(\nabla g)(x-ty)\cdot (-y)\varphi(y)\D y} \le C\frac{1}{m(\sB^{d-1}(x,t))}\int_{\sB^n(x,t)}\abs{\nabla g(y)}\D y\\ &\le C\frac{1}{m(\sB^{d-1}(x,t))}\int_{\sB^{d-1}(x,t)}\sum_{i}\abs{(D_{i}g)(y)}\D y,
         \end{aligned}
     \end{equation}
     since $\Abs{y}\le 1$ and $|\cdot|$ means the $l^2$-norm of a function vector. Combining equations~\eqref{eq::ZeroOrderControl} and~\eqref{eq::FirstOrderControl} and Lemma~\ref{lem::ExtensionCtrol}
     \begin{equation}\begin{aligned}
          \norm{ Eg(x,t)}_{W^{1,q}(\sR^{d-1}\times\sR^+)}^q&\le C \left(\norm{ g}_{L^p(\sR^{d-1})}^q+\sum_{\alpha}\norm{D_{\alpha} g}_{L^p(\sR^{d-1})}^q +\norm{\sum_{\alpha}\abs{D_{\alpha} g}}_{L^p(\sR^{d-1})}^q\right)\le C\norm{g}_{W^{1,p}(\sR^{d-1})}^q,
     \end{aligned}
     \end{equation}
     where $C$ a constant different from line to line. And we complete the proof.
\end{proof}

\begin{theorem}[Extension Theorem: $W^{1,p}(\partial\Omega)\to W^{1,q}(\Omega)$]\label{thm::Extension}
    Suppose $\Omega\in \sR^{d}$ belongs to the class $\fC^{0,1}$ with $d\ge 2$, and there is an $g$ defined on $\partial \Omega$ such that $g\in W^{1,p}(\partial \Omega)$, then there are a bounded linear operator
    \begin{equation*}
        T: W^{1, p}\left(\partial \Omega\right) \rightarrow W^{1, q} \left(\Omega\right), \quad \text { where } q=p+\frac{1}{d-1} .
    \end{equation*}
    such that $Tg \lfloor_{\partial \Omega} = g$ and a constant $C$ depending only on $\Omega, d$ such that
    \begin{equation}
        \norm{Tg}_{W^{1,p}(\Omega)}\le C\norm{g}_{W^{1,q}(\partial\Omega)}.
    \end{equation}
\end{theorem}
\begin{proof}[Proof of Theorem~\ref{thm::Extension}]
    In the following proof, we abuse the notation of $x^0\in \partial\Omega$ as $\fF_{r}^{-1} x^0=:\fX_{r}^{0}$, and $g(x^0)$ as $_{r}g(x_r').$
    
    For given $x^0\in \partial \Omega,$ and suppose firstly $\partial \Omega$ is flat near $x^0,$ lying in the plane with $x_{d}=0,$ then there is a ball $B$ with radius $r$ center at $x^0$ such that:
    \begin{equation*}
        \begin{cases}
                B^+:=B\bigcap \{x_{d}\ge 0\} \subseteq \overline{\Omega},\\
                B^-:=B\bigcap \{x_{d}< 0\} \subseteq \sR^{d}-\overline{\Omega}.
        \end{cases}
    \end{equation*}
    Choose $\zeta \in C_0^{\infty}(B\bigcap \{x_{n+1}= 0\}),$ we then have that $\zeta g\in W^{1,p}(B\bigcap \{x_{n+1}= 0\}).$
    
    Consider \begin{equation}
        T(\zeta g)(x,t) =\int_{\sB^n}(\zeta g)(x-ty)\varphi(y)\D y, 
    \end{equation}
    then by Proposition~\ref{lem::ExtensionEmbedding}, we have
    \begin{equation}
        \norm{T(\zeta g)}_{W^{1,q}(B^+)}\le \norm{T(\zeta g)}_{W^{1,q}(\sR^d\times\sR^+)}\le C\norm{\zeta g}_{W^{1,p}(B\bigcap \{x_{n+1}= 0\})}.
    \end{equation}
    
    Then consider the case that $\partial \Omega$ is not necessarily flat near $x^0$, with $\fF_{r}x^0$ belong to some $\Gamma_r$. Note that there exists a $C^1-$mapping $\Phi_r$ with inverse $\Psi_r$ such that $\Phi_r$ straightens out $\partial \Omega$ near $\fF_{r}x^0.$
    
    We may write $\fY_r=\Phi_r(\fX_{r})$ and $\fX_{r}=\Psi_r(\fY_r)$, $\Tilde{g}(\fY_r) = g(\Psi_r(\fY_r)),$ choose the $\Phi_r(U_r)$ under straighten coordinate and $\Tilde{\zeta} \in C_0^{\infty}(B\bigcap \{y_{r_d}=0\}),$ then by our above discussion
    \begin{equation}
        \norm{E(\Tilde{\zeta}\Tilde{g})}_{W^{1,q}(\Phi(U_{r}^{+}))}\le C\norm{\Tilde{\zeta}\Tilde{g}}_{W^{1,p}(\Phi_r(U_{r}^{+}) \bigcap \{y_{d}=0\})},
    \end{equation}
     converting back to $\fX_{r}-$coordinate, with $\zeta(\fX_{r}):=\tilde{\zeta}(\Phi(\fX_{r}))$, there is
    \begin{equation}
        \norm{E(\zeta g)}_{W^{1,q}(U_{r}^{+})}\le C\norm{\zeta g}_{W^{1,p}(\Gamma_{r})}.
    \end{equation}
    Note that $\{\Gamma_r\}_{r=1}^{m}$ is open cover of $\partial\Omega$ and consider $\{\zeta_r\}_{r=i}^{m}$ the associated partition of unity of $\{\Gamma_r\}_{r=1}^{m}$ on $\partial \Omega$. Define the extension operator as $$E(\zeta_r g)(x)= \int_{\sB^n}(\zeta_r g)(x-ty)\varphi(y)\D y,$$ on $\sR^{d-1}\times\sR^+.$ By proposition~\ref{lem::ExtensionEmbedding},
    \begin{equation}
        \norm{E(\zeta_r g)}_{W^{1,q}(U_r^+)}^p\le C\norm{\zeta_r g}_{W^{1,p}(\partial \Gamma_r)}^p,
    \end{equation}
    we thus have
    \begin{equation}\label{eq::EmbeddingControl}
         \norm{E(\zeta_r g)}_{W^{1,q}(U_r^+)}^p\le C\sum_{r=1}^{l}\norm{\zeta_r g}_{W^{1,p}(\partial \Gamma_r)}^p.
    \end{equation}
    Note that $\{\zeta_r\}_{r=1}^{m}$ for $r\in\{1\ldots,m\}$ satisfy
    
         $$1.\ \sum_{r=1}^{l}\zeta_r =1,\  \text{on $\partial \Omega$},\quad 2.\ 0\le \zeta_r\le 1,\  \text{on $\partial \Omega$},\quad 3.\ \zeta_r \in C_c^{\infty}(\partial \Omega), $$
         consider then
         \begin{equation}
             \norm{\sum_{r=1}^{l}\zeta_r g}^p_{W^{1,p}(\partial \Omega)} =\norm{\sum_{r=1}^{,}\zeta_r g}^p_{L^{p}(\partial \Omega)}+\sum_{\alpha=1}^{d-1}\norm{\sum_{r=1}^{l}(\zeta_rD_{\alpha}g+g D_{\alpha}\zeta_r)}^p_{L^{p}(\partial \Omega)}, 
         \end{equation}
         and we exactly have
         \begin{equation}\label{eq::DeompositionInequality}
              \norm{g}^p_{L^{p}(\partial \Omega)}=\norm{\sum_{r=1}^{l}\zeta_r g}^p_{L^{p}(\partial \Omega)} \ge\sum_{r=1}^{l}\norm{\zeta_r g}_{L^p(\Gamma_r)}^p, 
         \end{equation}
    where the last inequality results from  $\sum_{r=1}^{l} \zeta_r^p(x)\le 1$ for any $x\in \partial \Omega.$
    
    For the derivative part:
    \begin{equation}
    \begin{aligned}
         \norm{\sum_{r=1}^{l}(\zeta_rD_{\alpha}g+gD_{\alpha}\zeta_r)}^p_{L^{p}(\partial \Omega)}&=\norm{\sum_{r=1}^{l}\zeta_rD_{\alpha}g+g\sum_{r=1}^{l}D_{\alpha}\zeta_r}^p_{L^{p}(\partial \Omega)}\\&=\norm{\sum_{r=1}^{l}\zeta_rD_{\alpha}g}^p_{L^{p}(\partial \Omega)}=\norm{D_{\alpha}g}^p_{L^{p}(\partial \Omega)},
    \end{aligned}
    \end{equation}
    where the second equality above results from $\sum_{r=1}^{l}\zeta_r=1$ on $\partial \Omega.$ And by the inequality~\eqref{eq::DeompositionInequality}
    \begin{equation}\label{eq::FirstOrder4Pou}
        \begin{aligned}
        \sum_{r=1}^m\norm{\zeta_r D_{\alpha}g}_{L^p(\Gamma_r)}^p&\le \norm{\sum_{r=1}^{l}\zeta_r D_{\alpha}g}^p_{L^{p}(\partial \Omega)},
        \\
         \sum_{r=1}^m\norm{gD_{\alpha}\zeta_r }_{L^p(\Gamma_r)}^p&\le C\norm{g}_{L^p(\Omega)}^p,
        \end{aligned}
    \end{equation}
    where $C$ is a constant depends on $\{\zeta_r\}_{r=1}^m.$
    
    Combining the~\eqref{eq::DeompositionInequality} and~\eqref{eq::FirstOrder4Pou}, we thus have:
    \begin{equation}\label{eq::PartialToPartial}
        \sum_{r=1}^{l}\norm{\zeta_r g}_{W^{1,p}(\Gamma_r)}^p\le (1+C)\norm{ g}^p_{W^{1,p}(\partial \Omega)}.
    \end{equation}
    Notice then $\{U_r\}_{r=1}^{m}\subseteq \sR^{d}$ covers $\partial \Omega$ and we hence can choose a $U_0 \subseteq \Omega$ such that $\Omega \subseteq \bigcup_{r=0}^{m}U_r,$ consider then associated partition of unity as $\{\xi_r\}_{r=0}^{m}$ on $\Omega.$ Consider $\xi_rE(\zeta_r g)$ for $r\in \{1,\dots,l\}$ and define $Eg = 0$ and $\zeta_0=0$ on $U_0.$ Thus
    \begin{equation}\label{eq::ZeroOrderControld+1}
        \norm{\sum_{r=0}^{m}\xi_rE(\zeta_r g)}_{L^q(\Omega)}^p\le \norm{\sum_{r=1}^{l}E(\zeta_r g)}_{L^q(U_r^+)}^p\le \sum_{r=1}^{l}\norm{E(\zeta_r g)}_{L^q(U_r^+)}^p ,
    \end{equation}
    and that
    \begin{equation}\label{eq::FirstOrderControld+1}
        \begin{aligned}
             \norm{D_{\alpha}\sum_{r=0}^{m}\xi_rE(\zeta_r g)}_{L^q(\Omega)}^p&= \norm{\sum_{r=0}^{m}\xi_rD_{\alpha}E(\zeta_r g)}_{L^q(\Omega)}^p\le \sum_{j=1}^{N}\norm{D_{\alpha}E(\zeta_r g)}_{L^q(U_r^+)}^p,
        \end{aligned}
    \end{equation}
    combining~\eqref{eq::ZeroOrderControld+1} and~\eqref{eq::FirstOrderControld+1}, 
    \begin{equation}\label{eq::DomainInequality}
        \norm{\sum_{r=0}^{l}\xi_rE(\zeta_r g)}_{W^{1,q}(\Omega)}^p\le C\sum_{r=1}^{l}\norm{E(\zeta_r g)}_{W^{1,q}(U_r^+)}^p,
    \end{equation}
    where $C$ is a constant depends on $N$ only.
    
    By equations~\eqref{eq::EmbeddingControl},~\eqref{eq::PartialToPartial} and~\eqref{eq::DomainInequality}, we then have
    \begin{equation}
        \norm{\sum_{r=0}^{m}\xi_rE(\zeta_r g)}_{W^{1,q}(\Omega)}^p \le C\norm{ g}^p_{W^{1,p}(\partial \Omega)}.    
    \end{equation}
    By defining
    \begin{equation*}
        Tg=\sum_{r=0}^{m}\xi_rE(\zeta_r g) \in W^{1,q}(\Omega),
    \end{equation*}
    we have $\norm{Eg}_{W^{1,q}(\Omega)}\le C\norm{g}_{W^{1,p}(\partial\Omega)}$ and $Ef\lfloor_{\partial \Omega} =g$. Thus we complete the proof.
\end{proof}

\begin{remark}\label{rmk::EmbbingH1partialToH1}
    A natural consequence is that suppose $\Omega\in \sR^{d}$ belongs to the class $\fC^{0,1}$, and there is an $g$ defined on $\partial \Omega$ such that $g\in H^1(\partial \Omega)$, then there is a bounded linear operator
    \begin{equation*}
        T: H^1\left(\partial \Omega\right) \rightarrow H^1 \left(\Omega\right),
    \end{equation*}
such that $Tg \lfloor_{\partial \Omega} = g$. We obtain this result immediately by letting $p=2$ and the fact that $W^{1,2+\frac{1}{d-1}}(\Omega)$ is a subspace of $H^1(\Omega)$.
\end{remark}

\section{Rademahcer Complexity and Contraction Lemma}
\setcounter{equation}{0}
\setcounter{definition}{0}
\setcounter{prop}{0}
\setcounter{theorem}{0}
\setcounter{lemma}{0}
\renewcommand{\theequation}{B\arabic{equation}}
\renewcommand{\thedefinition}{B\arabic{definition}}
\renewcommand{\thelemma}{B\arabic{lemma}}
\renewcommand{\thetheorem}{B\arabic{theorem}}
\renewcommand{\theremark}{B\arabic{remark}}

\begin{definition}[Rademacher complexity]\label{def::Rademacher}
    Given a set $S=\{z_i\}_{i=1}^{n}$ on a domain $\fZ$ and a real-valued function class $\fF$ defined of $\fZ$. The Rademacher complexity of $\fF$ on $S$ is
    \begin{equation}
        \operatorname{Rad}_{S}(\fF)=\frac{1}{n}\Exp_{\vtau}\left[\sup_{f\in \fF}\sum_{i=1}^n \tau_i f(z_i)\right],
    \end{equation}
    where $\vtau=\{\tau_i\}_{i=1}^{n}$ are independent variable set sampled from Rademahcer distritution, i.e. $\Prob(\tau_i=1)=\Prob(\tau_i=-1)=\frac{1}{2}.$ More generally, for a given vector set $A\subset \sR^d$, define
    \begin{equation}
        \operatorname{Rad}(A)=\frac{1}{n}\Exp_{\vtau}\left[\sup_{\va\in A}\sum_{i=1}^{n}\tau_ia_i\right].
    \end{equation}
\end{definition}

\begin{lemma}[Contration lemma, rephrased from Lemma 26.9 in~\cite{Shwartz2014Understanding}]\label{lem::Contraction}
    For each $i\in \{1,\ldots,n\}$, let $\phi_{i}:\sR\to\sR$ be function with Lipschitz constant $\rho$. For $\va=(a_1,\ldots,a_n)\in \sR^d$, let $\vphi(\va)=(\phi_i(a_i))_{i=1}^{n}\in \sR^d$, and $\vphi\circ A=\{\vphi(\va)\big| \va\in A\}$, then
    \begin{equation}
        \operatorname{Rad}(\vphi\circ A)\le \rho \operatorname{Rad}(A).
    \end{equation}
\end{lemma}

\begin{lemma}[Rademacher complexity for linear predictors, rephrased from Lemma 26.20 in~\cite{Shwartz2014Understanding}]\label{lem::Rademacher4LinearPredictor}
    Let $S=\{x_i\}_{i=1}^{n}\subseteq \sR^d$ and  $\fF:=\{\vu^{\T}x \big| \vu,x\in \sR^d, \Abs{\vu}\le 1\}$. The Rademacher complexity of $\fF$ on $S$ satisfies
    \begin{equation}
        \operatorname{Rad}_{S}(\fF)\le \frac{\max_{i}\Abs{x_i}}{\sqrt{n}}.
    \end{equation}
\end{lemma}

\begin{theorem}[Generalization Gap, rephrased from Theorem 26.5 in~\cite{Shwartz2014Understanding}]\label{thm::Rade&Generalization}
    Suppose that for all $f\in \fF$ and all $z\in \fZ$, there is constant $B>0$ such that $0\le f(z)\le B$. Then for any $\delta\in (0,1)$, with probability at least $1-\delta$ over the choice of $n$ i.i.d random sampled set $S=\{z_i\}_{i=1}^{n}\subset \fZ$, we have 
    \begin{equation}
        \sup_{f\in\fF}\Abs{\frac{1}{n}\sum_{i=1}^{n}f(z_i)-\Exp_{z}f(z)}\le 2\operatorname{Rad}_{S}(\fF)+4B\sqrt{\frac{2\ln (4/\delta)}{n}}.
    \end{equation}
\end{theorem}

\section{Technical Details for Generalization Error Estimate}\label{sec::Detail4GeneralzationError}
\setcounter{equation}{0}
\setcounter{definition}{0}
\setcounter{theorem}{0}
\setcounter{lemma}{0}
\renewcommand{\theequation}{C\arabic{equation}}
\renewcommand{\thedefinition}{C\arabic{definition}}
\renewcommand{\thetheorem}{C\arabic{theorem}}
\renewcommand{\thelemma}{C\arabic{lemma}}

\begin{proof}[Proof of Lemma~\ref{lem::RademacherComplexity4BVP}]
    Firstly we estimate the Rademacher complxity of $\operatorname{Rad}(\fF_{Q})$. 
    
    By defining $\hat{\vw}_k=\frac{\vw_k}{\abs{\vw_k}}$ and $\vtau=\{\tau^{1},\ldots,\tau^{n_{\Omega}}\}$ which drawn from Radmacher distribution, i.e. $\Prob(\tau^i=1)=\Prob(\tau^i=-1)=\frac{1}{2}$ for $i\in \{1,\ldots,n_{\Omega}\}$, then
    \begin{equation}
        \begin{aligned}
            n_{\Omega}\operatorname{Rad}(\fF_{Q})&=\Exp_{\vtau}\left[\sup_{\norm{\theta}_{\fP}<Q}\sum_{i=1}^{n_{\Omega}}\tau^{i}\sum_{k=1}^{m}a_k\left[\vw_k^{\T}A(x^i)\vw_k\sigma''(\vw_k^{\T}x^i)+\hat{\vb}^{\T}(x^i)\vw_k\sigma'(\vw^{\T}x^i)+c(x^i)\sigma(\vw^{\T}x^i)\right]\right]\\
            &\le =\Exp_{\vtau}\left[\sup_{\norm{\theta}_{\fP}<Q}\sum_{i=1}^{n_{\Omega}}\tau^{i}\sum_{k=1}^{m}a_k\left[\vw_k^{\T}A(x^i)\vw_k\sigma''(\vw_k^{\T}x^i)\right]\right]\\
            &~~~+\Exp_{\vtau}\left[\sup_{\norm{\theta}_{\fP}<Q}\sum_{i=1}^{n_{\Omega}}\tau^{i}\sum_{k=1}^{m}a_k\left[\hat{\vb}^{\T}(x^i)\vw_k\sigma'(\vw^{\T}x^i)\right]\right]+\Exp_{\vtau}\left[\sup_{\norm{\theta}_{\fP}<Q}\sum_{i=1}^{n_{\Omega}}\tau^{i}\sum_{k=1}^{m}a_k\left[c(x^i)\sigma(\vw^{\T}x^i)\right]\right]\\
            &=: I_1+I_2+I_3.
        \end{aligned}
    \end{equation}
    Estimate the term $I_1$
    \begin{equation}
        \begin{aligned}
            I_1&=\Exp_{\vtau}\left[\sup_{\norm{\theta}_{\fP}<Q}\sum_{i=1}^{n_{\Omega}}\tau^{i}\sum_{k=1}^{m}a_k\abs{\vw_k}^3\left[\hat{\vw}_k^{\T}A(x^i)\hat{\vw}_k\sigma''(\hat{\vw}_k^{\T}x^i)\right]\right]\\
            &\le \Exp_{\vtau}\left[\sup_{\norm{\theta}_{\fP}<Q}\sum_{k=1}^{m}\Abs{a_k}\abs{\vw_k}^3\Abs{\sum_{i=1}^{n_{\Omega}}\tau^{i}\left[\hat{\vw}_k^{\T}A(x^i)\hat{\vw}_k\sigma''(\hat{\vw}_k^{\T}x^i)\right]}\right]\\
            &\le \Exp_{\vtau}\left[\sup_{\norm{\theta}_{\fP}<Q,\Abs{\vp}=\Abs{\vq}=\Abs{\vu}=1}\sum_{k=1}^{m}\Abs{a_k}\abs{\vw_k}^3\Abs{\sum_{i=1}^{n_{\Omega}}\tau^{i}\left[\vp^{\T}A(x^i)\vq\sigma''(\vu^{\T}x^i)\right]}\right]\\
            &\le Q \Exp_{\vtau}\left[\sup_{\Abs{\vp}=\Abs{\vq}=\Abs{\vu}=1}\sum_{\alpha,\beta=1}^{d}\Abs{\vp_{\alpha}}\Abs{\vq_{\beta}}\Abs{\sum_{i=1}^{n_{\Omega}}\tau^{i}\left[A_{\alpha\beta}(x^i)\sigma''(\vu^{\T}x^i)\right]}\right]\\
            &\le Q\sum_{\alpha,\beta=1}^{d}\Exp_{\vtau}\left[\sup_{\Abs{\vu}\le 1}\Abs{\sum_{i=1}^{n_{\Omega}}\tau^{i}\left[A_{\alpha\beta}(x^i)\sigma''(\vu^{\T}x^i)\right]}\right]\\
            &\le Q\sum_{\alpha,\beta=1}^{d}\left(\Exp_{\vtau}\left[\sup_{\Abs{\vu}\le 1}\sum_{i=1}^{n_{\Omega}}\tau^{i}\left[A_{\alpha\beta}(x^i)\sigma''(\vu^{\T}x^i)\right]\right]+\Exp_{\vtau}\left[\sup_{\Abs{\vu}\le 1}\sum_{i=1}^{n_{\Omega}}-\tau^{i}\left[A_{\alpha\beta}(x^i)\sigma''(\vu^{\T}x^i)\right]\right]\right)\\
            &=2Q\sum_{\alpha,\beta=1}^{d}\Exp_{\vtau}\left[\sup_{\Abs{\vu}\le 1}\sum_{i=1}^{n_{\Omega}}\tau^{i}\left[A_{\alpha\beta}(x^i)\sigma''(\vu^{\T}x^i)\right]\right],
        \end{aligned}
    \end{equation}
    by applying Lemma~\ref{lem::Contraction} and Lemma~\ref{lem::Rademacher4LinearPredictor} with $\psi^i(y^i)=A_{\alpha\beta}(x^i)\sigma''(y^i)$ with $i\in \{1,\ldots, n_{\Omega}\}$ whose Lipschitz constant is $M$ for $\alpha,\beta\in \{1,\ldots,d\}$, we immediately get
    \begin{equation}
        \Exp_{\vtau}\left[\sup_{\Abs{\vu}\le 1}\sum_{i=1}^{n_{\Omega}}\tau^{i}\left[A_{\alpha\beta}(x^i)\sigma''(\vu^{\T}x^i)\right]\right]\le M \Exp_{\vtau}\left[\sup_{\Abs{\vu}\le 1}\sum_{i=1}^{n_{\Omega}}\tau^{i}\vu^{\T}x^i\right]\le M\sqrt{n_{\Omega}},
    \end{equation}
    hence
    \begin{equation}\label{eq::EstimateofI_1}
        I_1\le 2Q\sum_{\alpha,\beta=1}^{d}\Exp_{\vtau}\left[\sup_{\Abs{\vu}\le 1}\sum_{i=1}^{n_{\Omega}}\tau^{i}\left[A_{\alpha\beta}(x^i)\sigma''(\vu^{\T}x^i)\right]\right]\le 2MQd^2\sqrt{n_{\Omega}}.
    \end{equation}
    Similarly, since $\sigma'(z)=\frac{1}{2}z\sigma''(z)$ and $\sigma'(z)=\frac{1}{6}z^2\sigma''(z)$, we can verify that
    \begin{equation}\label{eq::EstimateofI_2&I_3}
        \begin{aligned}
            I_2&\le MQd\sqrt{n_{\Omega}},\\
            I_3&\le \frac{1}{3}MQ\sqrt{n_{\Omega}}.
        \end{aligned}
    \end{equation}
    combing~\eqref{eq::EstimateofI_1} and~\eqref{eq::EstimateofI_2&I_3} together, we have
    \begin{equation}\label{eq::RademacherfF_Q}
        \operatorname{Rad}(\fF_{Q}) \le \frac{4MQd^2}{\sqrt{n_{\Omega}}}.
    \end{equation}

    Then we estimate the Rademacher complexity of $\operatorname{Rad}({_{r}\fG_{Q}})$ and $\operatorname{Rad}(D_{x_{r_{\alpha}}} {_{r}\fG_{Q}})$ respectively.

    Note that by the definition of $x_r$ and $_{r}\vw_k$, for $i\in\{ 1,\ldots, n_r\}$, there is $\tilde{x}^i\in \partial\Omega$ such that $\vw_{k}^{\T}\tilde{x}^i={_{r}\vw_{k}^{\T}}(x_{r}^{i}-\fB_{r})$, hence
    \begin{equation}
        g(x_{r}^{i},\theta)=a_k\sigma(\vw_{k}^{\T}\tilde{x}^{i}).
    \end{equation}
    combining the estimates of $I_3$, we immediately obtain that
    \begin{equation}\label{eq::Rademacer_rfG_Q}
        \operatorname{Rad}({_{r}\fG_{Q}})\le \frac{1}{3}\frac{Q}{\sqrt{n_r}}.
    \end{equation}
    For $D_{x_{r_{\alpha}}} {_{r}\fG_{Q}}$, we have that for $i\in \{1,\ldots,n_r\}$
    \begin{equation}
        D_{x_{r_{\alpha}}} g({x'_{r}}^{i},\theta)=\sum_{k=1}^{m}a_k({_{r}\vw_{k,\alpha}}+{_{r}\vw_{k,d}}D_{x_{r_{\alpha}}}\gamma_{r}({x'_{r}}^{i}))\sigma'(\vw_{k}^{\T}\tilde{x}^{i}),
    \end{equation}
    hence
    \begin{equation}
        \begin{aligned}
            n_r\operatorname{Rad}(D_{x_{r_{\alpha}}} {_{r}\fG_{Q}})&=\Exp_{\vtau}\left[\sup_{\norm{\theta}_{\fP}<Q}\sum_{i=1}^{n_r}\tau^i\sum_{k=1}^{m}a_k({_{r}\vw_{k,\alpha}}+{_{r}\vw_{k,d}}D_{x_{r_{\alpha}}}\gamma_{r}({x'_{r}}^{i}))\sigma'(\vw_{k}^{\T}\tilde{x}^{i})\right]\\
            &\le \Exp_{\vtau}\left[\sup_{\norm{\theta}_{\fP}<Q}\sum_{i=1}^{n_r}\tau^i\sum_{k=1}^{m}a_k{_{r}\vw_{k,\alpha}}\sigma'(\vw_{k}^{\T}\tilde{x}^{i})\right]\\
            &~~~+\Exp_{\vtau}\left[\sup_{\norm{\theta}_{\fP}<Q}\sum_{i=1}^{n_r}\tau^i\sum_{k=1}^{m}a_k{_{r}\vw_{k,d}}D_{x_{r_{\alpha}}}\gamma_{r}({x'_{r}}^{i})\sigma'(\vw_{k}^{\T}\tilde{x}^{i})\right]\\
            &=: \tilde{I}_1+\tilde{I}_2.
        \end{aligned}
    \end{equation}
    Note that $\fQ_r^{-1}$ is a orthogonal mapping, thus for all $r\in\{1,\ldots,l\}$, $k=\{1,\ldots,m\}$ and $\alpha\in\{1,\ldots,d-1\}$
    \begin{equation}
        \max_{\alpha,d}\{\Abs{{_{r}\vw_{k,\alpha}}},\Abs{{_{r}\vw_{k,d}}}\}\le \Abs{{_{r}\vw_{k}}}=\Abs{\vw_k},
    \end{equation}
    combining that $\sigma'(z)=\frac{1}{2}z\sigma''(z)$, for the term $\tilde{I}_1$, we have
    \begin{equation}
        \begin{aligned}
            \tilde{I}_1&\le\Exp_{\vtau}\left[\sup_{\norm{\theta}_{\fP}<Q}\sum_{k=1}^{m}\Abs{a_k{_{r}\vw_{k,\alpha}}}\Abs{\vw_k}^2\Abs{\sum_{i=1}^{n_r}\tau^i\sigma'(\hat{\vw}_{k}^{\T}\tilde{x}^{i})}\right]\\
            &\le \Exp_{\vtau}\left[\sup_{\norm{\theta}_{\fP}<Q, \Abs{\vu}\le 1}\sum_{k=1}^{m}\Abs{a_k{_{r}\vw_{k,\alpha}}}\Abs{\vw_k}^2\Abs{\sum_{i=1}^{n_r}\tau^i\sigma'(\vu^{\T}\tilde{x}^{i})}\right]\\
            &\le \Exp_{\vtau}\left[\sup_{\norm{\theta}_{\fP}<Q, \Abs{\vu}\le 1}\sum_{k=1}^{m}\Abs{a_k}\Abs{\vw_k}^3\Abs{\sum_{i=1}^{n_r}\tau^i\sigma'(\vu^{\T}\tilde{x}^{i})}\right]\\
            &\le Q\Exp_{\vtau}\left[\sup_{ \Abs{\vu}\le 1}\Abs{\sum_{i=1}^{n_r}\tau^i\sigma'(\vu^{\T}\tilde{x}^{i})}\right],
        \end{aligned}
    \end{equation}
    combining the estimate of $I_2$, we immediately obtain that
    \begin{equation}\label{eq::EstimateofTildeI_1}
        \tilde{I}_1\le Q\sqrt{n_r}.
    \end{equation}
    For the term $\tilde{I}_2$, note that by Definition~\ref{def::spcialdomain}, for all $r\in\{1,\ldots,l\}$, we have $\gamma_{r}({x'_{r}}^{i})\in C^1(\overline{B_{r}^{\kappa}})$, say, bounded by a constant $\widetilde{M}$, $\max_{x'_r\in\overline{B(0,\alpha)},r,\alpha}\{\gamma_r(x'_r),D_{x_{r_{\alpha}}}\gamma_r(x'_r)\}<\widetilde{M}$, thus similarly as the analysis on $\tilde{I}_1$, we have
    \begin{equation}
        \begin{aligned}
            \tilde{I}_2&\le Q\Exp_{\vtau}\left[\sup_{\abs{\vu}\le 1}\Abs{\sum_{i=1}^{n_r}\tau^{i}D_{x_{r_{\alpha}}}\gamma_r({x'_{r}}^{i})\sigma'(\vu^{\T}\tilde{x}^i)}\right],
        \end{aligned}
    \end{equation}
    still combining the estimate of $I_2$, we see
    \begin{equation}\label{eq::EstimateofTildeI_2}
        \tilde{I}_2\le \widetilde{M}Q\sqrt{n_r},
    \end{equation}
    thus~\eqref{eq::EstimateofTildeI_1} and~\eqref{eq::EstimateofTildeI_2} together leads to
    \begin{equation}\label{eq::RademacherD_alpha_rfG_Q}
        \operatorname{Rad}(D_{x_{r_{\alpha}}} {_{r}\fG_{Q}})\le \frac{(\widetilde{M}+1)Q}{\sqrt{n_r}}.
    \end{equation}
    Equation~\eqref{eq::RademacherfF_Q},~\eqref{eq::Rademacer_rfG_Q} and~\eqref{eq::RademacherD_alpha_rfG_Q} together complete the proof.
\end{proof}

\begin{proof}[Proof of Theorem~\ref{thm::ApproximationError}]
    Let $\rho$ be the best representation, i.e. $\norm{(f,g)}_{\fB}=\left(\Exp_{(a,\vw)\sim\rho}\abs{a}^2\abs{\vw}^6\right)^{1/2}$. Set $\theta=\{\frac{1}{m}a_k,\vw_k\}_{k=1}^{m}$ being the parameter with $\{a_k,\vw_k\}_{k=1}^{m}$ independently sampled from $\rho$. By defining
    \begin{equation}\label{eq::Def4RiskFunction}    
        \begin{aligned}
            \fR^{int}_{\Omega}(\theta)&:=\frac{1}{\Abs{\Omega}}\norm{f(x,\theta)-f(x)}_{L^2(\Omega)}^{2},\\
            {_{r}\fR^{bdry}_{B_{r}^{\kappa}}}(\theta)&:=\frac{1}{\Abs{B_{r}^{\kappa}}}\norm{g(x'_r,\tilde{\theta})-{_{r}g(x'_r)}}_{H_{r}^{1}(B_{r}^{\kappa})}^2,
        \end{aligned}
    \end{equation}
    we see
    \begin{equation}\label{eq::RiskFunctionDecomposition}
        \Exp_{\theta}(\fR_{D})=\Exp_{\theta}\left(\fR^{int}_{\Omega}(\theta)\right)+\sum_{r=1}^{l}\Exp_{\theta}\left({_{r}\fR^{bdry}_{B_{r}^{\kappa}}}(\theta)\right),
    \end{equation}
    thus
    \begin{equation}\label{eq::ApproximationError4Domaim}
        \begin{aligned}
            \Abs{\Omega}\Exp_{\theta}\left(\fR^{int}_{\Omega}(\theta)\right)&=\int_{\Omega}\Exp_{\theta}\Abs{f(x,\theta)-f(x)}^2\diff{x}\\
            &=\int_{\Omega}\operatorname{Var}_{(a_k,\vw_k)i.i.d\sim \rho}\left(\frac{1}{m}\sum_{k=1}^{m}a_k\left[\vw_k^{\T}A(x)\vw_k\sigma''(\vw_k^{\T}x)+\hat{\vb}^{\T}(x)\vw_k\sigma'(\vw_K^{\T}x)+c(x)\sigma(\vw_k^{\T}x)\right]\right)\diff{x}\\
            &=\frac{1}{m}\int_{\Omega}\operatorname{Var}_{(a,\vw)\sim \rho}\left(a\left[\vw^{\T}A(x)\vw\sigma''(\vw^{\T}x)+\hat{\vb}^{\T}(x)\vw\sigma'(\vw^{\T}x)+c(x)\sigma(\vw^{\T}x)\right]\right)\diff{x}\\
            &\le \frac{1}{m}\int_{\Omega}\Exp_{(a,\vw)\sim \rho}\left(a^2\left[\vw^{\T}A(x)\vw\sigma''(\vw^{\T}x)+\hat{\vb}^{\T}(x)\vw\sigma'(\vw^{\T}x)+c(x)\sigma(\vw^{\T}x)\right]^2\right)\diff{x}\\
            &\le \frac{1}{m}\int_{\Omega}\Exp_{(a,\vw)\sim \rho}\left(\Abs{a}^2\left(M\Abs{\vw}^3+\frac{1}{2}M\Abs{\vw}^3+\frac{1}{6}M\Abs{\vw}^3\right)^2\right)\diff{x}\\
            &\le \frac{4M\Abs{\Omega}}{m}\Exp_{(a,\vw)\sim\rho}\Abs{a}^2\Abs{\vw}^6\\
            &=\frac{4M\Abs{\Omega}}{m}\norm{(f,g)}_{\fB}^2.
        \end{aligned}
    \end{equation}
    As for $\Exp_{\theta}\left({_{r}\fR^{bdry}_{B_{r}^{\kappa}}}(\theta)\right)$, similarly we have
    \begin{equation}\label{eq::ApproximationError4Boundary}
        \begin{aligned}
            \Abs{B_{r}^{\kappa}}\Exp_{\theta}\left({_{r}\fR^{bdry}_{B_{r}^{\kappa}}}(\theta)\right)&=\int_{B_{r}^{\kappa}}\Exp_{\theta}\left(\Abs{g(x'_r,\tilde{\theta})-{_{r}g(x'_r)}}^2+\sum_{\alpha=1}^{d-1}\Abs{D_{x_{r_{\alpha}}}g(x'_r,\theta)-D_{x_{r_{\alpha}}}{_{r}g(x'_r)}}^2\right)\diff{x'_r}\\
            &\le \frac{1}{m}\int_{\Omega}\Exp_{(a,\vw)\sim \rho}\left(a^2\sigma^2( {_{r}\vw^{\T}}(x_r-\fB_{r}))\right) \diff{x'_r}\ \text{with ${_{r}\vw^{\T}}=\fO^{-1}\vw^{\T}$, $x_{r_{d}}=\gamma_r(x'_r)$}\\
            &~~~+\frac{1}{m}\sum_{\alpha=1}^{d-1}\int_{\Omega}\Exp_{(a,\vw)\sim \rho}\left(a^2\left(\left(_{r}\vw_{\alpha}+{_{r}\vw_{d}}D_{x_{r_{\alpha}}}\gamma_r(x_r')\right)\sigma'( {_{r}\vw^{\T}}(x_r-\fB_{r}))\right)^2\right) \diff{x'_r}\\
            &\le \frac{\Abs{B_{r}^{\kappa}}}{6m}\Exp_{(a,\vw)\sim\rho}\Abs{a}^2\Abs{\vw}^6+\sum_{\alpha=1}^{d-1}\frac{ (\widetilde{M}+1)\Abs{B_{r}^{\kappa}}}{2m}\Exp_{(a,\vw)\sim\rho}\Abs{a}^2\Abs{\vw}^6\\
            &\le \frac{d(\widetilde{M}+1)\Abs{B_{r}^{\kappa}}}{m}\norm{(f,g)}_{\fB}^2.
        \end{aligned}
    \end{equation}
    Thus~\eqref{eq::RiskFunctionDecomposition}, ~\eqref{eq::ApproximationError4Domaim} and~\eqref{eq::ApproximationError4Boundary} lead to
    \begin{equation}\label{eq::Expectaion4Risk}
        \Exp_{\theta}\left(\fR_{D}(\theta)\right)\le \frac{4M+l d(\widetilde{M}+1)}{m}\norm{(f,g)}_{\fB}^2.
    \end{equation}
    Moreover, we have
    \begin{equation}\label{eq::Expectaion4BarrpnTheta}
        \Exp_{\theta}\left(\norm{\theta}_{\fP}\right)=\frac{1}{m}\Exp_{(a_k,\vw_k)i.i.d\sim\rho}\sum_{k=1}^{m}\abs{a_k}\Abs{\vw_k}^3=\Exp_{(a,\vw)\sim\rho}\abs{a}\Abs{\vw}^3\le \norm{(f,g)}_{\fB}.
    \end{equation}
    By defining two events: 
    \begin{equation}
        \begin{aligned}
            H_1&:=\left\{\theta\sim\rho\ \big|\ \fR_{D}(\theta)\le \frac{12M+3l d(\widetilde{M}+1)}{m}\norm{(f,g)}_{\fB}^2\right\},\\
            H_2&:=\left\{ \theta\sim\rho\ \big|\ \norm{\theta}_{\fP}\le 2\norm{(f,g)}_{\fB}\right\},
        \end{aligned}
    \end{equation}
    then by the Markov's inequality and~\eqref{eq::Expectaion4Risk},
    \begin{equation}\label{eq::ProbEventR_D}
        \begin{aligned}
             \Prob(H_1)&=1-\Prob\left(\fR_{D}(\theta)>\frac{12M+3l d(\widetilde{M}+1)}{m}\norm{(f,g)}_{\fB}^2\right)\\
             &\ge 1-\frac{m\Exp_{\theta}\left(\fR_{D}(\theta)\right)}{(12M+3l d(\widetilde{M}+1))\norm{(f,g)}_{\fB}^2}\\
             &\ge \frac{2}{3},
        \end{aligned}
    \end{equation}
    moreover, due to~\eqref{eq::Expectaion4BarrpnTheta}
    \begin{equation}\label{eq::ProbEventBarronTheta}
        \Prob(H_2)=1-\Prob\left(\norm{\theta}_{\fP}>2\norm{(f,g)}_{\fB}\right)\ge 1-\frac{\Exp_{\theta}\left(\norm{\theta}_{\fP}\right)}{2\norm{(f,g)}_{\fB}}\ge \frac{1}{2}.
    \end{equation}
    Thus~\eqref{eq::ProbEventR_D} and~\eqref{eq::ProbEventBarronTheta} together imply that
    \begin{equation}
        \Prob(H_1\cap H_2)\ge \Prob(H_1)+\Prob(H_2)-1\ge \frac{1}{6}>0,
    \end{equation}
    which means the existence of $\tilde{\theta}$ such that
    \begin{equation*}
        \begin{aligned}
            \fR_{D}(\tilde{\theta})&\le \frac{12M+3l d(\widetilde{M}+1)}{m}\norm{(f,g)}_{\fB}^2,\\
            \norm{\tilde{\theta}}_{\fP}&\le 2\norm{(f,g)}_{\fB}.
        \end{aligned}
    \end{equation*}
\end{proof}

To obtain the following generalization result, we have to define the following class of functions for $r\in \{1,\ldots,l\}$ and $\alpha\in \{1,\ldots,d-1\}$:
\begin{equation}\label{eq::Def4ResidualFunctionSpace}
    \begin{aligned}
        \fH_{Q}&:=\left\{ \left(f(x)-f(x,\theta)\right)^2\ \big|\ \norm{\theta}_{\fP}<Q\right\},\\
        {_{r}\fJ}_{Q}&:=\left\{ \left({_{r}g(x'_r)}-g(x'_r,\theta)\right)^2\ \big|\ \norm{\theta}_{\fP}<Q\right\},\\
        D_{x_{r_{\alpha}}}{_{r}\fJ}_{Q}&:=\left\{ \left(D_{x_{r_{\alpha}}}{_{r}g(x'_r)}-D_{x_{r_{\alpha}}}g(x'_r,\theta)\right)^2\ \big|\ \norm{\theta}_{\fP}<Q\right\},
    \end{aligned}
\end{equation}
and define $\fH=\cup_{Q=1}^{\infty} \fH_{Q}$, ${_{r}\fJ}=\cup_{Q=1}^{\infty} {_{r}\fJ}_{Q}$ and $D_{x_{r_{\alpha}}}{_{r}\fJ}=\cup_{Q=1}^{\infty}D_{x_{r_{\alpha}}}{_{r}\fJ}_{Q}$, respectively.

\begin{proof}[Proof of Theorem~\ref{thm::APosterierGeneralizationBoundElliptic}]
    For $\norm{\theta}_{\fP}<Q$ and $x\in \Omega$, we have that $\fL u(x,\theta)=f(x,\theta)$ for some $f(x,\theta)\in \fF_{Q}$; for $\norm{\theta}_{\fP}<Q$ and $x\in \partial\Omega$, we have that $u(x,\theta)=g(x'_r,\theta)$ and $D_{x_{r_{\alpha}}} u(x,\theta)= D_{x_{r_{\alpha}}} g(x'_r,\theta)$ for all $\alpha\in\{1,\ldots,d-1\}$ and some $r\in \{1,\ldots, m\}$ and $g(x'_r,\theta)\in {_{r}\fG_{Q}}$.

    Note that for $f(x,\theta)\in \fF_{Q}$
    \begin{equation}\label{eq::UpperLimit4NNf}   
        \begin{aligned}
            \sup_{x\in \Omega}\Abs{f(x,\theta)}&=\sup_{x\in\Omega}\Abs{\sum_{k=1}^{m}a_k\left[\vw_k^{\T}A(x)\vw_k\sigma''(\vw_k^{\T}x)+\hat{\vb}^{\T}(x)\vw_k\sigma'(\vw_K^{\T}x)+c(x)\sigma(\vw_k^{\T}x)\right]}\\
            &\le \left(\frac{1}{6}+\frac{1}{2}+1\right)M\sum_{k=1}^{m}\Abs{a_k}\Abs{\vw_k}^3\\
            &\le 2M\norm{\theta}_{\fP}<2MQ.
        \end{aligned} 
    \end{equation}
    Similarly, for $g(x'_r,\theta)\in {_{r}\fG_{Q}}$ and $D_{x_{r_{\alpha}}} g(x'_r,\theta)\in D_{x_{r_{\alpha}}} {_{r}\fG_{Q}}$ combing the definition~\eqref{eq::Def4ClassicFunctionSpace}, we can obtain
    \begin{equation}\label{eq::UpperLimit4NNg&Derivative}
        \begin{aligned}
            \sup_{x'_r\in B_{r}^{\kappa}}g(x'_r,\theta)&\le \norm{\theta}_{\fP}<\frac{1}{6}Q,\\
            \sup_{x'_r\in B_{r}^{\kappa}}D_{x_{r_{\alpha}}} g(x'_r,\theta)&\le (\widetilde{M}+1)\norm{\theta}_{\fP}<\frac{1}{2}(\widetilde{M}+1)Q.
        \end{aligned}
    \end{equation}
    Since we assume without loss of generality that $\abs{f(x)},\Abs{{_{r}g(x'_r)}},\Abs{D_{x_{r_{\alpha}}}{_{r}g(x'_r)}} \le 1$ for all $x\in\Omega, x'_r\in B_{r}^{\kappa}, r\in \{1,\ldots,l\}, \alpha\in\{1,\ldots,d-1\}$, combining~\eqref{eq::UpperLimit4NNf} and~\eqref{eq::UpperLimit4NNg&Derivative}, we thus have for all functions in $\fH_{Q}, {_{r}\fJ}_{Q}, D_{x_{r_{\alpha}}}{_{r}\fJ}_{Q}$ respectively,
    \begin{equation}\label{eq::UpperLimit4ResidualFunctionSpace}
        \begin{aligned}
            \sup_{x\in \Omega}\left(f(x)-f(x,\theta)\right)^2&\le (1+2MQ)^2\le 9M^2Q^2,\\
            \sup_{x'_r\in B_{r}^{\kappa}}\left({_{r}g(x'_r)}-g(x'_r,\theta)\right)^2&\le (1+\frac{1}{6}Q)^2\le 4Q^2,\\
            \sup_{x'_r\in B_{r}^{\kappa}}\left(D_{x_{r_{\alpha}}}{_{r}g(x'_r)}-D_{x_{r_{\alpha}}}g(x'_r,\theta)\right)^2&\le (1+\frac{1}{2}(\widetilde{M}+1)Q)^2\le 4(\widetilde{M}+1)^2Q^2,
        \end{aligned}
    \end{equation}
    where we assume without loss of generality that $MQ>1$ and $(\widetilde{M}+1)Q>1$.

    By Lemma~\ref{lem::Contraction} , combining ~\eqref{eq::UpperLimit4ResidualFunctionSpace}, and~\eqref{eq::RademacherfF_Q},~\eqref{eq::Rademacer_rfG_Q},~\eqref{eq::RademacherD_alpha_rfG_Q} together, for any random sampled set $S$
    \begin{equation}\label{eq::RademacherComplexity4ResidualFunctionSpace}
        \begin{aligned}
            \operatorname{Rad}_{\{x_i\}_{i=1}^{n_{\Omega}}}(\fH_{Q})&\le 6MQ\operatorname{Rad}(\fF_{Q})\le\frac{24M^2Q^2d^2}{\sqrt{n_{\Omega}}},\\
            \operatorname{Rad}_{\{x_{r}^{i}\}_{i=1}^{n_r}}({_{r}\fJ}_{Q})&\le 4Q\operatorname{Rad}({_{r}\fG_{Q}})\le\frac{4Q^2}{3\sqrt{n_r}},\\
            \operatorname{Rad}{\{x_{r}^{i}\}_{i=1}^{n_r}}(D_{x_{r_{\alpha}}}{_{r}\fJ}_{Q})&\le 4(\widetilde{M}+1)Q\operatorname{Rad}(D_{x_{r_{\alpha}}} {_{r}\fG_{Q}})\le \frac{4(\widetilde{M}+1)^2Q^2}{\sqrt{n_r}}.
        \end{aligned}
    \end{equation}
    Note that
    \begin{equation}\label{eq::GeneralizationErrorDecompositionDomain&Boundary}
        \begin{aligned}
            \sup_{\norm{\theta}_{\fP}<Q}\Abs{\fR_{D}(\theta)-\fR_{S}(\theta)}&\le \sup_{\norm{\theta}_{\fP}<Q}\Abs{\fR^{int}_{\Omega}(\theta)-\fR^{int}_{\{x^i\}_{i=1}^{n_{\Omega}}}(\theta)}+\sum_{r=1}^{l}\sup_{\norm{\theta}_{\fP}<Q}\Abs{{_{r}\fR^{bdry}_{B_{r}^{\kappa}}}-{_{r}\fR^{bdry}_{\{x_{r}^{i}\}_{i=1}^{n_r}}}}\\
            &=:\operatorname{Err}_{1}+\sum_{r=1}^{l}\operatorname{Err}_{2,r}.
        \end{aligned}
    \end{equation}
    then Theorem~\ref{thm::Rade&Generalization}, Definition~\eqref{eq::Def4RiskFunction}, upper limit estimate~\eqref{eq::UpperLimit4ResidualFunctionSpace}, and Rademacher estimate~\eqref{eq::RademacherComplexity4ResidualFunctionSpace}, we see that for any $\delta>0$, with probability at least $1-\frac{3\delta}{\pi^2Q^2}$ over the choice of $\{x^i\}_{i=1}^{n_{\Omega}}$ and $1-\frac{3\delta}{\pi^2Q^2l}$ over the choice of $\{x_{r}^{i}\}_{i=1}^{n_r}$ for each $r\in \{1,\ldots,l\}$ 
    \begin{equation}\label{eq::UpperLimitGeneralizationByPart}
        \begin{aligned}
            \operatorname{Err}_{1}&\le \frac{48M^2Q^2d^2}{\sqrt{n_{\Omega}}}+36M^2Q^2\sqrt{\frac{2\ln\left(\frac{4\pi^2Q^2}{3\delta}\right)}{n_{\Omega}}},\\
            \operatorname{Err}_{2,r}&\le \frac{8(\widetilde{M}+1)^2Q^2d}{\sqrt{n_r}}+16(\widetilde{M}+1)^2Q^2d\sqrt{\frac{2\ln\left(\frac{4\pi^2Q^2l}{3\delta}\right)}{n_r}},
        \end{aligned}
    \end{equation}
    hence for any $\theta=\{(a_k,\vw_k)\}_{k=1}^{m}\in \sR^{m(d+1)}$, choose integer $Q$ such that $\norm{\theta}_{\fP}\le Q\le \norm{\theta}_{\fP}+1$, then
    \begin{equation}\label{eq::UpperLimitGeneralization}
        \begin{aligned}
            \Abs{\fR_{D}(\theta)-\fR_{S}(\theta)}&\le \left(48M^2d^2\frac{(\norm{\theta}_{\fP}+1)^2}{\sqrt{n_{\Omega}}}+\sum_{r=1}^{l}8(\widetilde{M}+1)^2d\frac{(\norm{\theta}_{\fP}+1)^2}{\sqrt{n_r}}\right)\\
            &~~~+\bigg(36M^2\left(2\ln(\pi(\norm{\theta}_{\fP}+1))+2\sqrt{\ln\left(2/(3\delta)\right)}\right)\frac{(\norm{\theta}_{\fP}+1)^2}{\sqrt{n_{\Omega}}}\\
            &~~~+\sum_{r=1}^{l}16(\widetilde{M}+1)^2d\left(2\ln(\pi(\norm{\theta}_{\fP}+1))+2\sqrt{\ln\left((2l)/(3\delta)\right)}\right)\frac{(\norm{\theta}_{\fP}+1)^2}{\sqrt{n_r}}\bigg)\\
            &\le \bigg(72M^2\left(d^2+\ln(\pi(\norm{\theta}_{\fP}+1))+\sqrt{\ln\left(1/\delta)\right)}\right)\frac{(\norm{\theta}_{\fP}+1)^2}{\sqrt{n_{\Omega}}}\\
            &~~~+\sum_{r=1}^{l}32(\widetilde{M}+1)^2d\left(1+\ln(\pi(\norm{\theta}_{\fP}+1))+\sqrt{\ln\left(l/\delta\right)}\right)\frac{(\norm{\theta}_{\fP}+1)^2}{\sqrt{n_r}}\bigg)\\
            &\le 2d\left(72M^2+32(\widetilde{M}+1)^2l\right)\left(d+\ln(\pi(\norm{\theta}_{\fP}+1))+\sqrt{\ln\left(l/\delta\right)}\right)\\
            &~~~\times\left(\frac{\norm{\theta}_{\fP}^{2}+1}{\sqrt{n_{\Omega}}}+\sum_{r=1}^{l}\frac{\norm{\theta}_{\fP}^{2}+1}{\sqrt{n_r}}\right),
        \end{aligned}
    \end{equation}
    moreover, note that $\sum_{Q=1}^{\infty}(\frac{3\delta}{\pi^2Q^2}+\sum_{r=1}^{l}\frac{3\delta}{\pi^2Q^2l})=\delta$, we see that with probability at least $1-\delta$ over the choice of sampled set $S$, equation~\eqref{eq::UpperLimitGeneralization} holds.
\end{proof}

\begin{proof}[Proof of Theorem~\ref{thm::APrioriGeneralizationBoundElliptic}]
    We have the following decomposition of $\fR_{D}(\theta_{S})$
    \begin{equation}\label{eq::Decomposition4RiskFunction}
        \begin{aligned}
            \fR_{D}(\theta_{S})&=\fR_{D}(\tilde{\theta})+\left(\fR_{D}(\theta_{S})-\fR_{S}(\theta_{S})\right)+\left(\fR_{S}(\theta_{S})-\fR_{S}(\tilde{\theta})\right)+\left(\fR_{S}(\tilde{\theta})-\fR_{D}(\tilde{\theta})\right)\\
            &=:\fR_{D}(\tilde{\theta})+G_{1}+G_{2}+G_{3}.
        \end{aligned}
    \end{equation}
    Note that, by definition, $G_{2}<0$. And according to Theorem~\ref{thm::ApproximationError}, we have
    \begin{equation}\label{eq::UpperLimit4G_0}
        \fR_{D}(\tilde{\theta})\le \frac{12M+3l d(\widetilde{M}+1)}{m}\norm{(f,g)}_{\fB}^2.
    \end{equation}
    Moreover, by Theorem~\ref{thm::APosterierGeneralizationBoundElliptic}, with probability at least $1-\frac{\delta}{2}$ over the choice of random sampled set $S$
    \begin{equation}\label{eq::UpperLimit4G1}
        \begin{aligned}
            \Abs{G_{1}}&=\fR_{D}(\theta_{S})-\fR_{S}(\theta_{S})\\
            &\le 2d\left(72M^2+32(\widetilde{M}+1)^2l\right)\left(d+\ln(\pi(\norm{\theta_{S}}_{\fP}+1))+\sqrt{\ln\left(2l/\delta\right)}\right)\left(\frac{\norm{\theta_{S}}_{\fP}^{2}+1}{\sqrt{n_{\Omega}}}+\sum_{r=1}^{l}\frac{\norm{\theta_{S}}_{\fP}^{2}+1}{\sqrt{n_r}}\right)\\
            &\le 2d\pi\left(72M^2+32(\widetilde{M}+1)^2l\right)(d+1+\sqrt{\ln(2l/\delta)}) (B+1)^3\left(\frac{1}{\sqrt{n_{\Omega}}}+\sum_{r=1}^{l}\frac{1}{\sqrt{n_r}}\right).
        \end{aligned}
    \end{equation}
    Similarly,  by Theorem~\ref{thm::APosterierGeneralizationBoundElliptic}, and Theorem~\ref{thm::ApproximationError} that $\norm{\tilde{\theta}}_{\fP}<2\norm{(f,g)}_{\fB}$ with probability at least $1-\frac{\delta}{2}$ over the choice of random sampled set $S$, we have
    \begin{equation}\label{eq::UpperLimit4G3}
        \begin{aligned}
            \Abs{G_{3}}&\le 2d\left(72M^2+32(\widetilde{M}+1)^2l\right)\left(d+\ln(\pi(\norm{\tilde{\theta}}_{\fP}+1))+\sqrt{\ln\left(2l/\delta\right)}\right)\left(\frac{\norm{\tilde{\theta}}_{\fP}^{2}+1}{\sqrt{n_{\Omega}}}+\sum_{r=1}^{l}\frac{\norm{\tilde{\theta}}_{\fP}^{2}+1}{\sqrt{n_r}}\right)\\
            &\le 2d\pi\left(72M^2+32(\widetilde{M}+1)^2l\right)(d+1+\sqrt{\ln(2l/\delta)}) (2\norm{(f,g)}_{\fB}+1)^3\left(\frac{1}{\sqrt{n_{\Omega}}}+\sum_{r=1}^{l}\frac{1}{\sqrt{n_r}}\right).
        \end{aligned}
    \end{equation}
    Thus combining~\eqref{eq::UpperLimit4G_0},~\eqref{eq::UpperLimit4G1} and~\eqref{eq::UpperLimit4G3} together lead to with probability at least $1-\delta$ over the choice of random sampled set $S$,
    \begin{equation}
        \begin{aligned}
            \fR_{D}(\theta_{S})&\le  \frac{12M+3ld(\widetilde{M}+1)}{m}\norm{(f,g)}_{\fB}^2\\
            &~~~+2d\pi\left(72M^2+32(\widetilde{M}+1)^2l\right)(d+1+\sqrt{\ln(2l/\delta)}) (2\norm{(f,g)}_{\fB}+2+B)^3\left(\frac{1}{\sqrt{n_{\Omega}}}+\sum_{r=1}^{l}\frac{1}{\sqrt{n_r}}\right).
        \end{aligned}
    \end{equation}
    Hence we complete the proof.
\end{proof}

    \bibliographystyle{plain}
    \bibliography{reference_CiCP}
\end{document}